\documentclass[11pt]{amsart}

\usepackage{amsfonts,amscd,mathabx}
\usepackage{calligra,mathrsfs}
\usepackage[all]{xy}
\usepackage{float, comment}
\usepackage{mathtools}
\usepackage{amsmath}
\usepackage{amsthm}
\usepackage{amssymb}
\usepackage{amsbsy}
\usepackage{amstext}
\usepackage{amsopn}
\usepackage{mathrsfs} 
\usepackage[mathscr]{eucal}
\usepackage{enumerate}
\usepackage{enumitem}
\usepackage{xcolor}
\usepackage{graphicx} 
\usepackage{scalerel}
\usepackage{microtype} 
\usepackage[margin=1in,marginparwidth=0.8in, marginparsep=0.1in]{geometry}
\usepackage[bookmarks=true, bookmarksopen=true, bookmarksdepth=3,bookmarksopenlevel=2, colorlinks=true, linkcolor=blue, citecolor=blue, filecolor=blue, menucolor=blue, urlcolor=blue]{hyperref}
\usepackage{tikz}
\usepackage{bbm}
\usepackage[all]{xy}
\usepackage{xspace}
\usetikzlibrary{3d,arrows,calc,positioning,decorations.pathreplacing,matrix} 
\usepackage{accents}
\usepackage{extarrows}

\tikzset{dbl/.style={double,
		double distance=6.0,
		-implies,
		shorten >=10pt,
		shorten <=10pt}}

\usepackage{float, comment}
\usepackage{mathtools}
\usepackage{amsmath}
\usepackage{amsthm}
\usepackage{amssymb}
\usepackage{amsbsy}
\usepackage{amstext}
\usepackage{amsopn}
\usepackage{mathrsfs} 
\usepackage{enumerate}
\usepackage{scalerel}
\usepackage{microtype} 
\usepackage[bookmarks=true, bookmarksopen=true, bookmarksdepth=3,bookmarksopenlevel=2, colorlinks=true, linkcolor=blue, citecolor=blue, filecolor=blue, menucolor=blue, urlcolor=blue]{hyperref}
\usepackage{newtxtext} 
\usepackage{mathrsfs}
\usepackage{bbm}
\usepackage{tikz}
\usepackage{tikz-cd}
\usetikzlibrary{matrix,arrows.meta}

\DeclareMathOperator{\cHom}{\mathscr{H}\text{\kern -3pt {\calligra\large om}}\,}

\usepackage{amsthm}
\usepackage{hyperref}
\usepackage{cleveref}

\theoremstyle{plain}
\newtheorem{Theorem}{Theorem}[section]
\newtheorem{Lemma}[Theorem]{Lemma}
\newtheorem{Proposition}[Theorem]{Proposition}
\newtheorem{Corollary}[Theorem]{Corollary}

\numberwithin{equation}{section} 

\theoremstyle{definition}
\newtheorem{Definition}[Theorem]{Definition}
\newtheorem{Example}[Theorem]{Example}
\newtheorem{Remark}[Theorem]{Remark}

\crefname{Theorem}{theorem}{theorems}
\Crefname{Theorem}{Theorem}{Theorems}
\crefname{Lemma}{lemma}{lemmas}
\Crefname{Lemma}{Lemma}{Lemmas}
\crefname{Proposition}{proposition}{propositions}
\Crefname{Proposition}{Proposition}{Propositions}
\crefname{Corollary}{corollary}{corollaries}
\Crefname{Corollary}{Corollary}{Corollaries}
\crefname{Definition}{definition}{definitions}
\Crefname{Definition}{Definition}{Definitions}
\crefname{Example}{example}{examples}
\Crefname{Example}{Example}{Examples}
\crefname{Remark}{remark}{remarks}
\Crefname{Remark}{Remark}{Remarks}

\numberwithin{figure}{section}

\newcommand{\Coh}{\mathrm{Coh}}

\def\ren{{\rm{ren}}}
\def\im{{\rm{im}}}
\def\barR{\bar{R}}
\def\barF{\overline{{\sf{F}}}}
\def\bscrM{\overline{\mathscr{M}}}
\newcommand{\dcong}{\mathrel{{\dot{\cong}}}}
\newcommand{\ddcong}{\mathrel{{\ddot{\cong}}}}

\def\C{\mathrm{C}}
\def\D{\mathrm{D}}

\def\F{\mathrm{F}}

\def\L{\mathrm{L}}

\def\N{\mathrm{N}}
\def\O{\mathrm{O}}
\def\P{\mathrm{P}}
\def\Q{\mathrm{Q}}
\def\R{\mathrm{R}}

\def\Z{\mathrm{Z}}
\def\u{\mathrm{u}}

\def\bbC{\mathbb{C}}
\def\bbD{\mathbb{D}}

\def\bbN{\mathbb{N}}

\def\bbZ{\mathbb{Z}}

\def\scrL{\mathscr{L}}
\def\scrM{\mathscr{M}}

\def\scrS{\mathscr{S}}

\def\calO{\mathfrak{O}}

\def\frakS{\mathfrak{S}}

\def\calC{\mathcal{C}}
\def\calD{\mathcal{D}}
\def\calE{\mathcal{E}}

\def\calI{\mathcal{I}}

\def\calK{\mathcal{K}}

\def\calM{\mathcal{M}}
\def\calN{\mathcal{N}}
\def\calO{\mathcal{O}}
\def\calP{\mathcal{P}}
\def\calQ{\mathcal{Q}}
\def\calR{\mathcal{R}}
\def\calS{\mathcal{S}}
\def\calT{\mathcal{T}}

\def\calV{\mathcal{V}}

\def\frakg{\mathfrak{g}}

\def\frakn{\mathfrak{n}}

\def\bfr{\mathbf{r}}

\def\bfR{\mathbf{R}}

\def\uum{{{\underline{m}}}}
\def\uun{{{\underline{n}}}}

\def\La{\Lambda}
\def\t{t}
\def\hz{\widehat{z}}

\def\cusp{{\operatorname{cusp}\nolimits}}
\def\dim{\operatorname{dim}\nolimits}

\def\gmod{\operatorname{gmod}\nolimits}
\def\gMod{\operatorname{gMod}\nolimits}

\def\Mod{\operatorname{Mod}\nolimits}

\def\Coh{\operatorname{Coh}\nolimits}
\def\IndCoh{\operatorname{IndCoh}\nolimits}

\def\Im{\operatorname{Im}\nolimits}
\def\Ker{\operatorname{Ker}\nolimits}

\def\Res{\operatorname{Res}\nolimits}

\def\id{\operatorname{id}\nolimits}

\def\Spec{\operatorname{Spec}\nolimits}

\def\height{\operatorname{ht}\nolimits}
\def\pt{\operatorname{pt}\nolimits}

\def\af{{\operatorname{af}\nolimits}}

\def\re{{\operatorname{re}\nolimits}}

\def\height{\operatorname{ht}\nolimits}

\def\c1{\operatorname{c_1}\nolimits}

\def\-{{{\text{-}}}}
\DeclareMathOperator{\calHom}{\mathscr{H}\text{\kern -3pt {\calligra\large om}}\,}

\def\tF{\widetilde{\sf{F}}}
\def\tcalC{\widetilde{\mathcal{C}}}

\def\d{\Delta}
\def\n{\nabla}
\def\bd{\overline{\Delta}}
\def\bn{\overline{\nabla}}

\def\hGO{{\widehat{G}_\calO}}
\def\hGK{{\widehat{G}_{\mathcal K}}}

\def\uk{{\underline{k}}}
\def\u0{{\underline{0}}}
\def\ul{{\underline{\ell}}}

\def\b{\mathfrak{b}}
\def\hd{{\rm{hd}}}

\def\cl{{\rm{cl}}}

\def\r{{\bf{r}}}

\def\hd{{\rm{hd}}}

\def\td{\widetilde{d}}
\def\tm{\widetilde{m}}
\def\coker{{\rm coker}}

\def\ttimes{\widetilde{\times}}

\def\C{\mathbb{C}}
\def\D{\mathbb{D}}
\def\N{\mathbb{N}}

\def\Z{\mathbb{Z}}

\def\F{{\sf{F}}}

\newcommand{\Cone}{\mathrm{Cone}}
\def\O{{\mathcal O}}

\def\dim{\mbox{dim}}

\def\Gr{\mathrm{Gr}}

\def\cKP{{\mathcal{KP}}}

\def\la{\langle}
\def\ra{\rangle}

\def\cF{{\mathscr{F}}}
\def\cG{{\mathscr{G}}}
\def\cP{{\mathscr{P}}}

\def\cK{{\mathscr{K}}}
\def\cQ{{\mathscr{Q}}}
\def\cS{{\mathscr{S}}}
\def\cV{{\mathscr{V}}}

\def\bA{{\mathbb{A}}}

\def\hG{{\widehat{G}}}

\DeclareMathOperator{\Hom}{Hom}
\DeclareMathOperator{\Ext}{Ext}

\DeclareMathOperator{\End}{End}

\newcommand{\llb}{[\![}
\newcommand{\rrb}{]\!]}
\newcommand{\llp}{(\!(}	
\newcommand{\rrp}{)\!)}

\begin{document}

\title{Duality functors for Coulomb branches I}

\author[Sabin Cautis]{Sabin Cautis}
\address[Sabin Cautis]{University of British Columbia \\ Vancouver BC, Canada}
\email{cautis@math.ubc.ca}

\author[Eric Vasserot]{Eric Vasserot}
\address[Eric Vasserot]{Universit\'e Paris Cit\'e \\ 75013 Paris, France}
\email{vasserot@imj-prg.fr}

\begin{abstract}
To a quiver of finite type one can associate two abelian categories: the category of finite dimensional modules over the 
corresponding affine quiver Hecke algebra and the Coulomb category of Koszul-perverse coherent sheaves. In this paper 
we define an exact functor from the former to the latter. This functor sends a simple to a simple or zero and is (almost) essentially 
surjective on simples. This implies that simple Koszul-perverse sheaves categorify the dual canonical basis. 
\end{abstract}

\maketitle
\setcounter{tocdepth}{1}
\tableofcontents

\section{Introduction}

Fix a finite type simply laced quiver $Q=(I,E)$ and its affinization $Q_\af$. Associated to $Q$ and a fixed dimension vector $\alpha$ one has the 
K-theoretic Coulomb branch $\calM_{G,N}$ as defined by Braverman, Nakajima and Finkelberg \cite{Nak,BFNa,BFNb}. On the other hand, 
associated to $Q_\af$ one has the quiver Hecke algebra $R = \oplus_\beta R(\beta)$ as introduced by Khovanov-Lauda and Rouquier 
\cite{KL,Rou}. The goal of this paper is to provide a precise, direct relationship between these two at the level of categories. Our main result is 
the following.

\begin{Theorem}[cf. Prop. \ref{prop:Fexact}, Prop. \ref{prop:simple2simple}, Cor. \ref{cor:surjective}] \label{thm:main}
Assuming that $\alpha$ is a strictly dominant weight,
there is a monoidal exact functor between finite length abelian categories
\begin{equation}\label{eq:main}
\F: \calC \to \cKP_{X^{(\bullet)}}.
\end{equation}
This functor sends simple modules to simple objects or zero and is (almost) essentially surjective on simples. 
\end{Theorem}

Let us explain the notation in \eqref{eq:main}. We denote by $\calC=\gmod(R)^\heartsuit$ the abelian, monoidal category of graded, finite dimensional modules 
over the quiver Hecke algebra $R$. On the other hand, recall that $\calM_{G,N} = \Spec(K^{G_\O}(\calR_{G,N}))$ where $G$ is the gauge group, 
$N$ the representation of $Q$ with dimension vector $\alpha$,
$\calR = \calR_{G,N}$ is a certain infinite dimensional moduli space, and $K^{G_\O}(-)$ the equivariant 
K-theory. The space $\calR$ carries a convolution structure which makes $K^{G_\O}(\calR)$ into an algebra.
The Grothendieck group $K^{G_\O}(\calR)$ has an obvious 
categorification $\Coh^{G_\O}(\calR)$, namely the derived category of coherent sheaves on $\calR$. 
In \cite{CW2} a non-standard t-structure was constructed on $\Coh^{\hGO}(\calR)$ where $\hGO$ is an extension of $G_\O$ as in \eqref{hatG}.  
This was called the Koszul-perverse t-structure and its heart $\cKP$ is a finite length, monoidal abelian category. 

The space $\calR$ also carries a factorization structure. This means that for a smooth curve $X$ and $m \ge 1$ one has a space $\calR_{X^m}$ 
fibering over $X^m$. In our case we take $X=\bA^1$. 
This space descends to a space $\calR_{X^{(m)}} \to X^{(m)}$ over the symmetric 
product $X^{(m)} = X^m/\frakS_m$. 
The derived category of equivariant coherent sheaves on $\calR_{X^{(m)}}$ supported over the origin in $X^{(m)}$ has a similarly 
defined Koszul-perverse t-structure whose heart we denote $\cKP_{X^{(m)}}$. Moreover, the factorization structure provides natural 
embeddings $\calR_{X^{(m)}} \to \calR_{X^{(dm)}}$. Pushforward under these is t-exact and allows us to define the colimit category $
\cKP_{X^{(\bullet)}}$ that appears on the right side of \eqref{eq:main}. 

Like $\cKP$ the category $\cKP_{X^{(\bullet)}}$ is abelian, of finite length, and monoidal. The embedding $\calR \to \calR_{X^{(m)}}$ induces an exact functor 
$\cKP \to \cKP_{X^{(\bullet)}}$. This functor is faithful and identifies simples on either sides (but is not full). Thus $\cKP_{X^{(\bullet)}}$ is a larger categorification of the same algebra. 
The relationship between $\cKP$ and $\cKP_{X^{(\bullet)}}$  is very much analogous to that of $\Coh(\pt)$ and $\Coh_0(\bA^1)$ where the latter is the category of sheaves on $\bA^1$ supported over $0 \in \bA^1$.  

Lastly note that similarities between $\cKP$ and quiver Hecke algebras, such as the existence of renormalized $r$-matrices for both, were 
already highlighted in \cite{CW2}. In fact, a closely related version of Theorem \ref{thm:main} appears as \cite[Conj. 1.8]{CW2}. However, we did 
not foresee in \cite{CW2} that the factorization structure plays a fundamental role and consequently that the target category should be $\cKP_{X^{(\bullet)}}$ rather than $\cKP$. 

\subsection{Construction of the functor $\F$}

To define $\F$ we were inspired by the technique from \cite{KKK} which generalizes classical Schur-Weyl duality by defining duality functors 
from quiver Hecke algebras to quantum affine algebras. More precisely, they define
\begin{align}\label{eq:quantumaffine} 
\gMod(R)^\heartsuit \to \Mod(U_q(\frakg)),\quad
M \mapsto V^{\otimes \beta} \otimes_{R(\beta)} M
\end{align}
where $U_q(\frakg)$ is the quantum group associated with a symmetrizable Cartan datum and $V$,
the duality datum, is a particular 
representation of $U_q(\frakg)$ so that $V^{\otimes \beta}$ carries an action of $R(\beta)$. 
This action on $V^{\otimes \beta}$ is a 
biproduct of the theory of renormalized $r$-matrices they developed. 

The theory of renormalized $r$-matrices for $\cKP$ developed in \cite{CW1,CW2} is constructed using the deformations $\calR_{X^m}$. 
More precisely, see \S \ref{sec:Fdefs}, the duality datum consists of a set of objects $\{\scrM_i\}_{i \in I_\af}$ in $\cKP$ which are all simple and 
real and such that, for $i \ne j \in I_\af$, one has 
$$\Lambda(\scrM_i,\scrM_j) = - \la \alpha_i, \alpha_j \ra$$ 
where $\La(-,-)$ is the pairing defined by the renormalized $r$-matrices. 
These objects have natural affinizations $(\scrM_i)_z \in \cKP_X$ given by trivial deformations. The direct sum of products  
\begin{equation}\label{eq:localM}
\bigoplus_{\sum \nu_i = \nu} (\scrM_{\nu_1})_z \circ (\scrM_{\nu_2})_z \circ \dots \circ (\scrM_{\nu_m})_z
\end{equation}
lives on $\calR_{X^{(m)}}$ and turns out to carry an action of the quiver Hecke algebra $R(\nu)$. Subsequently one obtains a functor $\F$ by 
using \eqref{eq:localM} in place of $V^{\otimes \beta}$ and tensoring as in \eqref{eq:quantumaffine}.

This action of the quiver Hecke algebra exists quite generally for any duality datum as above. Roughly, the generators $\tau_i$ of the 
quiver Hecke algebra act by the renormalized $r$-matrices while the $x_i$ act as multiplication by the coordinates $z_i$ of 
$X^m = \Spec \C[z_1,\dots,z_m]$. 

There are two major technical challenges in the construction above. The first is determining the duality datum $\{\scrM_i\}_{i \in I_\af}$. If $i \in I$ 
there is a natural choice for $\scrM_i = \scrM(\alpha_i)$ but for the affine node $0 \in I_\af \setminus I$ it is more difficult to determine 
$\scrM_0 = \scrM(-\theta+\delta)$ where $\theta$ is the highest root. This is related to the fact that the Coulomb algebra 
$K^{G_{\O}}(\calR_{G,N})$ appears naturally in its loop presentation, i.e., it is a quotient of a loop affine Lie algebra, while the quiver Hecke 
algebra appears in its Kac-Moody presentation. Thus one has to translate between the two presentations in order to define 
$\scrM_0 = \scrM(-\theta+\delta)$. This is similar to writing the affine generator $T_0$ of an affine braid group in terms of standard generators 
in the loop presentation. To do this one must understand more broadly all the root objects $\scrM(-\beta+\delta)$. 
It is in this process that we use the assumption that $\alpha$ is strictly dominant.  

The second difficulty is showing that $\F$ is bounded (from which it follows fairly quickly that it is t-exact).  Finite type quiver Hecke algebras have finite global dimension \cite[Thm. 4.7]{Mc1} so the functor is always bounded. This is no longer the case for affine types. In fact, a random choice of duality datum of affine type will usually induce a functor $\F$ which is unbounded.

\subsection{Dual canonical basis and related work}\label{sec:dual-canonical}
 
The simples in $\gmod(R)^\heartsuit$ categorify the dual canonical basis of $A_q(\widehat{\frakn})$. In K-theory we know, by Proposition \ref{prop:simple2simple}, that a dual canonical basis element is sent to either zero or the class of a simple Koszul-perverse sheaf. Moreover, we know by Corollary \ref{cor:surjective} and \cite{VV1} that, up to localization, the class of every simple Koszul-perverse sheaf is the image of a dual canonical basis element. In this context we get the following.

\begin{Corollary}\label{cor:main}
For finite type simply laced quivers, simple Koszul-perverse coherent sheaves categorify the dual canonical basis. 
\end{Corollary}

When the quiver $Q$ is type $A_1$ Corollary \ref{cor:main}  recovers the main result of \cite{FF}. The proof in \cite{FF}  uses a functor which relates $\cKP$ with sheaves on nilpotent cones. Though our proof is more direct it would be very interesting to generalize their functor to more general quivers. 

While the existence of a functor $\F$ as in \eqref{eq:main} for quivers not of finite type is unclear, Corollary \ref{cor:main} suggests that Koszul-perverse coherent sheaves can serve as analogues of dual canonical basis. It would be interesting to characterize the classes of these sheaves in K-theory the same way dual canonical basis are characterized (by some type of upper triangularity and positivity). 

A result similar to Theorem \ref{thm:main} appears in  \cite{SVV} where the authors construct by a different method a weakly monoidal equivalence between a category of representations of the quiver-Hecke algebra of type $A_{1}^{(1)}$ and equivariant perverse coherent sheaves on the nilpotent cone and the affine Grassmannian of type $A$. 

Since $\F$ from \eqref{eq:main} sends many simples to zero it is natural to ask if one can identify which simples are killed. 
In the subsequent paper \cite{CV} we  study the restriction of $\F$ to certain abelian subcategories $\calC_{w,v} \subset \calC$ studied in \cite{KKOP1}. 
Here $w,$ $v$ are explicit affine Weyl group elements which encode the dimension vector $\alpha$. We prove that $\F$ extends to a functor 
$\tF: \tcalC_{w,v} \to \cKP_{X^{(\bullet)}}$, where $\tcalC_{w,v}$ is a localization of $\calC_{w,v}$ introduced in \cite{KKOP2} to categorify open Richardson varieties. 
It turns out that this functor is faithful and identifies simples on either sides. In particular it yields an 
isomorphism between Grothendieck groups. This implies that $\tF$ categorifies an isomorphism between the Coulomb branch and an affine open Richardson variety and provides an enhancement of Theorem \ref{thm:main}.

Lastly, note that we only consider unframed quivers in this paper. In the framed case it is less clear what the appropriate duality datum should be. We hope to return to this question in the future. 

\subsection{Structure of paper}

The paper is organized as follows.
\begin{itemize}[leftmargin=8mm]
\item \S \ref{sec:quiverHecke} reviews the necessary background theory related to affine quiver Hecke algebras.
\item \S \ref{sec:catCB} discusses the Coulomb category, its heart $\cKP$ consisting of Koszul-perverse sheaves and the corresponding extended categories $\cKP_{X^{(m)}}$. 
\item \S \ref{sec:Fdefs} explains how to obtain the functor $\F$ from a duality datum. 
\item \S \ref{sec:quivercoulomb} introduces notation for objects in $\cKP$ and collects various properties such as the $\La(-,-)$ pairings between certain objects. 
\item \S \ref{sec:root-objects} studies root objects in $\cKP$ and their properties. 
\item \S \ref{sec:Fprops} identifies the main properties of the functor $\F$ and, in particular, shows that it is t-exact. 
\item \S \ref{sec:Fmore} shows that $\F$ sends simples to simples or zero and that it is (almost) essentially surjective on simples. 
\end{itemize}

\subsection{Acknowledgements}
We would like to thank P. Shan and H. Williams for inspiring discussions. 
S.C. was supported by NSERC Discovery Grant 2019-03961, PIMS fellowship and the Fondation des Sciences Math\'ematiques de Paris.
E.V. was supported by the UMR7586 (CNRS) and the ANR-24-CE40-3389.

\section{Affine quiver Hecke algebras}\label{sec:quiverHecke}

In this section we review some background representation theory of affine quiver Hecke algebras. 

\subsection{Affine root systems and convex orders}\label{sec:rootsystem}

Fix a simply laced quiver $Q = (I,E)$ of finite type and its affine extension $Q_\af = (I_\af,E_\af)$. 
We set $I_\af=\{0,1,\dots,\ell\}$ and $I=I_\af\ \smallsetminus\{0\}$.

Denote by $\Phi$ the associated simply laced finite root system, $\Delta = \{\alpha_1, \ldots, \alpha_n\}$ the set of simple roots, and $\Phi^+$ 
the set of positive finite roots. Denote by $\theta \in \Phi^+$ the highest root. Let $\P \supset \P^+ \supset \P^{++}$ denote the the set of weights, 
dominant weights and strictly dominant weights respectively. 

The associated affine root system $\Phi_\af$ is given by
\begin{equation*}
\Phi_\af = \{ \alpha + k\delta \,;\, \alpha \in \Phi , k \in \bbZ \} \cup \{ k\delta \,;\, k \in \bbZ \setminus \{0\} \} = \Phi_\re \cup \Phi_\im
\end{equation*}
where $\delta$ is the minimal positive imaginary root. The set of positive affine roots is
\begin{equation*}
\Phi_\af^+ = \Phi ^+ \cup \{ \alpha + k\delta \,;\, \alpha \in \Phi , k \in \mathbb{Z}_{>0} \} \cup \{ k\delta \,;\, k \in \mathbb{Z}_{>0} \}
\end{equation*}
The simple affine roots are $\Delta_\af = \Delta  \cup \{\alpha_0\}$ where $\alpha_0 = \delta - \theta$. We denote $\Phi_\af^- =-\Phi_\af^+$ and 
$\Phi^\pm_\re \subset \Phi_\af^\pm$ the subset of positive/negative real affine roots.  
Let $\Psi=\Phi_\re^+\cup\{\delta\}$ be the set of indivisible positive affine roots.

A pre-order $<$ on $\Phi_\af^+$ is called a convex order
if for any $\gamma_1, \gamma_2 \in \Phi_\af^+$ such that $\gamma_1 + \gamma_2 \in \Phi_\af^+$
we have either $\gamma_1 < \gamma_1 + \gamma_2 < \gamma_2$ or $\gamma_2 < \gamma_1 + \gamma_2 < \gamma_1$.
Let $W_\af = W  \ltimes \P$ be the affine Weyl group. 
For any $w\in W_\af$ with a reduced decomposition $w = s_{i_1} s_{i_2} \cdots s_{i_l}$ 
we equip the inversion set
\begin{equation*}
\calI(w) = \Phi_\af^+ \cap w(\Phi_\af^-) = \{ \beta_1, \beta_2, \ldots, \beta_l \}
\end{equation*}
with the order $\beta_1< \beta_2< \ldots<\beta_l$ where $\beta_k = s_{i_1} \cdots s_{i_{k-1}}(\alpha_{i_k})$. We are interested in convex 
orders extending this sequence, i.e., convex orders such that the inversion set is an initial set.

\begin{Proposition}\label{prop:minimalpair}
Suppose $\lambda \in \P^{++}$ and let $\leqslant$ be a convex order on $\Phi_\af^+$ extending the order on $\calI(t_\lambda)$ for a reduced 
decomposition of the translation $t_\lambda$ in $W_\af$. 
Let $\beta \in \Phi ^+$  be such that $\delta - \beta \notin \Delta_\af$. Set $I_\beta = \{ \alpha_i \in \Delta\,;\, (\beta, \alpha_i) = -1 \}$.
\begin{enumerate}[label=$\mathrm{(\alph*)}$,leftmargin=8mm]
\item
There exists a unique minimal simple finite root $\alpha_i\in I_\beta$ with respect to $<$.
\item
The pair $(\alpha_i\,;\, \delta - \alpha_i - \beta)$ consists of two positive affine roots summing to $\delta - \beta$.
\item
We have $\alpha_i > \delta  - \alpha_i-\beta$. There is no positive affine root $\gamma$ such that
\begin{equation}\label{eq:chain}
\alpha_i > \gamma > \delta - \beta - \gamma > \delta - \alpha_i - \beta
\end{equation}
meaning that $(\alpha_i\,;\, \delta - \beta - \alpha_i)$ is a minimal pair summing to $\delta - \beta$.
\end{enumerate}
\end{Proposition}

\begin{proof}
By hypothesis $\delta - \beta$ is a positive affine root such that $\delta - \beta \notin \Delta_\af$.
This yields
$\beta \neq \theta$.
Since $\beta \in \Phi ^+$ and $\beta\neq\theta$, there exists some $\alpha_k \in \Delta $ such that
\begin{equation}
\beta + \alpha_k \in \Phi ^+
\end{equation}
Since the root system is simply laced,  $\alpha_k$ is a simple finite root, and $\beta$, $\beta + \alpha_k$ are both positive roots, the inner product must satisfy $ (\beta, \alpha_k) = -1$.  Hence  the set $I_\beta $ is non-empty. For any $\alpha_i \in I_\beta$, we have
\begin{equation}
\delta - \beta = \alpha_i + (\delta - \alpha_i - \beta)
\end{equation}
Since $\beta + \alpha_i \in \Phi^+$, both $\alpha_i$ and $\delta - \alpha_i - \beta$ are positive affine roots.  The set $I_\beta$ is a non-empty, finite subset of the simple roots $\Delta $.  Since $<$ is a total order on $\Phi_\af^+$, we may define $\alpha_i$ to be this unique minimal element $\alpha_i= \min(I_\beta)$, proving parts (a) and (b) of the theorem.

We now prove that the pair $(\alpha_i, \delta - \beta - \alpha_i)$ forms a minimal pair.
We must check that
$\alpha_i > \delta - \beta - \alpha_i$ and that there exists no positive affine root $\gamma$ such that \eqref{eq:chain} holds. 

Because $\lambda$ is strictly dominant, the inversion set $\calI(t_\lambda) = \Phi_\af^+ \cap t_\lambda(\Phi_\af^-)$ contains $-\mu + \delta$ for each $\mu \in \Phi ^+$.  Thus  $\delta - \beta - \alpha_i \in \calI(t_\lambda)$.  Conversely, the simple finite root $\alpha_i$ lies outside $\calI(t_\lambda)$ and must appear after the inversion set elements, hence $\alpha_i > \delta - \beta - \alpha_i$, proving the first claim.

Suppose for contradiction that there exists a root $\gamma \in \Phi_\af^+$ satisfying \eqref{eq:chain}.
Since $\gamma + (\delta - \beta - \gamma) = \delta - \beta$, exactly one of the two positive affine roots 
$\gamma$ and $\delta - \beta - \gamma$ must contain the component $\delta$.

\begin{itemize}[leftmargin=8mm]
    \item {Case 1:} the affine root $\gamma$ contains $\delta$.  We claim that $\gamma$ belongs to the inversion set $\calI(t_\lambda)$.  Indeed, 
since $\delta - \beta - \gamma$ is a positive affine root without $\delta$ component, it is in $\Phi^+$.  Thus $\gamma=-(\delta-\gamma)+\delta$ 
is in $\calI(t_\lambda)$, because $-(\delta-\gamma)$ is a finite root and is the sum $-(\delta - \beta - \gamma)-\beta$ of two negative finite root, 
hence is a negative finite root.  Since $\gamma \in \calI(t_\lambda)$ and $\delta - \beta - \gamma$ is a positive finite root, we must have 
$\gamma< \delta - \beta - \gamma$.  This contradicts the hypothesis $\gamma > \delta - \beta - \gamma$.

    \item {Case 2:} the affine root $\gamma$ is a finite positive root.  We have $\gamma<\alpha_i $. We claim that this relation is impossible. 
Indeed, the sum $\beta + \gamma$ is a finite positive root because $\delta-\beta-\gamma$ is an affine root and $\beta$, $\gamma$ are both 
positive.  In a simply laced root system, this is true  only if  $(\beta, \gamma) = -1$.  Since $ \{ \rho \in \Phi^+ \,;\, \rho < \alpha_i \}$ is an initial 
segment of the convex order restricted to the finite roots, it is the inversion set $\calI(w)$ of some Weyl group element $w \in W$ such that 
$\gamma \in \calI(w)$.  Because $\alpha_i$ is the unique minimal element of $I_\beta$, any simple root $\alpha_j$ in $\calI(w)$ satisfies 
$(\beta, \alpha_j) \geqslant 0$.  Since $\gamma \in \calI(w)$ its support consists of simple roots in $\calI(w)$, hence 
$(\beta, \gamma) \geqslant 0$, yielding a contradiction.
\end{itemize}
\end{proof}

Given a convex order $<$ on the set $\Phi_\af^+$, we set
$$\Phi_{>\delta}=\{\beta\in\Phi^+_\af\,;\,\beta>\delta\}
,\quad
\Phi_{<\delta}=\{\beta\in\Phi^+_\af\,;\,\beta<\delta\}
$$
Let $p:\Phi_\af\to\Phi\cup\{0\}$ be the projection which maps $\pm\beta+n\delta$ to $\pm\beta$ for each $\beta\in\Phi\cup\{0\}$.  
There is a Weyl group element $w\in W$ such that $p(\Phi_{>\delta})=w(\Phi^+)$.  
If $w=1$ the preorder is said to be balanced. Balanced convex orders do exist in all type.
Let $\gamma_i=w(\alpha_i)$ and
$\gamma_i^\pm$ be the affinization of $\pm\gamma_i$, i.e., the minimal height positive real root in $\Phi_\re^+$ such that
$p(\gamma_i^\pm)=\pm\gamma_i$.
If the convex order is balanced, then we may set $\gamma_i=\alpha_i$ and
$$\gamma_i^+=\alpha_i>\delta>\gamma_i^-=\delta-\alpha_i.$$
We have
\begin{align*}
\Phi_\re^+\cap\Phi_{>\delta}=\Phi^++\bbN \delta
,\quad
\Phi_\re^+\cap\Phi_{<\delta}=-\Phi^++\bbZ_{>0}\delta.
\end{align*}

\begin{Remark}\label{rem:minimalpair}
Considering the infinite sequence of dominant translations $t_{n\lambda}$ with $n\to\infty$ one defines balanced convex orders on $\Phi_\af^+$ 
extending the order on $\calI(t_\lambda)$ for any reduced decomposition of $t_\lambda$. 
We will henceforth fix such a convex balanced order. 
\end{Remark}

An $\ell$-multipartition is an $\ell$-tuple $\mu=(\mu_1,\dots,\mu_\ell)$ where each $\mu_i$ is a partition.  
Let $\delta_i=(0,\dots,1,0\dots,0)$ be the $\ell$-multipartition of 1 with the 1 at the $i$th spot.  A root partition of $\beta \in \bbN\Phi_\af^+$ is a
 pair $(\uum,\mu)$ where $\uum$ is a tuple $(m_\psi\,;\,\psi\in \Psi)$ of non-negative integers such that 
 $\sum_{\psi\in\Psi}m_\psi\,\psi=\beta$ and $\mu$ is an $\ell$-multipartition of $m_\delta$.  
 We can write the root partition in the form 
\begin{align*}
(\psi_1^{m_1},\ldots,\psi_s^{m_s},\mu,\psi_{-t}^{m_t},\ldots,\psi_{-1}^{m_{-1}})
,\quad
\psi_1 > \cdots > \psi_s > \delta > \psi_{-t} > \cdots > \psi_{-1}
\end{align*}
where $m_u=m_{\psi_u}$.  Let $\Pi(\beta)$ be the set of root partitions of $\beta$. It is given the bilexicographic order such that
\begin{align*}
(\uum,\mu)\leqslant (\uun,\nu) \iff \uum\leqslant\uun
\iff (\uum\leqslant_l \uun\ \text{and}\ \uun\leqslant_r \uum)
\end{align*}
where $\leqslant_l$ and $\leqslant_r$ are the left and right lexicographic orders.

\begin{Remark}\hfill
The convexity of the pre-order implies that the root partition $(\beta)$ is a minimal element in $\Pi(\beta)$ for each $\beta\in\Phi_\re^+$, and that the root partition $(\mu)$ is a minimal element in $\Pi(n\delta)$ for each $\ell$-multipartition $\mu$ of a positive integer $n$. Moreover, a minimal pair for $\beta \in \Phi_\af^+$ is a minimal root partition in the set $\Pi(\beta)\setminus\{(\beta)\}$. By definition, a minimal pair exists if $\beta$ is not a simple affine root. 
\end{Remark}

Using the convexity of the order, one can prove that a minimal pair is always of the form $(\beta_1,\beta_2)$ where $\beta_1, \beta_2\in\Psi$, $\beta=\beta_1+\beta_2$, $\beta_1\succ\beta_2$, and there is no pair of roots $(\beta_1',\beta_2')$ such that $\beta=\beta_1'+\beta_2'$ and $\beta_1\succ\beta_1'\succ\beta_2'\succ\beta_2$.  The minimal pair $(\beta_1,\beta_2)$ is said to be real if $\beta_1$, $\beta_2$ are both real, i.e., the affine roots $\beta_1,$ $\beta_2$ both belong to $\Phi_\re^+$. 


\subsection{Quiver Hecke algebras}\label{subsec:quiverHecke}

We now recall some background theory of affine quiver Hecke algebras, following \cite{Kl}, \cite{Mc2}, \cite{KM1}.  
We fix a matrix of polynomials $(Q_{ij}(u,v)\,;\,i,j\in I_\af)$ in $\bbC[u-v]$ such that 
$Q_{ii}(u,v)=0$ and the following relations holds for all $i\neq j\in I_\af$
\begin{align}\label{QQ}
\deg Q_{ij}(u,v) = -\langle\alpha_i,\alpha_j\ra
,\quad
Q_{ij}(u,v) = Q_{ji}(v,u).\end{align}
For each $\beta\in\Q^+_\af$ of height $\height(\beta)=h$ let $R(\beta)$ be the symmetric quiver Hecke algebra associated with the matrix
$(Q_{ij}(u,v)\,;\,i,j\in I_\af)$.  
We use the same notation  as in \cite[\S1.2]{KKK} or \cite[\S2.2]{KKOP3}.
Thus, the algebra $R(\beta)$ is generated by the elements
\begin{align}\label{xte}
x_1\,,\,x_2\,,\,\dots\,,\,x_h\,,\,\tau_1\,,\,\tau_2\,,\,\dots\,,\,\tau_{h-1}\,,\,e(\nu)
\end{align}
 with $\nu$ an $h$-tuple in 
the set
$$(I_\af)^\beta=\{(\nu_1,\nu_2,\dots,\nu_h)\in(I_\af)^h\,;\,\nu_1+\nu_2+\cdots+\nu_h=\beta\},$$ 
modulo the relations in loc.~cit.
In particular, the elements $e(\nu)$ is an idempotent.
Given two tuples $\nu\in(I_\af)^\beta$ and $\nu'\in(I_\af)^{\beta'}$ let $e(\nu,\nu')$ be the idempotent corresponding to the concatenation of
$\nu$ and $\nu'$ and set
$$e(\beta,\beta')=\sum_{\nu\in(I_\af)^\beta\,,\,\nu'\in(I_\af)^{\beta'}}e(\nu,\nu')\in R(\beta+\beta').$$
Let $\gMod(R(\beta))$ be the $\infty$-category of finitely generated graded $R(\beta)$-modules and $\gMod(R(\beta))^\heartsuit$ its standard 
heart. We also consider the subcategories $\gmod(R(\beta))$ and $\gmod(R(\beta))^\heartsuit$ where we impose the condition that the modules 
are finite dimensional. We abbreviate
\begin{equation}\label{gmodules}
\gMod(R) = \bigoplus_{\beta \in \bbN\Phi_\af^+} \gMod(R(\beta))
\end{equation}
and similarly for $\gMod(R)^\heartsuit,$ $ \gmod(R)$ and $ \gmod(R)^\heartsuit$. 
Forgetting the grading we get categories $\Mod(R),$ $ \Mod(R)^\heartsuit$ etc. 
Let $\circ$ be the induction functor 
$$\circ:\gMod(R(\beta))^\heartsuit\times \gMod(R(\beta'))^\heartsuit\to \gMod(R(\beta+\beta'))^\heartsuit ,\quad \beta,\beta' \in \bbN\Phi_\af^+.$$
Recall that
\begin{align}\label{circ}M\circ M'=R(\beta+\gamma)e(\beta,\gamma)\otimes_{R(\beta)\otimes R(\beta')}(M\otimes M').\end{align}
For any graded $R(\beta)$-module $M$, we define the grading shift functor $\{1\}$ as
\begin{align}\label{shift-KLR}
(M\{1\})_k = M_{k+1}.
\end{align}  
In K-theory, the shift $\{1\}$ is denoted by multiplication by $q^{-1}$.
 
The anti-automorphism of $R(\beta)$ which fixes all standard generators equips the dual space $M^*=\Hom_\bbC(M,\bbC)$ of an 
$R(\beta)$-module with an $R(\beta)$-action. The dual may not be finitely generated if $M$ is not finite dimensional.  A graded module 
$M=\bigoplus_{k\in\bbZ}M_k$ is self-dual if $M^* \cong M$ as a graded module. Every simple module is isomorphic to the grading shift of a 
unique self-dual simple module.

\subsection{Simple and standard modules}\label{sec:SSmodules}

For each indivisible positive root $\beta\in\Psi$ and each $n \in \bbZ_{>0}$, let
$\gMod(R(n\beta))^\cusp$ be the full subcategory of cuspidal graded modules in $\gMod(R(n\beta))^\heartsuit$ 
considered in \cite{Mc2}. If $\beta$ is a real positive affine root there is a unique 
self-dual simple cuspidal graded $R(\beta)$-module $L(\beta)$, up to isomorphisms.
If $\beta=\delta$, there are $\ell$ isomorphism classes  $L(\delta_i)$, 
with $i\in I$, of self-dual simple cuspidal graded $R(\delta)$-modules.
The following hold.
\begin{itemize}[leftmargin=8mm]
\item
If $\beta\in\Phi_\re^+$ and $n\in\bbZ_{>0}$ there is a unique self-dual simple cuspidal graded $R(n\beta)$-module $L(\beta^n)$. 
We have
$L(\beta)^{\circ n}=q^{n(n-1)/2}L(\beta^n).$
Let $\Delta(\beta^n)$ be the projective cover of
$L(\beta^n)$ in the category $\gMod(R(n\beta))^\cusp$.
We have
$\Delta(\beta)^{\circ n}=q^{n(n-1)/2}[n]!\Delta(\beta^n).$
\item
For each $n \in \bbZ_{>0}$ we have an algebra isomorphism 
$\End_{R(n\delta)}(L(\delta_i)^{\circ n}) \cong \bbC\frakS_n$
which yields a parametrization of the composition factors of $L(\delta_i)^{\circ n}$ by partitions
of $n$. Let $L(\mu_i)$ be the $\mu$-th factor.
\item
Given an $\ell$-multipartition $\mu=(\mu_1,\dots,\mu_\ell)$ of $n$, 
we define the graded $R(n\delta)$-module
$L(\mu)=L(\mu_1)\circ\cdots\circ L(\mu_\ell).$
It is self-dual simple and cuspidal.
Let $\Delta(\mu)$ be the projective cover of
$L(\mu)$ in $\gMod(R(n\delta))^\cusp$.
\end{itemize}

For each root partition $\pi$ of $\beta$ we define the proper standard module $\bd(\pi)$ to be the graded module
given by
\begin{align}\label{eq:bd}
\bd(\pi)&=L(\psi_1^{m_1})\circ\cdots\circ L(\psi_s^{m_s})\circ L(\mu)\circ
L(\psi_{-t}^{m_{-t}})\circ\cdots\circ L(\psi_{-1}^{m_{-1}}),
\end{align}
and the standard module $\Delta(\pi)$ to be the graded module given by
\begin{align*}
\Delta(\pi)&=\Delta(\psi_1^{m_1})\circ\cdots\circ \Delta(\psi_s^{m_s})\circ \Delta(\mu)\circ
\Delta(\psi_{-t}^{m_{-t}})\circ\cdots\circ \Delta(\psi_{-1}^{m_{-1}})
\end{align*}
Proper standard modules are finite dimensional while standard modules are not, meaning that
$$\bd(\pi)\in \gmod(R)^\heartsuit ,\quad \Delta(\pi)\in \gMod(R)^\heartsuit.$$

\begin{Proposition}\label{prop:hereditary}
Let $\pi$, $\omega$ be root partitions.\hfill
\begin{enumerate}[label=$\mathrm{(\alph*)}$,leftmargin=8mm]
\item
$\bd(\pi)$ is a quotient of $\Delta(\pi)$.
Let $L(\pi)$ be the head of $\bd(\pi)$ and $\Delta(\pi)$.
It is simple and self-dual.
\item
$\{L(\pi)\{n\}\,;\,\pi\in\Pi(\beta)\,,\,n\in\bbZ\}$ is a complete and irredundant system of simple graded $R(\beta)$-modules up to isomorphism.
\item
$[\bd(\pi)\,:\,L(\pi)]=1$ and $[\bd(\pi)\,:\, L(\omega)]=0$ unless $\omega\leqslant\pi$.
\qed
\end{enumerate}
\end{Proposition}

\begin{Proposition}\label{prop:ses1}
We have the following short exact sequences 
\begin{align}
\label{eq:Ddelta} && 0 \to \Delta(\alpha_i) \circ \Delta(-\alpha_i+\delta) \{-2\} \to \Delta(-\alpha_i+\delta) \circ \Delta(\alpha_i) \to \Delta(\delta_i) \to 0 \\
\label{eq:DL} && 0 \to \Delta(\rho) \{-2\} \to \Delta(\rho) \to L(\rho) \to 0 \ \ \text{ for $\rho \in \Phi_\re^+$. }
\end{align}
\end{Proposition}
\begin{proof}
Sequence \eqref{eq:Ddelta} follows from \cite[Lem.~6.21]{KM1} and \eqref{eq:DL} from \cite[Thm.~18.2]{Mc2}.
\end{proof}
 
\begin{Proposition}\label{prop:ses2}
We have the following relationships among simple objects.
\hfill
\begin{enumerate}[label=$\mathrm{(\alph*)}$,leftmargin=8mm]
\item
If $(\beta,\gamma)$ is a real minimal pair for $\rho \in \Phi_\re^+$ then one has short exact sequences
\begin{align}
\label{eq:Lminpair}
\begin{split}
0 \to L(\rho) \{-1\} \to L(\beta) \circ L(\gamma) \to L(\beta,\gamma) \to 0 \\
0 \to L(\beta,\gamma) \{-1\} \to L(\gamma) \circ L(\beta) \to L(\rho) \to 0.
\end{split}
\end{align}
\item
If $\rho = \alpha_i + n\delta$ for $n \ge 0$ then one has short exact sequences
\begin{align}
\label{eq:Lal1}
\begin{split}
 0 \to L(\rho+\delta) \{-1\} \to L(\rho)\circ L(\delta_i) \to L(\rho,\delta_i) \to 0 \\
 0 \to L(\rho,\delta_i) \to L(\delta_i) \circ L(\rho) \to L(\rho+\delta) \{1\} \to 0.
\end{split}
\end{align}
\item
If $\rho = -\alpha_i + (n+1)\delta$ for $n \ge 0$ then one has short exact sequences
\begin{align}
\label{eq:Lal2}
\begin{split}
0 \to L(\rho+\delta) \{-1\} \to L(\delta_i) \circ L(\rho) \to L(\delta_i,\rho) \to 0 \\
0 \to L(\delta_i,\rho) \to L(\rho) \circ L(\delta_i) \to L(\rho+\delta) \{1\} \to 0.
\end{split}
\end{align}
\item The sequences 
\begin{align}
\label{eq:Ldelta1}
\begin{split}
0 \to L(\delta_i) \{-2\} \to L(\alpha_i) \circ L(-\alpha_i+\delta) \to L(\alpha_i,-\alpha_i+\delta) \to 0 \\
0 \to L(\alpha_i,-\alpha_i+\delta) \{-2\} \to L(-\alpha_i+\delta) \circ L(\alpha_i) \to L(\delta_i) \to 0
\end{split}
\end{align}
are exact on the left and right but not necessarily in the middle. 
In the Grothendieck group we have
\begin{align}
\label{eq:Ldelta2}
\begin{split}
[L(\alpha_i) \circ L(-\alpha_i+\delta)] = q^{2}[L(\delta_i)] + [L(\alpha_i,-\alpha_i+\delta)] + q \sum_{i \leftrightarrow j} [L(\delta_j)] \\
[L(-\alpha_i+\delta) \circ L(\alpha_i)] = [L(\delta_i)] + q^{2}[L(\alpha_i,-\alpha_i+\delta)] + q  \sum_{i \leftrightarrow j} [L(\delta_j)]
\end{split}
\end{align}
\end{enumerate}
\end{Proposition}

\begin{proof}
The exact sequences \eqref{eq:Lminpair} follow from \cite[(6.2)]{Kl}, \eqref{eq:Lal1} from \cite[Prop.~6.7]{Kl}, \eqref{eq:Lal2} 
from \cite[Prop.~6.8]{Kl}. Finally, we concentrate on part (d).
By Proposition \ref{prop:lsubquotient},
in the Grothendieck group, we have
$$[L(\alpha_i) \circ L(-\alpha_i+\delta)]=[L(\alpha_i) \n L(-\alpha_i+\delta)]+[L(\alpha_i) \Delta L(-\alpha_i+\delta)]+\sum_kq^{a_k}[L_k]$$
with simple selfdual objects $L_k$ not isomorphic to 
$L(\alpha_i) \n L(-\alpha_i+\delta)$, $L(\alpha_i) \Delta L(-\alpha_i+\delta)$ and their shifts and $a_k\in\bbZ$. 
Since $\alpha_i>-\alpha_i+\delta$, we have
$L(\alpha_i) \n L(-\alpha_i+\delta)=L(\alpha_i,-\alpha_i+\delta)$
by definition of $L(\alpha_i,-\alpha_i+\delta)$.
By \eqref{eq:Ldelta1}, we also have
$L(\alpha_i) \Delta L(-\alpha_i+\delta)=q^2L(\delta_i).$
Hence, we get
$$[L(\alpha_i) \circ L(-\alpha_i+\delta)]=[L(\alpha_i,-\alpha_i+\delta)]+q^2[L(\delta_i)]+\sum_kq^{a_k}[L_k]$$
applying the duality and \cite[(5.2)]{Mc2}, we deduce that
$$[L(-\alpha_i+\delta)\circ L(\alpha_i)]=q^2[L(\alpha_i,-\alpha_i+\delta)]+[L(\delta_i)]+\sum_kq^{2-a_k}[L_k]$$
Next, by \eqref{eq:DL} we have 
$$[L(\alpha_i)]=(1-q^2)[\Delta(\alpha_i)],\quad [L(-\alpha_i+\delta)]=(1-q^2)[\Delta(-\alpha_i+\delta)]$$
which together with \eqref{eq:Ddelta}, \cite[Prop. 5.4]{McT} implies
$$q^2[L(\alpha_i) \circ L(-\alpha_i+\delta)]-[L(-\alpha_i+\delta)\circ L(\alpha_i)] = (q^4-1)[L(\delta_i)]+q(q^2-1)\sum_{i \leftrightarrow j} [L(\delta_j)].$$
Thus, we have
$$\sum_k(q^{a_k}-q^{-a_k})[L_k]=(q-q^{-1})\sum_{i \leftrightarrow j} [L(\delta_j)]$$
To conclude it is enough to observe that by \cite[Thm.~24.10]{Mc2} we have $a_k>0$ for all $k$,
from which we deduce that indeed $a_k=1$ for all $k$ and the family $(L_k)$ is precisely $(L(\delta_j)\,;\,i \leftrightarrow j)$.
\end{proof}

By \cite{KM2}, the cuspidal $R(\delta)$-modules are controlled by $B_\delta = \C[z] \otimes \sf{Z}$ 
where $\sf{Z}$ is the zig-zag algebra of the finite Dynkin diagram $Q=(I,E)$. 
In particular, for any connected vertices $i,j \in I$ there exists a unique non-trivial extension 
\begin{equation}\label{eq:delta}
0 \to L(\delta_i) \{-1\} \to L(\delta_{ij}) \to L(\delta_j) \to 0. 
\end{equation}

\begin{Corollary}\label{cor:gamma}
If $i,j \in I$ with $\la \alpha_i,\alpha_j \ra = -1$ then, up to scaling, there are unique nonzero maps 
\begin{equation}\label{eq:gamma}
L(-\alpha_j+\delta) \circ L(\alpha_j) \to L(\delta_{ij}) \to L(\alpha_i) \circ L(-\alpha_i+\delta) \{1\}
\end{equation}
where the first map is surjective and the second injective. We denote the composition $\gamma_{ji}$.
\end{Corollary}
\begin{proof}
Consider the cokernel $\coker$ of 
$$L(\alpha_j, -\alpha_j+\delta)\{-2\} \to L(-\alpha_j+\delta) \circ L(\alpha_j)$$ 
from \eqref{eq:Ldelta1}. By Proposition \ref{prop:ses2} the composition factors of $\coker$ consist of 
$\{L(\delta_j), L(\delta_k)\{-1\}\,;\,j \leftrightarrow k\}$. Since $L(\delta_j)$ is the head of $L(-\alpha_j+\delta) \circ L(\alpha_j)$ we also have a
 surjective map $\coker \to L(\delta_j)$. The kernel of this map has composition factors $\{L(\delta_k)\{-1\}\,;\,k \leftrightarrow j\}$. 
 Since $Q$ is simply laced, finite type it does not contain any 3-cycles and thus $\Ext^1(L(\delta_k), L(\delta_{k'}) = 0$ for distinct 
 $k,k' \ne j$ with $k \leftrightarrow j \leftrightarrow k'$. This means that one has a short exact sequence
$$0 \to \bigoplus_{k \leftrightarrow j} L(\delta_k) \{-1\} \to \coker \to L(\delta_j) \to 0.$$
Thus quotienting $\coker$ by the direct sum of $L(\delta_k)$ where $k \leftrightarrow j$ and $k \ne i$ we get a quotient with composition 
factors $\{L(\delta_j), L(\delta_i)\{-1\}\}$ and with simple head $L(\delta_j)$. This quotient must be $L(\delta_{ij})$ which proves the existence 
and surjectivity of the first map in \eqref{eq:gamma}. The existence and injectivity of the second map in \eqref{eq:gamma} follows similarly.
\end{proof}

\subsection{Affinizations and renormalized $r$-matrices}\label{sec:ren}

An important result in \cite[\S 2.2]{KKKO2} is that the category $\gmod(R)^\heartsuit$ is equipped with a system of renormalized $r$-matrices.  
In the Appendix \ref{app:A} we recall some relevant properties shared by abelian categories with renormalized $r$-matrices.  
Let us recall briefly the definition of the renormalized $r$-matrices, following \cite[Def.~2.2.1]{KKKO2}.

An affinization $M_z$ of a graded $R(\beta)$-module $M$ is a module with an injective degree two endomorphism $z$ and an isomorphism 
$M_z/zM_z \cong M$ satisfying the conditions in \cite[Def.~2.2]{KP}. For symmetric quiver Hecke algebras, such as the ones in this paper, 
an affinization exists with underlying graded vector space $M[z] = M \otimes \bbC[z]$ and action of $R(\beta)$ given by 
\begin{align}\label{eq:affinization}
\begin{split}
e(\nu)(a \otimes u) &= a \otimes e(\nu)u \\
x_k(a \otimes u) &= a \otimes (x_ku) + (za) \otimes u \\
\tau_k(a \otimes u) &= a \otimes (\tau_k u)
\end{split}
\end{align}
where $\nu \in I^\beta, a \in \bbC[z]$ and $u \in M$. 

\begin{Remark}\label{rem:barR}
Let $z$ be of degree two.
The assignment
\begin{align}\label{eq:Rmap}
\begin{split}
R(\beta) \to R(\beta)[z],\quad
x_i \mapsto x_i + z ,\quad
\tau_i \mapsto \tau_i
\end{split}
\end{align}
defines a homomorphism. If $M$ is an $R(\beta)$-module then $M[z]$ is an $R(\beta)[z]$-module and thus an $R(\beta)$-module via \eqref{eq:Rmap}. Then $M_z \cong M[z]$ as $R(\beta)$-modules. 
\end{Remark}

Given nonzero graded modules $M,N\in \gmod(R)^\heartsuit$ one can define an $R$-module homomorphism
\begin{align}\label{KRMN}
\bfR_{M_z,N}:M_z \circ N \to N \circ M_z.
\end{align}
From this one obtains the renormalized map 
\begin{align}\label{KRrenMN}\bfR_{M_z,N}^\ren=z^{-s}\bfR_{M_z,N}\end{align} 
where $s \ge 0$ is the largest integer such that the image of $\bfR_{M_z,N}$ is contained in $z^s(N\circ M_z)$. 
We define
\begin{align}\label{KrMN}
\r_{M,N} = \bfR_{M_z,N}^\ren|_{z=0}: M * N\to N * M \{\Lambda(M,N)\}
\end{align}
where $\Lambda(M,N) \in \Z$ is the homogeneous degree of $\r_{M,N}$.

\section{Categorification of the Coulomb branch}\label{sec:catCB}

In this section we set up notation and recall some background on Coulomb branches of 4d $\calN = 2$ quiver gauge theories and the associated categories. We follow the conventions from \cite{CW2}. 

\subsection{Coulomb branches}\label{sec:CB}

Given a finite dimensional representation $N$ of a complex reductive group $G$, let $\calR_{G,N}^\cl$ be the classical space of triples as 
defined by Braverman, Nakajima and Finkelberg \cite{Nak, BFNa,BFNb}, and let $\calR_{G,N}$ be its derived version studied in \cite{CW2}. 
When $G$ and $N$ are clear from the context, we abbreviate $\calR^\cl = \calR_{G,N}^\cl$ and $\calR = \calR_{G,N}$. The space $\calR^\cl$ 
is an ind-scheme of infinite type which is given by the classical intersection of two infinite dimensional bundles over the affine Grassmannian 
$\Gr_G$. The derived intersection of these bundles yields the derived ind-scheme $\calR$. 
More precisely, set $\calO=\bbC\llb t\rrb$ and $K=\bbC\llp t\rrp$ and denote the corresponding loop functors applied to $G$ by $G_\calO$ and $G_\calK$ respectively. We abbreviate $N_\calO=N\otimes\calO$ and $N_\calK = N \otimes  \calK$. Define the classical ind-scheme of infinite type
$$\calT_{G,N} = \calT = G_\calK \times_{G_\calO} N_\calO.$$
The derived ind-scheme $\calR$ is given by the following Cartesian square in the category of derived ind-schemes
\begin{align}\label{def-R}
\begin{split}
\xymatrix{\calR\ar[r]^-{i_1}\ar[d]_-{i_2}&\Gr\times N_\calO\ar[d]\\
\calT\ar[r]&\Gr\times N_\calK}
\end{split}
\end{align}
We equip $\calR$ with the obvious action of the group 
\begin{align}\label{hatG}\hGO = (G_\calO \rtimes \bbC^\times) \times \bbC^\times.\end{align}
The inner $\bbC^\times$ is called the loop $\bbC^\times$, 
the outer $\bbC^\times$ the scaling $\bbC^\times$.
The $\bbC^\times$'s act as follow
\begin{itemize}[leftmargin=8mm]
\item
the loop $\bbC^\times$ acts by loop rotation on $\Gr_G$ with weight two on 
the uniformizer $t$ and weight $-1$ on $N$, 
\item
the scaling $\bbC^\times$ acts trivially on $\Gr_G$ and with weight $1$ on $N$.
\end{itemize}
Let $\{1\}$ and $\la 1 \ra$ denote the twists by the linear characters of 
these $\C^\times$. In K-theory, the shifts $\{1\}$ and $\la 1 \ra$ are denoted by the multiplication by $q^{-1}$ and $t$ respectively. 

The K-theory $K^{\hGO}(\calR)$ of the quotient stack $\calR/\hGO$ is equipped with a product \cite{BFNa}. 
Although described more explicitly in \cite{BFNa} this product is induced by the natural convolution structure from the identification 
\begin{align}\label{R/G} 
\calR/\hGO \cong N_\O/\hGO \times_{N_\calK/\hGK} N_\O/\hGO 
\end{align}
as explained in \cite[\S 5]{CW2}. If we replace $\hGO$ by $G_\O$ then this product becomes commutative and the scheme 
\begin{align}\label{specR}
\calM_{G,N} = \Spec(K^{G_\O}(\calR) \otimes \bbC)
\end{align}
is the Coulomb branch of the 4d, $\calN=2$ quiver gauge theory (of cotangent type) associated to $(G,N \oplus N^\vee)$. 

\begin{Remark}
The notation $\{1\}$ and $\la 1 \ra$ agrees with \cite{CW2} 
but differs from \cite[p.~719]{CW1} where the scaling $\bbC^\times$ did not play a role and $\la 1 \ra$ was used as short-hand for $[1]\{-1\}$.
Note that \cite[\S8]{CW2} uses the diagonal $\C^\times$ in $GL_n$ as the scaling $\C^\times$ so the shift $\langle 1 \rangle$ in that section  does not refer to our scaling shift but rather to $[1]\{-1\}$.
\end{Remark}

\subsection{Coulomb categories}\label{sec:coulombcat}

We are interested in the stable $\infty$-category $\calD = \Coh^{\hGO}(\calR) = \Coh(\calR/\hGO)$ of $\hGO$-equivariant coherent sheaves on 
$\calR$. Since the stack $\calR/\hGO$ is ind-tamely presented in the sense of \cite{CW3} the coherent sheaves are the bounded 
objects in $\IndCoh^{\hGO}(\calR)$ with coherent cohomology. 
The identification from \eqref{R/G} can be used to endow $\Coh^{\hGO}(\calR)$ with a product. 

\subsubsection{Factorization categories}\label{sec:factcat}

For the purposes of defining the main functor in this paper we will need to consider certain enlargements of these categories as in 
\cite[\S 7]{CW2}. To describe these we first consider the case $N=0$ where $\calR_{G,0} = \Gr$ is the affine Grassmannian. In this case $\Gr$ 
can be deformed to a factorization space $\Gr_{X^J}$ which is the Beilinson-Drinfeld (BD) Grassmannian over $X^J$, where $X = \bA^1$ 
denotes the affine line and $J$ is some finite index set. This space represents the presheaf over $X^J$ defined by having $S$-points 
triples $(f_J, \calE, \beta)$ where $f_J: S \to X^J$, $\calE$ is a $G$-torsor over $X \times S$ and $\beta$ is a trivialization of $\calE$ away from 
the graph of $f_J$, see, e.g., \cite[Def. 3.1.1]{Z}. 

The group $G_\O$ also has a global analogue $G_{\O,X^J}$, given by the $J$-jet space of the affine group $G$, which acts naturally on $\Gr_{X^J}$. We also take 
$$\hG_{\O,X^J} = (G_{\O,X^J} \rtimes \bbC^\times) \times \bbC^\times$$
where the two $\bbC^\times$ are the loop and scaling $\bbC^\times$ as before, acting on $X^J$ with weight two and zero respectively. 
The category $\Coh^{\hG_{\O,X^J}}(\Gr_{X^J})$ carries a functorial convolution product 
\begin{align*}
\Coh^{\hG_{\O,X^J}}(\Gr_{X^J}) \otimes \Coh^{\hG_{\O,X^J}}(\Gr_{X^J})  \to \Coh^{\hG_{\O,X^J}}(\Gr_{X^J})
,\quad
(\cF, \cG)  \mapsto \cF*\cG.
\end{align*}
There is also a unital structure consisting of monoidal functors 
\begin{align*}
\eta_{J_1}^{J_1 \sqcup J_2}: & \Coh^{\hG_{\O,X^{J_1}}}(\Gr_{X^{J_1}} \times X^{J_2}) \to \Coh^{\hG_{\O,X^{J_1 \sqcup J_2}}}(\Gr_{X^{J_1 \sqcup J_2}}) \\
\eta_{J_2}^{J_1 \sqcup J_2}: & \Coh^{\hG_{\O,X^{J_2}}}(X^{J_1} \times \Gr_{X^{J_2}}) \to \Coh^{\hG_{\O,X^{J_1 \sqcup J_2}}}(\Gr_{X^{J_1 \sqcup J_2}})
\end{align*}
for any $J_1,J_2$. These can be used to define another product $\odot$ as follows
\begin{align*}
\Coh^{\hG_{\O,X^{J_1}}}(\Gr_{X^{J_1}}) \otimes \Coh^{\hG_{\O,X^{J_2}}}(\Gr_{X^{J_2}}) & \to \Coh^{\hG_{\O,X^{J_1 \sqcup J_2}}}(\Gr_{X^{J_1 \sqcup J_2}})  \\
(\cF, \cG) & \mapsto \cF \odot \cG = \eta_{J_1}^{J_1 \sqcup J_2}(\cF \boxtimes \O_{X^{J_2}}) * \eta_{J_2}^{J_1 \sqcup J_2}(\O_{X^{J_1}} \boxtimes \cG)
\end{align*}

As explained in \cite[\S 7]{CW2} these constructions all have analogues when we replace $\Gr_{X^J}$ with $\calR_{X^J}$. More precisely, the role of \eqref{R/G} is played by 
\begin{equation}\label{R/G2}
\calR_{X^J}/\hG_{\O,X^J} \cong N_{\O,X^J}/\hG_{\O,X^J} \times_{N_{\calK,X^J}/\hG_{\calK,X^J}} N_{\O,X^J}/\hG_{\O,X^J}
\end{equation}
where $N_{\O,X^J}$ and $N_{\calK,X^J}$ are defined by having $S$-points
\begin{align*}
N_{\O,X^J}(S) &= \{f_J: S \to X^J, \rho: \hz_J \to N\} \\
N_{\calK,X^J}(S) &= \{f_J: S \to X^J, \rho: \hz_J \setminus z_J \to N\} 
\end{align*}
where $z_J \subset S \times X$ is the graph of $f_J$ and $\hz_J$ is its formal neighborhood. We denote 
$$\calD_{X^J} = \Coh^{\hG_{\O,X^J}}(\calR_{X^J}).$$

Subsequently we obtain a monoidal product $*$ on each $\calD_{X^J}$ as well as the product $\odot$ given by
\begin{align}\label{eq:odot}
\begin{split}
\calD_{X^{J_1}} \otimes \calD_{X^{J_2}} & \to \calD_{X^{J_1 \sqcup J_2}} \\
(\cF, \cG) & \mapsto \cF \odot \cG = \eta_{J_1}^{J_1 \sqcup J_2}(\cF \boxtimes \O_{X^{J_2}}) * \eta_{J_2}^{J_1 \sqcup J_2}(\O_{X^{J_1}} \boxtimes \cG).
\end{split}
\end{align}

\subsubsection{Symmetrized categories}\label{sec:symmcat}

Let $\frakS_J$ be the symmetric group permuting the finite set $J$ and $X^{(J)} = X^J/\frakS_J$. 
The BD Grassmannian $\Gr_{X^J}$ has a natural action of $\frakS_J$ arising from permuting the points in $X^J$. 
We consider the symmetrized version $\Gr_{X^{(J)}}$ of $\Gr_{X^J}$ which represents the same functor except 
that now $f_J: S \to X^{(J)}$ is a relative effective divisor over $X \times S$, see, e.g., \cite[Eq.~3.1.24]{Z}. 
Replacing $X^J$ with $X^{(J)}$ one gets analogues $\hG_{\O,X^{(J)}}$ and $\hG_{\calK,X^{(J)}}$ as well as $N_{\O,X^{(J)}},$
 $N_{\calK,X^{(J)}}$ and $\calR_{X^{(J)}}$, see, e.g., \cite[\S 3]{Ri}. 
The fiber product corresponding to \eqref{R/G2} is now
\begin{equation}\label{R/G3}
\calR_{X^{(J)}}/\hG_{\O,X^{(J)}} \cong N_{\O,X^{(J)}}/\hG_{\O,X^{(J)}} 
\times_{N_{\calK,X^{(J)}}/\hG_{\calK,X^{(J)}}} N_{\O,X^{(J)}}/\hG_{\O,X^{(J)}}. 
\end{equation}
Following our earlier notation we denote $\calD_{X^{(J)}} = \Coh^{\hG_{\O,X^{(J)}}}(\calR_{X^{(J)}})$. 
As before we have a monoidal structure on each $\calD_{X^{(J)}}$ whose product is denoted $*$ as well as an 
analogue of $\odot$ from \eqref{eq:odot}
\begin{align}\label{eq:circ} 
\begin{split}
\calD_{X^{(J_1)}} \otimes \calD_{X^{(J_2)}}  \to \calD_{X^{(J_1 \sqcup J_2)}}
,\quad
\notag (\cF, \cG)  \mapsto \cF \circ \cG. 
\end{split}
\end{align}
which we denote $\circ$ in order to easily distinguish it from $\odot$. 

Let $\Pi$ denote the natural projection
\begin{equation}\label{eq:Pi}
\Pi: \calR_{X^J}/\hG_{\O,X^J} \to \calR_{X^{(J)}}/\hG_{\O,X^{(J)}}
\end{equation}
The compatibility with base change $X^J \to X^{(J)}$ implies that 
\begin{equation}\label{eq:Pi2} 
\Pi(\cF \odot \cG) \cong \Pi(\cF) \circ \Pi(\cG) \in \calD_{X^{J_1 \sqcup J_2}}
\end{equation}
for $\cF \in \calD_{X^{J_1}}$ and $\cG \in \calD_{X^{J_2}}$.

\subsubsection{Subcategories supported over $0$}

We can also consider the full subcategory of $\calD_{X^J,0} \subset \calD_{X^J}$ consisting of objects set theoretically supported over 
$\calR = \calR_{X^J} \times_{X^J} \{0\} \subset \calR_{X^J}$. Compatibility with base change implies that these subcategories are preserved 
by $*$ and $\odot$. One similarly has a subcategory $\calD_{X^{(J)},0} \subset \calD_{X^{(J)}}$ which is preserved by $*$ and $\circ$.

\subsubsection{Colimit categories}\label{sec:colimit}

The categories $\calD_{X^{(J)}}$ do not carry a traditional factorization structure. However, for any $d \in \N$ we have a natural map 
\begin{align}\label{eq:dD}
\begin{split}
X^{(m)}  \to X^{(dm)}, \quad
[D]  \mapsto [dD]
\end{split}
\end{align}
where $X^{(m)} = X^m/\frakS_m$ and $D$ denotes a degree $m$ divisor on $X$. Let $\calD_{X^{(m)}}$ be 
the category $\Coh^{\hG_{\O,X^{(m)}}}(\calR_{X^{(m)}})$ over $X^{(m)}$. The stacks $\calR_{X^{(m)}}/\hG_{\O,X^{(m)}}$ are compatible 
under the base change \eqref{eq:dD} in the sense that 
$$\calR_{X^{(m)}}/\hG_{\O,X^{(m)}} \cong \calR_{X^{(dm)}}/\hG_{\O,X^{(dm)}} \times_{X^{(dm)}} X^{(m)}.$$
Thus, via pushforward, we get faithful functors 
\begin{equation}\label{eq:i}
i_m^d: \calD_{X^{(m)}} \to \calD_{X^{(md)}}.
\end{equation}
which are monoidal with respect to the $\circ$ product. 
Taking the colimit over such maps we get a category that we denote $\calD_{X^{(\bullet)}}$. The functor in \eqref{eq:i} restricts to a faithful functor
\begin{equation}\label{eq:i2}
i_m^d: \calD_{X^{(m)},0} \to \calD_{X^{(md)},0}
\end{equation}
which can be used to define the analogous category $\calD_{X^{(\bullet)},0}$. 

Note that in the discussion above one interprets $X^{(0)}$ as just a point so that $\calD_{X^{(0)}} = \calD_{X^{(0)},0} = \calD$. 
In particular one has a natural faithful functor 
\begin{equation}\label{eq:i3}
i: \calD \to \calD_{X^{(\bullet)},0}.
\end{equation}
This functor is compatible with the monoidal structures in the sense that $i(\cF*\cG) \cong i(\cF) \circ i(\cG)$ for any $\cF,\cG \in \calD$. 
Because of this compatibility we will often work with and make computations in $\calD$ and only after apply $i$. 

The functors from \eqref{eq:i2} induce an isomorphism in K-theory.  
This means that the map $i$ from \eqref{eq:i3} is also an isomorphism in K-theory, see Remark \ref{rem:Ktheory}.  

\begin{Remark}
The categories $\calD_{X^m,0}$ and $\calD_{X^{(m)},0}$ are very similar. The discussion above is still valid with $\calD_{X^m,0}$ in place of 
$\calD_{X^{(m)},0}$. We work with the latter in order to later define the duality functor $\F$ in \S \ref{sec:Fdefs}. 
\end{Remark}



\subsection{Koszul-perverse coherent sheaves}\label{sec:KP}

In \cite{CW2} a t-structure on $\calD = \Coh^{\hGO}(\calR)$ is constructed for which the $*$-product is t-exact.
Note that the $*$-product is not t-exact for the standard t-structure. 
We call it the Koszul-perverse t-structure and denote its heart $\cKP_{G,N}$, or $\cKP$ for short. 
It is proved there that $\cKP$ is a finite-length Abelian, rigid, monoidal category  which
 is $\Z^2$-graded by the loop shift $\{1\}$ and the Koszul shift $[1]\la -1 \ra$. 
For $\cF \in \calD$ we denote by $H^i_{\cKP}(\cF) \in \cKP$ its cohomology with respect to the Koszul-perverse t-structure. 

This t-structure can be extended from $\calD$ to $\calD_{X^J,0}$. We briefly sketch this construction. 
Recall that the Koszul-perverse t-structure on $\calD$ begins with the perverse coherent t-structure on 
$\Coh^{\hGO}(\Gr)$ which is then lifted from $\Gr$ to $\calR$ using the fact that $\calR \to \Gr$ is the derived intersection of two bundles over $\Gr$. 
The key thing to note is that the construction of the perverse coherent t-structure on $\Coh^{\hGO}(\Gr)$ extends via the usual definition to 
$\Coh_0^{\hG_{\O,X^J}}(\Gr_{X^J})$. 
One then lifts this to $\calD_{X^J,0}$ in exactly the same way as before. 
One similarly obtains a Koszul-perverse t-structure on 
$\calD_{X^{(J)},0}$. We denote the corresponding hearts $\cKP_{X^J}$ and $\cKP_{X^{(J)}}$. 

\begin{Remark}
The perverse t-structure on $\Coh^{\hGO}(\Gr)$ does not extend to $\Coh^{\hG_{\O,X^J}}(\Gr_{X^J})$ because the $\hG_{\O,X^J}$-orbits on 
$\Gr_{X^J}$ are not all even or all odd dimensional on any fixed connected component. Thus there are no obvious Koszul-perverse 
t-structures on the bigger categories $\calD_{X^J}$ or $\calD_{X^{(J)}}$.
\end{Remark}

The functors $i_m^d: \calD_{X^{(m)},0} \to \calD_{X^{(dm)},0}$ from \eqref{eq:i2} are t-exact for the Koszul-perverse t-structures and identifies 
Koszul-perverse simples on either side. This is a consequence of \cite[Prop. 3.29]{CW2}. 
Taking the colimit over such maps we get a Koszul-perverse t-structure on  $\calD_{X^{(\bullet)}}$.
Since $i(\cF * \cG) \cong i(\cF) \circ i(\cG)$, the product $\circ$ on $\calD_{X^{(\bullet)},0}$ 
is t-exact for the Koszul-perverse t-structure.  
Hence, the heart $\cKP_{X^{(\bullet)}}$ is an abelian, monoidal subcategory such that
every simple in $\cKP_{X^{(\bullet)}}$ is the image of a simple in $\cKP$. 
It is again $\Z^2$-graded by the loop shift $\{1\}$ and the Koszul shift $[1]\la -1 \ra$. 

\begin{Remark}\label{rem:Ktheory}
Since $i_0^d: \cKP \to \cKP_{X^{(d)}}$ identifies simples on either side,
 it follows that the Grothendieck groups of all categories $\cKP_{X^{(d)}}$ are canonically isomorphic. 
\end{Remark}

\subsection{Renormalized $r$-matrices for Coulomb categories}\label{sec:ren2}

The monoidal category $\cKP$ is equipped with renormalized $r$-matrices in the same way as the category $\gmod(R)^\heartsuit$,
see \cite[\S2.2]{KKKO2} and \S\ref{sec:ren} for the terminology. This is a consequence of the general construction of renormalized 
$r$-matrices from factorization structures as explained in \cite[\S 5]{CW1}. 
We now review this in part to fix notation but also because a related idea is used to construct the functor $\F$ in \S\ref{sec:Fdefs}. 

A system of renormalized $r$-matrices consists of integers $\La(\cF,\cG)$ and non-zero maps
\begin{align}\label{r-renor}
\r_{\cF,\cG}: \cF * \cG \to \cG * \cF \{\Lambda(\cF,\cG)\}
\end{align}
for any non-zero $\cF, \cG \in \cKP$. This data is required to satisfy several properties familiar from the theory of quantum groups,
see \cite[Def. 4.1]{CW1}. 

Recall that $\calR_X \cong \calR \times X$. For $\cF \in \cKP$ we denote by 
$$\cF_z = \cF \boxtimes \O_X \in \calD_{X}$$ 
the trivial deformation of $\cF$ where $z$ is the parameter of $X = \Spec \C[z]$.

\begin{Remark}
The sheaf $\cF_z$ was denoted $\widetilde{\cF}$ in \cite{CW1,CW2}. We switch to $\cF_z$ in order to mimic the notation used in \cite{KKKO2} and related literature. 
Note that while $\cF \in \cKP$ the object $\cF_z \in \calD_{X}$ will not lie in $\calD_{X,0}$ and hence cannot belong to $\cKP_X$. 
\end{Remark}

Consider the sheaves in $\calD_{X^2} = \Coh^{\hG_{\O,X^2}}(\calR_{X^2})$ given by
$\cF_{z} \odot \cG_{z}$ and
$s_{12}^*(\cG_{z} \odot \cF_{z}).$
Their restriction  to the diagonal $\Delta \subset X^2$ are $(\cF * \cG)_z$ and 
$(\cG * \cF)_z$ respectively. Their restriction away from $\Delta$ are $(\cF \boxtimes \cG) \times (X^2 \setminus \Delta)$ 
in both cases. This means that the identity generic isomorphism extends to a morphism
\begin{equation}\label{eq:ren}
\bfR_{\cF,\cG}^{\ren} : \cF_{z} \odot \cG_{z} \to s_{12}^*(\cG_{z} \odot \cF_{z}) \{d\}
\end{equation}
which is nonzero over $\Delta$ for some integer $d$. Specializing \eqref{eq:ren} to $0 \in X^2$ yields the map $\r_{\cF,\cG}$ from \eqref{r-renor} where we define $\Lambda(\cF,\cG)= d$. 
These maps define a system of renormalized $r$-matrices.

\section{Duality functors}\label{sec:Fdefs}

Fix an orientation of the edges of $Q_\af$.
This distinguishes a quiver Hecke algebra $R$ as follows.
Let $d_{ij}$ be the number of arrows $i\to j$.
Thus, if $i \ne j \in I_\af$ we have $d_{ij}+d_{ji}=\b(\scrM_i,\scrM_j)$ and we define
$$Q_{ij}(u,v) = (-1)^{d_{ij}}(u-v)^{\b(\scrM_i, \scrM_j)}.$$
The relation \eqref{QQ} holds. Hence we can define $R$ as in \S \ref{subsec:quiverHecke}.
Recall the categories $\gMod(R)$ from \S \ref{subsec:quiverHecke} and $\calD_{X^{(\bullet)}}$ from \S \ref{sec:colimit}. 
We denote by $\calD^-_{X^{(\bullet)}}$ the analogue of $\calD_{X^{(\bullet)}}$ where objects are allowed to be unbounded below. The goal of this section is to explain how to define a monoidal functor
\begin{align}\label{eq:FMod}
\F: (\gMod(R), \circ) \to (\calD^-_{X^{(\bullet)}}, \circ)
\end{align}
analogous to the duality functors studied in \cite{KKK}, \cite{KP}, \cite{KKOP3}. Later we will show that the particular functor we are interested in factors through $\calD_{X^{(\bullet)}}$. 

Following \cite[\S4]{KKOP3}, the functor $\F$ depends on the choice of a duality datum in $\cKP$. This consists of a set of objects $(\scrM_i\,;\,i \in I_\af)$ in $\cKP$ such that
\begin{itemize}[leftmargin=8mm]
\item $\scrM_i$ is simple and real for each $i \in I_\af$,
\item $\Lambda(\scrM_i,\scrM_j) = - \la \alpha_i, \alpha_j \ra$ for each $i \ne j \in I_\af$.
\end{itemize}
Here $\Lambda(-,-)$ is the pairing induced by the $r$-matrices as in \eqref{r-renor}. 
Note that if $i=j$ then $\Lambda(\scrM_i, \scrM_i)=0$ by Proposition \ref{prop:qcom}. We will now assume such a datum is given and explain how to use it to define $\F$ as in \eqref{eq:FMod}.  


\subsection{The maps $\tau$}

To define $\F$ we first consider two objects $\scrM_i,$ $ \scrM_j$ with $i,j\in I_\af$.
Let $J = \{1,2\}$ and $\Pi: \calR_{X^J}/\hG_{\O,X^J} \to \calR_{X^{(J)}}/\hG_{\O,X^{(J)}}$
be the map from \eqref{eq:Pi}. We have the following objects in $ \Coh^{\hG_{\O,X^{(J)}}}(\calR_{X^{(J)}})$
\begin{align*}
(\scrM_i)_{z} \circ (\scrM_j)_{z} &= \Pi_{*}((\scrM_i)_{z} \odot (\scrM_j)_{z}) \\
(\scrM_j)_{z} \circ (\scrM_i)_{z} &= \Pi_{*}((\scrM_j)_{z} \odot (\scrM_i)_{z})
\end{align*}
By \eqref{eq:ren} we have a morphism
\begin{equation}\label{eq:ren2}
\bfR_{\scrM_i,\scrM_j}^{\ren} : (\scrM_i)_{z} \odot (\scrM_j)_{z} \to s_{12}^*((\scrM_j)_{z} \odot (\scrM_i)_{z}) \{\La(\scrM_i,\scrM_j)\}
\end{equation}
Applying $\Pi_{*}$ and using that $\Pi \circ s_{12} = \Pi$ we get a map 
\begin{equation}\label{eq:ren3}
\t: (\scrM_i)_z \circ (\scrM_j)_z \to (\scrM_j)_z \circ (\scrM_i)_z \{\La(\scrM_i,\scrM_j)\}.
\end{equation}
Since $(\scrM_i)_z \odot (\scrM_j)_z $ lies in $ \Coh^{\hG_{\O,X^2}}(\calR_{X^2})$ 
there is a natural action by multiplication of the coordinates $(z_1,z_2) \in X^2$ on it. Let
$$x_1 , x_2: (\scrM_i)_z \circ (\scrM_j)_z \to (\scrM_i)_z \circ (\scrM_j)_z \{2\}$$
denote the pushforwards $x_1 = \Pi_{*}(z_1)$ and $x_2 = \Pi_{*}(z_2)$.
The shift by $\{2\}$ is because $z_1$ and $z_2$ have weight $2$ with respect to the loop $\bbC^\times$. 

\begin{Proposition}\label{prop:t}
The maps $\t$ satisfy the braiding relation.
Further, for any $i,j \in I_\af$ we have
$$\t^2 = c_{ij} (x_1-x_2)^{\b(\scrM_i, \scrM_j)},\quad \t x_1 = x_2 \t $$
for some $c_{ij} \in \bbC^\times$ with $c_{ii}=1$. 
\end{Proposition}
\begin{proof}
If $i=j$ then $\bfR_{\scrM_i,\scrM_i}^{\ren}$ is the identity morphism and so $\t^2=\id$ with $c_{ii}=1$. If $i \ne j$ then
$$\bfR_{\scrM_j,\scrM_i}^{\ren} \circ \bfR_{\scrM_i,\scrM_j}^{\ren}: 
(\scrM_i)_z \odot (\scrM_j)_z \to (\scrM_i)_z \odot (\scrM_j)_z \{2\b(\scrM_i, \scrM_j)\}$$
is the identity map away from the diagonal $\Delta \subset X^2$. 
Thus it must be equal to $c(z_1-z_2)^d$ for some $c \in \bbC^\times$ and $d \in \Z$. 
Degree considerations implies that $d = \b(\scrM_i,\scrM_j)$. 
Applying $\Pi_{*}$ then gives 
$$\t^2 = c(x_1-x_2)^{\b(\scrM_i, \scrM_j)}.$$
Next, since $\bfR_{\scrM_i,\scrM_j}^{\ren}$ is an extension of the identity map, it commutes with the multiplication by $z_1$. 
Since $s_{12}^*(z_1) = z_2$ it follows that after applying $\Pi_{*}$ to \eqref{eq:ren2} we get $\t x_1 = x_2 \t$. 
Finally, for the braiding, consider the composition of maps in $\Coh^{\hG_{\O,X^3}}(\calR_{X^3})$ given by
\begin{align*}
(\scrM_i)_z \odot (\scrM_j)_z \odot (\scrM_k)_z 
& \xrightarrow{\bfR_{\scrM_i,\scrM_j}^{\ren} \odot \id} s_{12}^* ((\scrM_j)_z \odot (\scrM_i)_z \odot (\scrM_k)_z) \{m_{ij}\} \\ 
& \xrightarrow{s_{12}^*(\id \odot \bfR_{\scrM_i,\scrM_k}^{\ren})} s_{12}^* s_{23}^* ((\scrM_j)_z \odot (\scrM_k)_z \odot (\scrM_i)_z) \{m_{ij}+m_{ik}\} \\
& \xrightarrow{s_{12}^* s_{23}^* (\bfR_{\scrM_j,\scrM_k}^{\ren} \odot \id)} s_{12}^* s_{23}^* s_{12}^* ((\scrM_k)_z \odot (\scrM_j)_z \odot (\scrM_i)_z)\{m_{ij}+m_{ik}+m_{jk}\}
\end{align*}
where $m_{ij} = \La(\scrM_i,\scrM_j)$, $m_{jk} = \La(\scrM_j,\scrM_k)$ and $m_{ik} = \La(\scrM_i,\scrM_k)$. 
This map is the extension of the identity on the complement of the big diagonal in $X^3$ multiplied by 
$$(z_1-z_2)^{m_{ij}} \cdot s_{12}^*(z_2-z_3)^{m_{ik}} \cdot s_{12}^*s_{23}^* (z_1-z_2)^{m_{jk}} = (z_1-z_2)^{m_{ij}} (z_1-z_3)^{m_{ik}} (z_2-z_3)^{m_{jk}}.$$

This is the same map as the one obtained from the composition 
\begin{align*}
(\scrM_i)_z \odot (\scrM_j)_z \odot (\scrM_k)_z 
& \xrightarrow{\id \odot \bfR_{\scrM_j,\scrM_k}^{\ren}} s_{23}^* ((\scrM_i)_z \odot (\scrM_k)_z \odot (\scrM_j)_z) \{m_{jk}\} \\             
& \xrightarrow{s_{23}^*(\bfR_{\scrM_i,\scrM_k}^{\ren} \odot \id)} s_{23}^* s_{12}^* ((\scrM_k)_z \odot (\scrM_i)_z \odot (\scrM_j)_z) \{m_{ik}+m_{ik}\} \\
& \xrightarrow{s_{23}^* s_{12}^* (\id \odot \bfR_{\scrM_i,\scrM_j}^{\ren})} s_{23}^* s_{12}^* s_{23}^* ((\scrM_k)_z \odot (\scrM_j)_z \odot (\scrM_i)_z)\{m_{jk}+m_{ik}+m_{ij}\}
\end{align*}
Pushing forward by 
$$\Pi: \calR_{X^3}/\hG_{\O,X^{3}} \to \calR_{X^{(3)}}/\hG_{\O,X^{(3)}}$$ 
recovers the braiding relation for the maps $\t$. 
\end{proof}

We now rescale $\t$ in order to absorb the constants $c_{ij}$ from Proposition \ref{prop:t}. We choose a total ordering $(I_\af, \le)$ of the vertices and let
$$\varphi = \begin{cases} 
(-1)^{d_{ij}}(c_{ij})^{-1} \t & \text{ if $i < j$ } \\
 \t & \text{ if $i \ge j$ } 
\end{cases}$$
For any $i,j \in I_\af$ we have
\begin{align}\label{intertwiner}
\varphi^2 = Q_{ij}(x_1,x_2)+\delta_{ij}
\end{align}

Lastly, recall that $\bfR_{\scrM_i,\scrM_i}^{\ren}$ from \eqref{eq:ren2} restricts to the identity over the diagonal $\Delta \subset X^2$. Equivalently, the map $\bfR_{\scrM_i,\scrM_i}^{\ren} - \id$ restricts to zero.  This means that there exists some $\hat{\bfR}_{\scrM_i,\scrM_i}^{\ren}$ such that $\bfR_{\scrM_i,\scrM_j}^{\ren} = \hat{\bfR}_{\scrM_i,\scrM_i}^{\ren}(z_1-z_2)+\id$.  Then, we define 
$$\tau = \begin{cases}
\Pi_{*}(\hat{\bfR}_{\scrM_i,\scrM_i}^{\ren}) & \text{ if $i=j$ } \\
\varphi & \text{ if $i\ne j$ }
\end{cases}$$ 

\subsection{The functor $\F$}

For each $\beta \in \bbN\Phi_\af^+$ of height $h=\height(\beta)$ and $\nu = (\nu_1,\dots,\nu_h) \in (I_\af)^\beta$ we define
\begin{align}\label{eq:M}
\begin{split}
\scrM(\nu) &= (\scrM_{\nu_1})_{z} \circ (\scrM_{\nu_2})_{z} \circ \cdots \circ (\scrM_{\nu_h})_{z} \in \Coh^{\hG_{\O,X^{(h)}}}(\calR_{X^{(h)}}) \\
\scrM(\beta) &= \bigoplus_{\nu \in (I_\af)^\beta} \scrM(\nu).
\end{split}
\end{align}
Following the notation above, we denote  
\begin{align*}
\varphi_k e(\nu) & = \id^{k-1} \circ \varphi \circ \id^{h-k-1}: \scrM(\nu) \to \scrM(s_k \cdot \nu) \{\La(\scrM_{\nu_k}, \scrM_{\nu_{k+1}})\} \\
\tau_k e(\nu) & = \id^{k-1} \circ \tau \circ \id^{h-k-1}: \scrM(\nu) \to \scrM(s_k \cdot \nu) \{\La(\scrM_{\nu_k}, \scrM_{\nu_{k+1}} - 2\delta_{\nu_k,\nu_{k+1}})\} \\
x_k e(\nu) & = \id^{k-1} \circ x \circ \id^{h-k}: \scrM(\nu) \to \scrM(\nu) \{2\}
\end{align*}
where $e(\nu)$ denotes the idempotent projecting onto $\scrM(\nu)$. 
Here  that in \S\ref{subsec:quiverHecke} the symbols $x_k$, $\tau_k$ and $e(\nu)$ denote the generators of $R(\beta)$,
see \eqref{xte}.
Let $\varphi_k$ be the intertwiner in $R(\beta)$, as defined in \cite[(1.6)]{KKK}.

\begin{Proposition}\label{prop:action}
The maps $e(\nu)$, $\tau_ke(\nu)$ and $x_ke(\nu)$ above define an algebra homomorphism $R(\beta) \to \End(\scrM(\beta))$. Moreover, we 
have $\End^i(\scrM(\beta)) = 0$ for $i < 0$. 
\end{Proposition}

\begin{proof}
The proposition follows from the relations between the elements $x_k$, $\varphi_k$ and $e(\nu)$ in $R(\beta)$,
given in \cite[Lem.~1.5]{KKK}, and the corresponding relations between the operators $x_k$, $\varphi_k$ and $e(\nu)$ in $\End(\scrM(\beta))$ which follow from Proposition \ref{prop:t} and \eqref{intertwiner}. 
\end{proof}

\begin{Remark}
Under this homomorphism, the intertwiner $\varphi_k e(\nu) \in R(\beta)$ in \cite[(1.6)]{KKK} is taken to the operator $\varphi_k e(\nu)$ above. 
\end{Remark}

We define the functor
\begin{align}\label{eq:F}
\F_\beta : \gMod(R(\beta))  \to \calD^-_{X^{(h)}},\quad
\notag N  \mapsto \scrM(\beta) \otimes_{R(\beta)} N
\end{align}
where $h = \height(\beta)$ and denote $\F = \bigoplus_{\beta \in \bbN\Phi_\af^+} \F_\beta$. 
The following result is analogous to \cite[Thm.~4.2]{KKOP3}.

\begin{Proposition}\label{prop:Fmonoidal}
The functor $\F$ is monoidal. 
\end{Proposition}

\begin{proof}
Suppose $M \in \gMod(R(\beta))$ and $M' \in \gMod(R(\beta'))$. Using \eqref{circ} we get 
\begin{align*}
\F_{\beta+\beta'}(M \circ M') 
&\cong \scrM(\beta+\beta') \otimes_{R(\beta+\beta')} R(\beta+\beta') e(\beta,\beta') \otimes_{(R(\beta) \otimes R(\beta'))} (M \otimes M') \\
&\cong \scrM(\beta+\beta') e(\beta,\beta') \otimes_{(R(\beta) \otimes R(\beta'))} (M \otimes M') \\
&\cong (\scrM(\beta) \circ \scrM(\beta')) \otimes_{(R(\beta) \otimes R(\beta'))} (M \otimes M') \\
&\cong (\scrM(\beta) \otimes_{R(\beta)} M) \circ (\scrM(\beta') \otimes_{R(\beta')} M') \\
&\cong \F_{\beta}(M) \circ \F_{\beta'}(M')
\end{align*}
For the third isomorphism, note that
\begin{align*}
\scrM(\beta+\beta')e(\beta,\beta')
&= \bigoplus_{\nu \in (I_\af)^{\beta+\beta'}} \scrM(\nu) e(\beta,\beta')\\
&= \bigoplus_{(\nu,\nu') \in (I_\af)^{\beta}\times(I_\af)^{\beta'}} (\scrM(\nu) \circ \scrM(\nu'))e(\beta,\beta')\\
&= \scrM(\beta) \circ \scrM(\beta')
\end{align*}

\end{proof}

\subsection{The reduced functor $\barF$}\label{sec:bF}

Symmetric functions in the $x_i \in R(\beta)$ generate a central subalgebra of $R(\beta)$. We denote $e_1(x) = \sum_i x_i$ and $\barR(\beta) = R(\beta)/(e_1(x))$. It is easy to check that the composition 
\begin{equation}\label{eq:Rmap2}
R(\beta) \to R(\beta)[z] \to \barR(\beta)[z]
\end{equation}
is an isomorphism, where the first map is from \eqref{eq:Rmap} and the second is projection. We denote 
\begin{align}\label{eq:varpi}
\begin{split}
\varpi: \gMod(\barR(\beta)) \to \gMod(R(\beta)) ,\quad
M \mapsto M[z]
\end{split}
\end{align}
which, by Remark \ref{rem:barR}, is the affinization of $M$. The goal of this section is to discuss an analogous structure on the Coulomb side. While this is mildly interesting on its own the main motivation is to show that the functor $\F$ behaves well with respect to affinizations (Corollary \ref{cor:FLz}). 

We denote by $\overline{X^J} \subset X^J$ the locus of points $(z_1,\dots,z_j) \in X^J$ where $j=|J|$ and $\sum_i z_i=0$. 
Let $\calR_{\overline{X^J}}/\hG_{\O,\overline{X^J}}$ be the base change of $\calR_{X^J}/\hG_{\O,X^J}$ to $\overline{X^J} \subset X^J$. 
Note that we have an isomorphism 
\begin{align}\label{eq:XJ}
\begin{split}
\overline{X^J} \times X \xrightarrow{\sim} X^J ,\quad
(z_1,\dots,z_j,a) \mapsto (z_1+a,z_2+a,\dots,z_j+a).
\end{split}
\end{align}
There is a natural isomorphism between the fibers of $\calR_{\overline{X^J}}/\hG_{\O,\overline{X^J}}$ over $(z_1,\dots,z_j)$ and 
$(z_1+a,\dots,z_j+a)$ induced by the automorphism $z \mapsto z+a$ of $X=\bA^1$. Thus we get an isomorphism 
$$\calR_{X^J}/\hG_{\O,X^J} \cong \calR_{\overline{X^J}}/\hG_{\O,\overline{X^J}} \times X.$$
This is a generalization of the isomorphism $\calR_X \cong \calR \times X$ that came up in \S \ref{sec:ren2}. 

We note that the construction above is equivariant with respect to the action of the symmetric group $\frakS_J$. 
One similarly obtains an isomorphism 
$$\calR_{X^{(J)}}/\hG_{\O,X^{(J)}} \cong \calR_{\overline{X^{(J)}}}/\hG_{\O,\overline{X^{(J)}}} \times X.$$
where $\overline{X^{(J)}} = \overline{X^J}/\frakS_J$ and $\calR_{\overline{X^{(J)}}}/\hG_{\O,\overline{X^{(J)}}}$ is the base change to 
$\overline{X^{(J)}} \subset X^{(J)}$. We denote 
$$\calR_{\overline{X^{(J)}}}/\hG_{\O,\overline{X^{(J)}}} \xrightarrow{\iota} \calR_{X^{(J)}}/\hG_{\O,X^{(J)}} \xrightarrow{\pi} \calR_{\overline{X^{(J)}}}/\hG_{\O,\overline{X^{(J)}}}$$ 
the obvious inclusion and projection and, following our earlier notation, let $\calD_{\overline{X^{(J)}}} = \Coh^{\hG_{\O,\overline{X^{(J)}}}}(\calR_{\overline{X^{(J)}}})$. 

The objects $\scrM(\nu) \in \calD_{X^{(J)}}$ defined in \eqref{eq:M} have analogues $\bscrM(\nu) \in \calD_{\overline{X^{(J)}}}$. Explicitly, $\bscrM(\nu) = \iota^* \scrM(\nu)$. We also have $\bscrM(\beta) = \bigoplus_{\nu \in I^\beta_\af} \bscrM(\nu)$. Since $\iota^*(e_1(x)) = 0$ it follows that the map $R(\beta) \to \End(\scrM(\beta))$ from Proposition \ref{prop:action} factors through $\barR(\beta)$. 
Thus, we get a functor
\begin{align*}
\barF_\beta : \gMod(\barR(\beta))  \to \calD^-_{\overline{X^{(J)}}},\quad
N  \mapsto \bscrM(\beta) \otimes_{\barR(\beta)} N
\end{align*}

\begin{Lemma}\label{lem:barF}
The following squares are commutative
\begin{equation}\label{eq:barF}
\begin{split}
\xymatrix{
\gMod(\barR(\beta)) \ar[r]^-{\barF_\beta} \ar[d]^{\varpi} & \calD^-_{\overline{X^{(J)}}} \ar[d]^{\pi^*} \\
\gMod(R(\beta)) \ar[r]^-{\F_\beta} & \calD^-_{X^{(J)}}} 
\hspace{2cm}
\xymatrix{
\gMod(\barR(\beta)) \ar[r]^-{\barF_\beta} \ar[d]^{\Res} & \calD^-_{\overline{X^{(J)}}} \ar[d]^{\iota_*} \\
\gMod(R(\beta)) \ar[r]^-{\F_\beta} & \calD^-_{X^{(J)}}} 
\end{split}
\end{equation}
\end{Lemma}

\begin{proof}
It is not hard to see that $\scrM(\beta) \cong \pi^* \bscrM(\beta)$. It follows that
\begin{align*}
\pi^* (\barF_\beta(N))
\cong \pi^* (\bscrM(\beta) \otimes_{\barR(\beta)} N) \cong \scrM(\beta) \otimes_{R(\beta)} \varpi(N) \cong \F_\beta(\varpi(N))
\end{align*}
for any $N \in \gMod(\barR(\beta))$. Similarly, using that $\iota^*(\scrM(\beta)) \cong \bscrM(\beta)$, we have 
$$\iota_*(\barF_\beta(N)) \cong \iota_*(\bscrM(\beta) \otimes_{\barR(\beta)} N) \cong \iota_*(\iota^*(\scrM(\beta)) \otimes_{\barR(\beta)} N) \cong \scrM(\beta) \otimes_{\barR(\beta)} \Res(N) \cong \F_\beta(\Res(N)).$$
\end{proof}

\begin{Corollary}\label{cor:FLz}
Suppose $L \in \gmod(R(\beta))$ is simple such that $\scrM = \F(L)$ is a simple module in $\cKP_{X^{(J)}}$ where $|J| = \height(\beta)$.  Then $\F(L_z) \cong \scrM_z$. 
\end{Corollary}
\begin{proof}
If $L$ is simple then $e_1(x)$, which is central in $R(\beta)$, acts by zero. Thus $L$ descends to a simple $\barR(\beta)$-module which we denote $L'$. Then $L_z \cong \varpi(L')$ and we have
$$\F_\beta(L_z) \cong \F_\beta(\varpi(L')) \cong \pi^*(\barF_\beta(L'))$$
where the second isomorphism is from the left square in \eqref{eq:barF}. But, using the right square in \eqref{eq:barF} we have
$$\iota_*(\barF_\beta(L')) \cong \F_\beta(\Res(L')) \cong \scrM$$ 
and so we get $\pi^*(\barF_\beta(L')) \cong \scrM_z$. 
\end{proof}

\section{Coulomb categories of quiver type}\label{sec:quivercoulomb}

In \S \ref{sec:CB} we discussed the Coulomb branch categories $\cKP_{G,N}$ for arbitrary pairs $(G,N)$. We now restrict to the case when $(G,N)$ is associated with a simply laced quiver $Q = (I,E)$ of finite type as in \S \ref{sec:quiverHecke}. The main goal is to find a duality datum in $\cKP = \cKP_{G,N}$ which is ultimately achieved in Proposition \ref{prop:alpha0}.

\subsection{General notation}\label{sec:notation}

Fix finite dimensional $I$-graded vector spaces $V=\bigoplus_{i \in I} V^{(i)}$ and let 
\begin{equation}\label{GN}
G = \prod_{i\in I}GL(V^{(i)}),\quad N = \bigoplus_{i\to j \in E} \Hom(V^{(i)},V^{(j)}).
\end{equation}
Write $a_i = \dim V^{(i)}$.
We identify the dimension vector of $V$ with the weight 
\begin{align}\label{alpha}\alpha=\sum_{i\in I} a_i \alpha_i \in \Q^+.\end{align}
We will also make use of the following quantities 
\begin{align}\label{m}
m_i = m_i^+ +m_i^- ,\quad m_i^+ = - a_i + \sum_{j \to i \in E} a_j ,\quad m_i^- = - a_i + \sum_{i \to j \in E} a_j.
\end{align}

As before we abbreviate $\Gr_G = \Gr$, $\calR_{G,N} = \calR$, etc.  We also denote the associated Coulomb branch $\calM_\alpha$ if we want to emphasize the dimension vector. 
The sets of $\bbC$-points of the ind-schemes $\calR^\cl$, $\calT$ and $\Gr \times N_\O$ have the following explicit description
\begin{align}\label{i1i2}
\begin{split}
\calR^\cl&=\{(L,x)\in\Gr\times N_\calK\,;\,x(L)\subset L\,,\,x(L_0)\subset L_0\},\\
\calT&=\{(L,x)\in\Gr\times N_\calK\,;\,x(L)\subset L\},\\
\Gr\times N_\calO&=\{(L,x)\in\Gr\times N_\calK\,;\,x(L_0)\subset L_0\}
\end{split}
\end{align}
where $[L_0] \in \Gr_G$ is the standard lattice. The maps $i_1$, $i_2$ in \eqref{def-R} are both given by $(L,x)\mapsto (L,x)$.

\subsection{The twisted product} \label{sec:twisted product}

The connected components of $\Gr_{GL(V^{(i)})}$ are indexed by $\Z$, with the orbit $\Gr_{\lambda^\vee_i}$ labeled by the dominant 
cocharacter $\lambda^\vee_i$ of $GL(V^{(i)})$ belonging to the component labelled by $\la \lambda^\vee_i, \omega_{i,a_i} \ra$.  Thus the 
connected 
components of $\Gr$ and $\calR$ are indexed by $\Z I$ which we identify with $\Z\Phi$, with the orbit $\Gr_\lambda^\vee$ belonging to the 
component labelled by $\sum_i \la \lambda^\vee_i, \omega_{i,a_i} \ra\, \alpha_i$.  Here we view a cocharacter of $G$ has a tuple 
$\lambda^\vee=(\lambda^\vee_i)$ of cocharacters of the $GL(V^{(i)})$'s.
The diagonal $\bbC^\times \subset G$ acts trivially on $\calR$ and thus gives a further decomposition of $\cKP$ indexed by $\Z$. Together this 
gives us a decomposition of the categories $\Coh^{\hGO}(\calR)$ and $\cKP$ indexed by $\Z\Phi_\af$.  For $\beta \in \Z\Phi_\af$ we denote the 
corresponding block of $\cKP$ by $\cKP_\beta$ yielding the decomposition 
\begin{align}\label{KPbeta}
\cKP=\bigoplus_\beta\cKP_\beta.
\end{align}

Because of this decomposition one can combinatorically tweak the product $*$ by loop shifts.
If $\cF \in \cKP_\beta$ and $\cF' \in \cKP_{\beta'}$ for $\beta,\beta' \in \Z \Phi_\af$ we define 
\begin{equation}\label{eq:twist}
\cF \star \cF' = \cF * \cF' \{ - \sum_{e \in E} \la \beta, \omega_{t(e),1} \ra \cdot \la \beta', \omega_{h(e),1} \ra\}
\end{equation}
where $t(e), h(e) \in I$ denote the tail and head of $e \in E$. Such twists are often considered, for example, in the context of Hall algebras. For 
our 
purposes the main effect of this twist is to tweak the $\La(-,-)$ pairing. The twist in \eqref{eq:twist} is necessary in order to match up $\cKP$ with 
the 
grading on the quiver Hecke algebra side. The definition of $\cKP$ depends on the orientation of the quiver $Q$ while the quiver Hecke algebra 
does not and this twist makes up for this asymmetry. From hereon we will use this twisted product which we will still denote $*$ in order to 
simplify 
notation. 

There are similar decompositions for the categories $\cKP_{X^n},$ $ \cKP_{X^{(n)}}$ in \S\ref{sec:KP}.
Thus, we can twist the product as in \eqref{eq:twist}. From hereon we will use these twisted products.

\subsection{Some natural objects}

The $G_\calO$-orbits $\Gr_{\lambda^\vee}$ in the affine Grassmannian $\Gr_G$ are indexed by dominant cocharacters $\lambda^\vee$ of $G$. 
Let $\Gr_{\le\lambda^\vee}$ be the closure of such an orbit. 
For $k\in\bbZ_{>0}$, let $\omega^\vee_{i,k}$ be the $k$th fundamental cocharacter of the subgroup $GL(V^{(i)}) \subset G$. 
We abbreviate 
$$\Gr_{ke_i}=\Gr_{\omega^\vee_{i,k}}\subset \Gr_{GL(V^{(i)})}$$
where $e_i\in\N I$ is the Dirac function at $i$.  If $L_0^{(i)} \in \Gr_{GL(V^{(i)})}$ denotes the standard lattice in $(V^{(i)})_\calK$, then we have 
\begin{align}\label{Gr}\Gr_{ke_i} \cong \{tL_0^{(i)} \subset L^{(i)} \subset L_0^{(i)}\,;\,\dim(L_0^{(i)}/L^{(i)}) = k\}.\end{align}
More generally, for a tuple $\uk = (k_i)$ of positive integers we set
$$\Gr_\uk=\prod_{i\in I}\Gr_{k_ie_i} \subset \Gr_G = \prod_{i\in I} \Gr_{GL(V^{(i)})}.$$
Thus $\Gr_\uk$ is the $G_\O$-orbit corresponding to the cocharacter $\sum_{i\in I} \omega^\vee_{i,k_i}$.
Replacing $\omega^\vee_{i,k}$ with the cocharacter 
$$\omega^\vee_{i,-k}=(\omega^\vee_{i,k})^*=- w_0 (\omega^\vee_{i,k})$$ we get a similar subset 
$$\Gr_{-ke_i}= \{ L_0^{(i)} \subset L ^{(i)}\subset t^{-1}L_0^{(i)}\,;\,\dim(L^{(i)}/L_0^{(i)}) = k\}\subset\Gr_{GL(V^{(i)})}$$
This allows us to define $\Gr_\uk$ for any sequence $\uk$ with $-a_i \le k_i \le a_i$.
Using the projection map $\calR \to \Gr$ we define the base changes 
$$\calR_{\le \lambda^\vee} = \calR \times_{\Gr} \Gr_{\le\lambda^\vee}
,\quad
\calR_{\uk} = \calR \times_{\Gr} \Gr_\uk.$$
Let $i_{\le \lambda^\vee}$ and $ i_\uk$ be the closed embeddings of the corresponding classical loci 
$$i_{\le \lambda^\vee}: \calR^\cl_{\le \lambda^\vee} \to \calR , \quad i_\uk: \calR^\cl_{\uk} \to \calR.$$

For tuples $\uk$, $\ul \in \bbZ I$ with $-a_i \le k_i \le a_i$ we 
consider the following objects in the category $\Coh^{\hGO}(\calR)$ 
\begin{align}\label{eq:cP}
\begin{split}
\cP_{\uk,\ul} = i_{\uk*} \Big(\O_{\calR_\uk^\cl} \bigotimes_{i \in I} \scrL_i^{\ell_i}\Big) [\frac12\dim \Gr_\uk]
\end{split}
\end{align}
where  
\begin{align*} 
\begin{split}
\scrL_i = \begin{cases} \det((L_0^{(i)}/L^{(i)}) \{-1\}) & \text{ if $k_i \ge 0$ } \\ 
\det((L^{(i)}/L_0^{(i)}) \{-1\}) & \text{ if $k_i \le 0$. } \end{cases}
\end{split}
\end{align*}
In particular, we have $\cP_{ke_i,\ul} \cong \cP_{ke_i, \ell_i e_i}$ for each $i \in I$ and $k \in \bbZ$, 
because in this case the line bundle
$\scrL_j$ is trivial for  $j \ne i$. 
We abbreviate 
$$\cP_{ke_i,\ell} = \cP_{ke_i, \ell e_i}
,\quad
\cP_{\uk} =\cP_{\uk,0}.$$ 
Finally, we define the following objects 
\begin{align}
\label{eq:V-phi-psi}
\begin{split}
\cV_i &= (i_\u0)_* \Big(\O_{\calR_\u0^\cl} \otimes (L^{(i)}_0/tL^{(i)}_0) \{-1\}  \Big), \\
\phi_i &= (i_\u0)_* \Big(\O_{\calR_\u0^\cl} \otimes \det((L^{(i)}_0/tL^{(i)}_0) [1]\la-1\ra)\Big), \\
\psi_i^- &= \phi_i *\bigast_{i \to j \in E} \phi_j^{-1} ,\\ 
\psi_i^+ &= \phi_i *\bigast_{j \to i \in E} \phi_j^{-1} ,\\
\psi_i &= \psi_i^- * \psi_i^+.
\end{split}
\end{align}
Note that  $\phi_i$ and $\phi_j$ commute and are invertible for the convolution product in $\Coh^\hGO(\calR)$. 
We abbreviate
$$\phi_i ^k=(\phi_i )^{*k} \cong (i_\u0)_* \Big(\O_{\calR_\u0^\cl} \otimes \det((L^{(i)}_0/tL^{(i)}_0) [1]\la -1 \ra)^k\Big)
,\quad
k\in\bbZ.$$

\begin{Remark} Some miscelaneous remarks. 
\begin{enumerate}[label=$\mathrm{(\alph*)}$,leftmargin=8mm]
\item In \eqref{eq:cP} the fractional shift $[\frac12]$ is interpreted formally. This is possible to do because in any component of $\Gr_G$ the $\hGO$-
orbits have either all even or odd dimension. As a result the $[-]$ shifts one obtains are either all integer or all half-integer in any block of $\cKP$. 
\item The object $\cP_{\u0,\ul}=(i_\u0)_*(\O_{\calR_\u0^\cl})$ is the monoidal unit in $\cKP$ and $\Coh^\hGO(\calR)$. We will denote it $1$. 
\item Let $\cKP^\Gr$ denote the Koszul-perverse heart in $\Coh^\hGO(\Gr_G)$. One can define analogous objects $\cP_{\uk,\ul}^{\Gr},$ 
$ \cV_i^{\Gr}$, etc., in $\cKP^\Gr$ by applying the definitions to a quiver with no edges. 
\item It will often be convenient to ignore loop shifts in various calculations. 
We write $\cF \dcong \cF'$ if $\cF \cong \cF' \{s\}$ for some $s \in \Z$. 
Note that if $\cF \dcong \cF'$ then $\La(\cF,\cG) \cong \La(\cF',\cG)$ for any $\cG$. 
On some occasions we will also want to ignore the Koszul shift. 
We write $\cF \ddcong \cF'$ if $\cF \cong \cF' \{s\} [s'] \la -s' \ra$ for some $s,s' \in \Z$. 
\end{enumerate}
\end{Remark}

\subsection{Simple objects in $\cKP$}\label{sec:simples}

We now recall some generalities from \cite{CW2} about the simple objects in $\cKP$. 
Following Theorem 3.31 of \cite{CW2} the simple objects in $\cKP$ are in bijection with those in $\cKP^\Gr$. 
To describe this bijection recall that for any cocharacter $\lambda^\vee$ the map $\pi: \calR^\cl_{\lambda^\vee} \to \Gr_{\lambda^\vee}$ is a 
vector bundle. Given any simple $\cF^\Gr$ in $\cKP^\Gr$ supported on $\Gr_{\le \lambda^\vee}$ the pullback 
$\pi^*(\cF^\Gr|_{\Gr_{\lambda^\vee}})$ has a unique extension to a simple Koszul-perverse coherent sheaf $\cF$ supported on 
$\calR_{\le \lambda^\vee}$.  The bijection is then $\cF^\Gr \mapsto \cF$. 
The extension $\cF$ may be very complicated, because the map $\pi:\calR^\cl_{\le \lambda^\vee} \to \Gr_{\le \lambda^\vee}$ may be not flat. 
The following result illustrates that this is indeed the source for this complication.

\begin{Lemma}\label{lem:extension}
Suppose the map 
$\pi: \calR^\cl_{\le \lambda^\vee} \to \Gr_{\le \lambda^\vee}$ is a bundle, and $\cF^\Gr$ a simple object in $\cKP^\Gr$ supported on 
$\Gr_{\le \lambda^\vee}$. Then $\pi^*(\cF^\Gr) \in \cKP$ is simple. 
\end{Lemma}

\begin{proof}
We first check that $\pi^*(\cF^\Gr)$ is Koszul-perverse. To do that, we consider the obvious inclusions
$$\xymatrix{\Gr_{\le \lambda^\vee}\ar[r]^s\ar@/_1.5pc/[rr]_-\sigma
&\ar@/_1pc/[l]_-\pi\calR^\cl_{\le \lambda^\vee}  \ar[r]^{i_2} &\calT_{\le \lambda^\vee} }.$$
By definition of $\cKP$ in \cite[Thm.~3.1]{CW2}, the object $\pi^*(\cF^\Gr)$ is Koszul-perverse if and only if 
$\sigma^* (i_2)_* \pi^*(\cF^\Gr)$ is Koszul-perverse. 
The map $i_2$ is an inclusion of bundles of finite codimension.  We have the isomorphism 
$$\sigma^* (i_2)_*\pi^*(\cF^\Gr) \cong s^* (i_2)^* (i_2)_* \pi^*(\cF^\Gr).$$ 
Since $\pi \circ s$ is the identity, the object $s^* (i_2)^* (i_2)_* \pi^*(\cF^\Gr)$ has a filtration with sub-quotients isomorphic to 
$\cF^\Gr \otimes \cV[k] \la -k \ra$ for bundles $\cV$ which are isomorphic to exterior products of the conormal bundle 
of the inclusion $i_2$. Since Koszul-perverse sheaves remain Koszul-perverse after tensoring with any $\cV[k] \la -k \ra$ 
it follows that $s^* (i_2)^* (i_2)_* \pi^*(\cF^\Gr)$ is 
Koszul-perverse. Finally, note that the functor $s^*$ is conservative on (scaling equivariant) coherent sheaves and t-exact for the 
Koszul-perverse t-structure \cite[Lem. 3.17]{CW2}. Since $s^*\pi^*(\cF^\Gr) \cong \cF^\Gr$ is simple, it follows that $\pi^*(\cF^\Gr)$ 
must also be simple. 
\end{proof}

Since $\Gr_{\le \lambda^\vee} = \Gr_{\lambda^\vee}$ if $\lambda^\vee$ is minuscule, we deduce the following.

\begin{Corollary} \label{cor:simple}
The objects $\cP_{\uk,\ul},$ $ \cV_i$ and $\phi_i$  in $\cKP$ are simple.  
\qed
\end{Corollary}

Simple Koszul-perverse sheaves in $\cKP$ are indexed, up to shifts, by pairs $(\lambda^\vee,\mu)$ of a coweight and weight of $G$ up  to the diagonal action of the Weyl group.  
We denote these simples  $\cP_{\lambda^\vee,\mu}$ with $\lambda^\vee$ dominant.
Given tuples $\uk,\ul\in\bbZ I$ with $-a_i\leqslant k_i\leqslant a_i$, one can check that 
$$\cP_{\uk,\ul} \cong \cP_{\lambda^\vee,\mu} \iff \lambda^\vee = \sum_i \omega^\vee_{i,k_i} ,\quad \mu = \sum_i \ell_i \omega_{i,k_i}.$$

Recall the decomposition $\cKP=\bigoplus_\beta\cKP_\beta$ in \eqref{KPbeta} where $\beta$ runs over $\bbZ\Phi_\af$. We have
$$\cP_{\lambda^\vee, \mu} \in\cKP_\beta \iff \beta =
\sum_i \la \lambda^\vee_i, \omega_{i,a_i} \ra\, \alpha_i + \sum_i \la \mu_i, \omega^\vee_{i,a_i} \ra\, \delta.$$

\subsection{Real objects in $\cKP$}\label{sec:real}

In the theory of renormalized $r$-matrices, real simple objects play an important role. In this section we prove that the simple objects 
$\cP_{\pm e_i, \ell}$ in $\cKP_{\pm \alpha_i + \ell \delta}$ are real (Corollary \ref{cor:real}). Let us first recall how to compute a product of the 
form $\cP_{k_1e_i} * \cP_{k_2e_i}$. Following the notation from \cite[\S 5.2]{CW2} consider the following diagram
\begin{equation}\label{eq:conv1}
\begin{split}
\xymatrix{
& \calR_{k_1e_i}^\cl \ttimes \calR_{k_2e_i}^\cl \ar[d]^{\cl} & \ar[l]_-{d'} \calS' \ar[d]^{\cl'} & \\
\calR_{k_1e_i} \ar[d]^{i_2} & \calR_{k_1e_i} \ttimes \calR_{k_2e_i} \ar[l] \ar[d]^{i_2 \ttimes \id} & \calS \ar[d] \ar[l]_-d \ar[r]^-m & \calR_{\le \omega^\vee_{i,k_1} + \omega^\vee_{i,k_2}} \ar[d] \\ 
\calT_{k_1e_i} & \calT_{k_1e_i} \ttimes \calR_{k_2e_i} \ar[l] & \Gr_{k_1e_i} \ttimes \calR_{k_2e_i} \ar[l]_{\td} \ar[r]^{\tm} & \calT_{\le \omega^\vee_{i,k_1} + \omega^\vee_{i,k_2}}
}
\end{split}
\end{equation}
where the top square is the fibre product and the bottom squares form the standard diagram used to compute convolution.
In particular, the bottom middle and right squares are Cartesian. Up to shifts, which are irrelevant for the current discussion, we have
\begin{align}\label{P*P}
\cP_{k_1e_i} * \cP_{k_2e_i}\cong m_* d^* \cl_* (\O_{\calR_{k_1e_i}^\cl \ttimes \calR_{k_2e_i}^\cl})
\cong m_* \cl'_* {d'}^*  (\O_{\calR_{k_1e_i}^\cl \ttimes \calR_{k_2e_i}^\cl})
\cong m_* \cl'_* \O_{\calS'}
\end{align}
where the first isomorphism is by base change.

\begin{Lemma}\label{lem:S}
We have the following Cartesian square where the map $\pi$ is a vector bundle
\begin{equation}\label{eq:local2}
\begin{split}
\xymatrix{
\calS' \ar[r]^-{m_2} \ar[d]^{\pi'} & \calR^\cl_{\le \omega^\vee_{i,k_1} + \omega^\vee_{i,k_2}} \ar[d]^{\pi} \\
\Gr_{k_1e_i} \ttimes \Gr_{k_2e_i} \ar[r]^{m_1} & \Gr_{\le \omega^\vee_{i,k_1} + \omega^\vee_{i,k_2}}
}
\end{split}
\end{equation}
\end{Lemma}

\begin{proof}
We can extend the middle two squares of (\ref{eq:conv1}) as follows
$$\xymatrix{
\calR_{k_1e_i}^\cl  \ar[d] & \calR_{k_1e_i}^\cl \ttimes \calR_{k_2e_i}^\cl  \ar[l]  \ar[d]^{\cl} & \ar[l]_-{d'} \calS' \ar[d] \\
\calT_{k_1e_i} & \calT_{k_1e_i} \ttimes \calR_{k_2e_i}^\cl \ar[l] \ar[d]  & \Gr_{k_1e_i} \ttimes \calR_{k_2e_i}^\cl  \ar[d] \ar[l] \\
 & \calT_{k_1e_i} \ttimes \calR_{k_2e_i} & \Gr_{k_1e_i} \ttimes \calR_{k_2e_i} \ar[l]_{\td}
}$$
Whence, we have
\begin{align}\label{eq:local4}
\begin{split}
\calS' \cong \calR_{k_1e_i}^\cl \times_{\calT_{k_1e_i}} (\Gr_{k_1e_i} \ttimes \calR_{k_2e_i}^\cl)
\cong (\calR_{k_1e_i}^\cl \ttimes \Gr_{k_2e_i}) \times_{(\calT_{k_1e_i} \ttimes \Gr_{k_2e_i})} (\Gr_{k_1e_i} \ttimes \calR_{k_2e_i}^\cl).
\end{split}
\end{align}
All three spaces in the fiber product above are vector bundles over $\Gr_{k_1e_i} \ttimes \Gr_{k_2e_i}$ with the following lattice descriptions
\begin{itemize}[leftmargin=8mm]
\item The bundle $\calT_{k_1e_i} \ttimes \Gr_{k_2e_i}$ is described by
$$\{L_2^{(i)} \overset{k_2}{\subset} L_1^{(i)} \overset{k_1}{\subset} L_0^{(i)}
\,;\,
xL_1^{(i)} \subset L_0^{(j)} \text{ if $i \to j$ }
\,,\,
xL_0^{(j)} \subset L_1^{(i)} \text{ if $j \to i$ }
\,,\, tL_0^{(i)} \subset L_1^{(i)}
\,,\, tL_1^{(i)} \subset L_2^{(i)} \}$$
where the superscripts denote the codimension. In particular we have an isomorphism
$$\calT_{k_1e_i} \ttimes \Gr_{k_2e_i} \cong \calT' \times_{(\Gr_{k_1e_i} \ttimes \Gr_{k_2e_i})} \calT''$$
where $\calT'$ and $\calT''$ are bundles over $\Gr_{k_1e_i} \ttimes \Gr_{k_2e_i}$ corresponding to the space of maps $x$ for arrows $i \to j$ and $j \to i$ respectively.
\item The factor $\calR_{k_1e_i}^\cl \ttimes \Gr_{k_2e_i}$ in (\ref{eq:local4}) has a similar lattice description as $\calT_{k_1e_i} \ttimes \Gr_{k_2e_i}$ with the only difference being the additional conditions that $xL_0^{(i)} \subset L_0^{(j)}$ if $i \to j$. This additional restriction is on the $\calT'$ factor.
\item The factor $\Gr_{e_i} \ttimes \calR_{e_i}^\cl$ in (\ref{eq:local4}) carries the additional condition that $xL_0^{(j)} \subset L_2^{(i)}$ if $j \to i$. This restriction is entirely on the $\calT''$ factor.
\end{itemize}

All this means that the fiber product in (\ref{eq:local4}) is classical (non-derived). Thus $\calS'$ is a vector bundle over $\Gr_{k_1e_i} \ttimes \Gr_{k_2e_i}$ that can be described as
$$\{L_2^{(i)} \overset{k_2}{\subset} L_1^{(i)} \overset{k_1}{\subset} L_0^{(i)}\,;\, xL_0^{(i)} \subset L_0^{(j)} \text{ if $i \to j$ }
\,,\, xL_0^{(j)} \subset L_2^{(i)} \text{ if $j \to i$ }
\,,\, tL_0^{(i)} \subset L_1^{(i)}
\,,\, tL_1^{(i)} \subset L_2^{(i)} \}$$
where $L_2^{(j)}=L_1^{(j)}=L_0^{(j)}$ if $j\neq i$. The lemma follows from this description.
\end{proof}

\begin{Corollary}\label{cor:composition}
For any $0 \le k_1,k_2 \le a_i$ and $\ell_1, \ell_2 \in \Z$ we have
\begin{align*}
\cP_{k_1e_i,\ell_1} * \cP_{k_2e_i, \ell_2} \cong \pi^* (\cP^{\Gr}_{k_1e_i,\ell_1} * \cP^{\Gr}_{k_2e_i, \ell_2})
,\quad
\cP_{-k_1e_i,\ell_1} * \cP_{-k_2e_i, \ell_2} \cong \pi^* (\cP^{\Gr}_{-k_1e_i,\ell_1} * \cP^{\Gr}_{-k_2e_i, \ell_2})
\end{align*}
where the map $\pi$ is as in Lemma $\ref{lem:S}$.
\end{Corollary}

\begin{proof}
To prove the first isomorphism we follow the argument and notation used to compute $\cP_{k_1e_i} * \cP_{k_2e_i}$ in \eqref{P*P}.
Up to shifts, we have
\begin{align*}
\cP_{k_1e_i,\ell_1} * \cP_{k_2e_i, \ell_2} & \cong m_{2*} \pi^{\prime *}(\det(L_0^{(i)}/L_1^{(i)})^{\ell_1} \otimes \det(L_1^{(i)}/L_2^{(i)})^{\ell_2}) \\
&\cong \pi^* m_{1*}(\det(L_0^{(i)}/L_1^{(i)})^{\ell_1} \otimes \det(L_1^{(i)}/L_2^{(i)})^{\ell_2})
\end{align*}
The second isomorphism is by base change using (\ref{eq:local2}). This, up to the very same shift, is isomorphic to $\pi^* (\cP^{\Gr}_{k_1e_i,\ell_1} * \cP^{\Gr}_{k_2e_i, \ell_2})$. The second isomorphism is proved similarly.
\end{proof}

\begin{Corollary}\label{cor:real}
For any $0 \le k \le a_i$ and $\ell \in \Z$ the simple objects $\cP_{\pm k e_i, \ell} \in \cKP_{\pm k \alpha_i + k \ell \delta}$ are real.
\end{Corollary}

\begin{proof}
Up to shift, the proof of \cite[Prop.~2.9]{CW1} yields
$$\cP^{\Gr}_{ke_i,\ell} * \cP^{\Gr}_{ke_i,\ell}=\O_{\Gr_{\le 2 \omega_{i,k}^\vee}} \otimes \det(L_0^{(i)}/L^{(i)})^{\ell}$$
Hence $\cP^{\Gr}_{ke_i,\ell} * \cP^{\Gr}_{ke_i,\ell}$ is simple in $\cKP^\Gr$.
By Corollary \ref{cor:composition} we have
$$\cP_{ke_i,\ell} * \cP_{ke_i, \ell} \cong \pi^* (\cP^\Gr_{ke_i,\ell} * \cP^{\Gr}_{ke_i,\ell}).$$
Since, by Lemma \ref{lem:S}, the map
$\pi: \calR^\cl_{\le 2\omega_{i,k}^\vee} \to \Gr_{\le 2\omega_{i,k}^\vee}$ is a vector, bundle it follows from Lemma \ref{lem:extension} that $\pi^* (\cP^{\Gr}_{ke_i,\ell} * \cP^{\Gr}_{ke_i,\ell})$ is also simple.  Thus $\cP_{ke_i, \ell}$ is real. The proof that $\cP_{-ke_i,\ell}$ is real is similar.
\end{proof}

\subsection{Some relations in $\cKP$}

\begin{Lemma}\label{lem:GrG}
If $\ell \ge \ell'$ there is an object $\cQ^{\Gr} \in \cKP^\Gr$
with short exact sequences 
\begin{align}
\label{eq:ses3} &0 \to \cQ^{\Gr}  \{1\} \to  \cP^{\Gr}_{e_i,\ell+1} * \cP^{\Gr}_{e_i,\ell'} \to \cP^{\Gr}_{e_i,\ell} * \cP^{\Gr}_{e_i,\ell'+1} \{2\} \to 0 \\
\label{eq:ses4} &0 \to \cP^{\Gr}_{e_i,\ell'+1} * \cP^{\Gr}_{e_i,\ell} \{-2\} \to  \cP^{\Gr}_{e_i,\ell'} * \cP^{\Gr}_{e_i,\ell+1} \to \cQ^{\Gr} \{-1\} \to 0 
\end{align}
If $\ell > \ell'$ then $\cQ^{\Gr}$ is simple. If $\ell = \ell'$ then $\cQ^{\Gr} = 0$. 
\end{Lemma}
\begin{proof}
If $\ell = \ell'$ one recovers this result from \cite[Prop.~2.10]{CW1} by taking $k_1=k_2=1$. If $\ell - \ell'=1$ then one recovers this result 
from \cite[Prop.~2.17]{CW1} by taking $k=1$. If $\ell-\ell'>1$ then the argument from \cite[Prop.~2.17]{CW1} still applies near verbatim 
(by taking $k=1$ again). We leave the details to the reader. 
\end{proof}

\begin{Corollary}\label{cor:GrG}
If $\ell \ge \ell'$ there is an object $\cQ \in \cKP$
with short exact sequences 
\begin{align}
\label{eq:ses5}& 0 \to \cQ  \{1\} \to  \cP_{e_i,\ell+1} * \cP_{e_i,\ell'} \to \cP_{e_i,\ell} * \cP_{e_i,\ell'+1} \{2\} \to 0 \\
\label{eq:ses6} &0 \to \cP_{e_i,\ell'+1} * \cP_{e_i,\ell} \{-2\} \to  \cP_{e_i,\ell'} * \cP_{e_i,\ell+1} \to \cQ \{-1\} \to 0 
\end{align}
If $\ell > \ell'$ then $\cQ$ is simple. If $\ell = \ell'$ then $\cQ = 0$. 
\end{Corollary}

\begin{proof}
Sequences (\ref{eq:ses5}) and (\ref{eq:ses6}) are obtained from (\ref{eq:ses3}) and (\ref{eq:ses4}) by applying the functor $\pi^*$ and setting $\cQ = \pi^*(\cQ^{\Gr})$.  One uses Corollary \ref{cor:composition} to prove that 
$$\cP_{e_i,\ell+1} * \cP_{e_i,\ell'} \cong \pi^*(\cP^{\Gr}_{e_i,\ell+1} * \cP^{\Gr}_{e_i,\ell'}),$$
and Lemmas \ref{lem:extension}, \ref{lem:S} to conclude that the object $\cQ $ is simple. 
\end{proof}

\begin{Lemma}\label{lem:map1}
If $i \to j$ in $E$ and $\ell,\ell' \in \Z$, there are exact sequences
\begin{align}
\label{eq:ses7} & 0 \to \cP_{e_i+e_j, \ell e_i + \ell' e_j} \{-1\} \to \cP_{e_j,\ell'} * \cP_{e_i,\ell} \to \cP_{e_i+e_j,(\ell+1)e_i+(\ell'-1)e_j} [1]\la-1\ra \to 0 \\
\label{eq:ses8} & 0 \to \cP_{e_i+e_j, (\ell+1)e_i+(\ell'-1)e_j} [1]\la-1\ra \{-1\} \to \cP_{e_i, \ell} * \cP_{e_j, \ell'} \to \cP_{e_i+e_j, \ell e_i + \ell' e_j} \to 0 
\end{align}
\end{Lemma}
\begin{proof}
To simplify notation we construct these sequences when $\ell=\ell'=0$ (the general case is very similar). To obtain (\ref{eq:ses7}) consider the following diagram similar to the one from (\ref{eq:conv1})
\begin{equation}\label{eq:conv2}
\begin{split}
\xymatrix{
& \calR_{e_j}^\cl \ttimes \calR_{e_i}^\cl \ar[d]^{\cl} & \ar[l]_-{d'} \calS' \ar[d]^{\cl'} & \\
\calR_{e_j} \ar[d]^{i_2} & \calR_{e_j} \ttimes \calR_{e_i} \ar[l] \ar[d]^{i_2 \ttimes \id} & \calS \ar[d] \ar[l]_-{d} \ar[r]^-{m} & \calR_{e_i+e_j} \ar[d] \\
\calT_{e_j} & \calT_{e_j} \ttimes \calR_{e_i} \ar[l] & \Gr_{e_j} \ttimes \calR_{e_i} \ar[l]_{\td} \ar[r]^{\tm} & \calT_{e_i+e_j} 
}
\end{split}
\end{equation}
Arguing as in \eqref{P*P}, we get
$$\cP_{e_j} * \cP_{e_i} \cong m_* \cl'_* (\O_{\calS'}) [\frac12(a_i+a_j-2)].$$
As in the proof of Lemma \ref{lem:S}, we have
$$\calS' \cong (\calR_{e_j}^{\cl} \ttimes \Gr_{e_i}) \times_{(\calT_{e_j} \ttimes \Gr_{e_i})} (\Gr_{e_j} \ttimes \calR_{e_i}^\cl)$$
and consequently 
\begin{align}\label{eq:S'}
\begin{split}
\calS' \cong \{(L^{(i)} \overset{1}{\subset} L_0^{(i)} \,,\, L^{(j)} \overset{1}{\subset} L_0^{(j)} \,,\,x) \,;\,
& xL_0^{(k)} \subset L_0^{(k')} \text{ if $k \to k'$ }, xL^{(i)} \subset L^{(j)}, \\
& xL_0^{(k)} \subset L^{(i)} \text{ if $k \to i$ }, \\
& xL_0^{(k)} \subset L^{(j)} \text{ if $k \to j$} \}
\end{split}
\end{align}
On the other hand, $\calR^\cl_{e_i+e_j}$ is defined by the same conditions as $\calS'$ from \eqref{eq:S'} but {\it without} the condition $xL_0^{(i)} \subset L^{(j)}$. Thus the map $m \circ \cl': \calS' \to \calR_{e_i+e_j}$ factors as in immersion  through $\calR_{e_i+e_j}^\cl$, 
with image the subset
$$\{(L,x)\in\calR_{e_i+e_j}^\cl\,;\,xL_0^{(i)} \subset L^{(j)}\}.$$ 
This locus has codimension one, and is carved out by the vanishing of the map
$$x: L_0^{(i)}/L^{(i)} \to (L_0^{(j)}/L^{(j)}) \la 1 \ra \{-1\}.$$
Here the degrees $\la 1 \ra$ and $ \{-1\}$ come from the fact that the scaling and loop actions on $N$ have the weights 1 and -1
according to the conventions in \S\ref{sec:CB}.
Thus we obtain the following short exact sequence 
$$0 \to \O_{\calR^\cl_{e_i+e_j}} \otimes (L_0^{(i)}/L^{(i)}) \otimes (L_0^{(j)}/L^{(j)})^\vee \la -1 \ra \{1\} \to \O_{\calR^\cl_{e_i+e_j}} \to 
m_* \cl'_* (\O_{\calS'}) \to 0$$
which we can rewrite as a triangle
$$\O_{\calR^\cl_{e_i+e_j}} \to m_* \cl'_* (\O_{\calS'}) \to \O_{\calR^\cl_{e_i+e_j}} \otimes (L_0^{(i)}/L^{(i)}) \otimes (L_0^{(j)}/L^{(j)})^\vee [1] \la -1 \ra \{1\}.$$
Shifting by the appropriate amount recovers (\ref{eq:ses7}). 

To obtain (\ref{eq:ses8}) we have a similar diagram as in (\ref{eq:conv2}) except with the roles of $i$ and $j$ exchanged. We have
$$\calS' \cong (\calR_{e_i}^{\cl} \ttimes \Gr_{e_j}) \times_{(\calT_{e_i} \ttimes \Gr_{e_j})} (\Gr_{e_i} \ttimes \calR_{e_j}^\cl)$$
but this time there is excess intersection. 
To see this denote 
$$X_1 = \calR_{e_i}^{\cl} \ttimes \Gr_{e_j}
,\quad
X_2 = \Gr_{e_i} \ttimes \calR_{e_j}^\cl
,\quad
Y = \calT_{e_i} \ttimes \Gr_{e_j}.$$
As bundles over $\Gr_{e_i} \times \Gr_{e_j}$ these have the following descriptions
$$
X_1 = \left\{
\vcenter{\xymatrix{
L^{(i)} \ar@{}[r]|-{\subset} & L_0^{(i)} \ar[d]^x \\
L^{(j)} \ar@{}[r]|-{\subset} & L_0^{(j)} 
}}
\right\} \ \ \ \
X_2 = \left\{
\vcenter{\xymatrix{
L^{(i)} \ar[d]^x \ar@{}[r]|-{\subset} & L_0^{(i)} \\
L^{(j)} \ar@{}[r]|-{\subset} & L_0^{(j)}
}} \right\} \ \ \ \ 
Y = \left\{
\vcenter{\xymatrix{
L^{(i)} \ar[dr]^x \ar@{}[r]|-{\subset} & L_0^{(i)} \\
L^{(j)} \ar@{}[r]|-{\subset} & L_0^{(j)}
}} \right\}
$$
where we omit the bundles involving vertices other than $i$ and $j$ in order to simplify the exposition (they do not play a significant role)
Inside $Y$ one also has the locus 
$$Z = \left\{
\vcenter{\xymatrix{
L^{(i)} \ar[dr]^x \ar@{}[r]|-{\subset} & L_0^{(i)} \ar[dr]^x & \\
L^{(j)} \ar@{}[r]|-{\subset} & L_0^{(j)} \ar@{}[r]|-{\subset} & t^{-1}L^{(j)} 
}} \right\}
$$
One can check that $X_1,X_2 \subset Z$ and that both have codimension $a_j-1$. On the other hand,
\begin{equation}\label{eq:X1X2}
X_1 \cap X_2 = \left\{
\vcenter{\xymatrix{
L^{(i)} \ar[d]^x \ar@{}[r]|-{\subset} & L_0^{(i)} \ar[d]^x \\
L^{(j)} \ar@{}[r]|-{\subset} & L_0^{(j)}
}}
\right\} 
\subset
\left\{
\vcenter{\xymatrix{
L^{(i)} \ar[dr]^x \ar@{}[r]|-{\subset} & L_0^{(i)} \ar[d]^x \\
L^{(j)} \ar@{}[r]|-{\subset} & L_0^{(j)}
}}
\right\} \subset Z
\end{equation}
with both inclusions of codimension $a_j-1$. In particular $X_1$ and $X_2$ intersect inside $Z$ in the expected codimension. 
Thus the excess intersection of $X_1 \times_Y X_2$ is given by the quotient $Y/Z$. To identify this quotient more explicitly we note that on $Y$ there exists a natural map 
$$x: L_0^{(i)}/L^{(i)} \to t^{-1}L_0^{(j)}/t^{-1}L^{(j)} \la 1 \ra \{-1\}.$$
which vanishes exactly along $Z$. It follows that 
$$Y/Z \cong (L_0^{(j)}/L^{(j)}) \otimes (L_0^{(i)}/L^{(i)})^\vee \la 1 \ra \{1\}$$ 
where we use that the product with $t$ yields an isomorphism
$$t^{-1}L_0^{(j)}/t^{-1}L^{(j)} \cong L_0^{(j)}/L^{(j)} \{2\}.$$
Thus, we have the following triangle
$$\O_{(\calS')^\cl} \otimes (L_0^{(i)}/L^{(i)}) \otimes (L_0^{(j)}/L^{(j)})^\vee [1] \la -1 \ra \{-1\} \to  \O_{\calS'} \to \O_{(\calS')^{\cl}} $$
of sheaves on the derived scheme $\calS'$. On the other hand, as can be seen from \eqref{eq:X1X2}, we have 
$$(\calS')^\cl = X_1 \cap X_2 = \calR_{e_i+e_j}^{\cl}.$$ 
Thus, applying the functor $m_* \cl'_*$ gives us an exact triangle
$$\O_{\calR^\cl_{e_i+e_j}} (L_0^{(i)}/L^{(i)}) \otimes (L_0^{(j)}/L^{(j)})^\vee [1] \la -1 \ra \{-1\} \to m_* \cl'_* (\O_{\calS'}) 
\to \O_{\calR^\cl_{e_i+e_j}}.$$
Shifting appropriately this recovers (\ref{eq:ses8}). 
\end{proof}

\begin{Corollary}\label{cor:map2}
If $i,j \in I$ with $\la \alpha_i,\alpha_j \ra = -1$ we have 
$$\cP_{e_i} \n \cP_{e_j,-1} \cong 
\begin{cases}
\cP_{e_i,-1} {\Delta} \cP_{e_j} [-1] \la 1 \ra \{1\} & \text{ if $i \to j$ } \\
\cP_{e_i,-1} {\Delta} \cP_{e_j} [1] \la -1 \ra \{1\} & \text{ if $j \to i$ }.
\end{cases}$$
\end{Corollary}
\begin{proof}
We use the exact sequences from Lemma \ref{lem:map1}. If $i \to j$ then we have 
$$\cP_{e_i} \n \cP_{e_j, -1} \cong \cP_{e_i+e_j,-e_j} \cong \cP_{e_i,-1} {\Delta} \cP_{e_j} [-1] \la1\ra \{1\}.$$
If $j \to i$ then we have
$$\cP_{e_i} \n \cP_{e_j, -1} \cong \cP_{e_i+e_j,-e_i} [1]\la-1 \ra \cong \cP_{e_i,-1} {\Delta} \cP_{e_j} [1]\la-1\ra \{1\}.$$
\end{proof}


\begin{Lemma}\label{lem:V}
For any $i \in I$ and $\ell \in \Z $, there is a simple object $\cQ \in \cKP$
and exact triangles
\begin{align}\label{eq:VP}
\begin{split}
\cQ \to \cV_i * \cP_{e_i, \ell} \to \cP_{e_i, \ell+1} \{1\},\\
\cP_{e_i, \ell+1} \{-1\} \to \cP_{e_i, \ell} * \cV_i \to \cQ 
\end{split}
\end{align}
\end{Lemma}
\begin{proof}
The short exact sequences
\begin{align*}
& 0 \to \O_{\calR_{e_i}} \otimes (L^{(i)}/tL_0^{(i)}) \to \O_{\calR_{e_i}} \otimes (L_0^{(i)}/tL_0^{(i)}) \to  \O_{\calR_{e_i}} \otimes  (L_0^{(i)}/L^{(i)}) \to 0 \\
& 0 \to \O_{\calR_{e_i}} \otimes (tL_0^{(i)}/tL^{(i)}) \to \O_{\calR_{e_i}} \otimes (L^{(i)}/tL^{(i)}) \to \O_{\calR_{e_i}} \otimes (L^{(i)}/tL_0^{(i)}) \to 0
\end{align*}
yield
\begin{align*}
& 0 \to \cQ[-d]\{1\} \to \cV_i * \cP_{e_i, 0}[-d] \to \cP_{e_i,1}[-d] \{1\}\to 0 \\
& 0 \to  \cP_{e_i,1}[-d] \{-1\} \to \cP_{e_i, 0}*\cV_i[-d]\{1\} \to \cQ[-d] \to 0
\end{align*}
where $d=\frac12(a_i-1)$ and
$$\cQ \cong  \O_{\calR_{e_i}} \otimes (L^{(i)}/tL_0^{(i)})[d]\{-1\}.$$
Recall that $t$ yields an isomorphism
\begin{align}\label{loop-t}(L_0^{(i)}/L^{(i)}) \{-2\} \cong (tL_0^{(i)}/tL^{(i)}).\end{align}
Shifting these by $[d] \{-1\}$ recovers the triangles in (\ref{eq:VP}) when $\ul=0$. For more general $\ul$ we twist by the appropriate line bundle and argue as above.
\end{proof}

\subsection{Duals}\label{sec:duals}
In this section we gather some formulas for the duals of some objects in the category $\cKP$. By \cite[Thm.~6.1]{CW2}, the monoidal categories $\Coh^\hGO(\calR)$ and $\cKP$ are rigid with respect to the $*$-product. We denote ${}^\L \cF$ and ${}^\R\cF$ the left and right duals of an object $\cF$. By loc.~cit. we have
\begin{equation}\label{eq:duals}
{}^\L \cF \cong \D_1(\cF^*) ,\quad {}^\R\cF \cong \D_1(\cF)^*.
\end{equation}
where $\D_1$ and $(-)^*$ are involutions on $\Coh^\hGO(\calR)$ which preserve $\cKP$. To define these involutions following \cite[\S 5,6]{CW2} denote
$$\pi_1, \pi_2: \calR/\hGO \to N_\O/\hGO ,\quad \iota: \calR/\hGO \to \calR/\hGO$$
the first and second projections from \eqref{R/G} and the involution induced by exchanging the two factors from \eqref{R/G}. Then $\bbD_1 = \calHom(-, \omega_1)$ and $(-)^*=\iota_* \cong \iota^*$ where
$$\omega_1 = (\pi_1)^!(\calO_{N_\O/\hGO}) \in \IndCoh^{\hG_\calO}(\calR).$$

\begin{Lemma}\label{lem:*}
For any $\uk, \ul \in \Z I$ we have an isomorphism $(\cP_{\uk,\ul})^* \cong \cP_{-\uk,\ul}$.
\end{Lemma}
\begin{proof}
This follows similarly to \cite[Lem.~8.18]{CW2}. 
\end{proof}

\begin{Lemma}\label{lem:D1}
For $k \in\bbN$ and $\ell \in \Z$ we have  isomorphisms
\begin{align*}
\D_1(\cP_{ke_i,\ell}) \dcong (\psi_i^+)^k * \cP_{ke_i,-\ell + m_i^+}
,\quad
\D_1(\cP_{-ke_i,\ell}) \dcong (\psi_i^-)^{-k} * \cP_{-ke_i,-\ell - m_i^-}
\end{align*}
where $m_i^\pm$ were introduced in \eqref{m}.
\end{Lemma}

\begin{proof}
We write out only the first isomorphism as the proof of the second is similar. We take $\ul = \u0$ to simplify notation -- the general case just involves taking the dual of the corresponding bundle and thus changing $\ul$ to $\-\ul$. With $i_1$ as in \eqref{def-R}.we have 
\begin{align*}
\D_1(\cP_\uk) 
&\cong \cHom\Big(\cl_* (\O_{\calR_{\uk}^\cl} [\frac12 k(a_i-k)] \,, i_1^! (\omega_{\Gr_{\uk}})\Big) \\
&\cong \cl_* \cHom\Big(\O_{\calR_{\uk}^\cl} , \cl^! i_1^! \omega_{\Gr_{\uk}}\Big) 
[-\frac12 k(a_i-k)] 
\end{align*}
The composed map $i_1 \circ \cl:\calR_{ke_i}\to\Gr_{ke_i}\times N_\calO$ is the obvious inclusion
$$\begin{array}{c}
\left\{(L,x)\in\Gr\times N_\calK\,;\,tL_0 \subset L \subset L_0,\,
\dim(L_0/L)=ke_i,\,
x(L)\subset L,\,
x(L_0)\subset L_0\right\}\\[-1.5ex]
\rotatebox{270}{$\subseteq$}\\
\left\{(L,x)\in\Gr\times N_\calK\,;\,tL_0 \subset L \subset L_0,\,
\dim(L_0/L)=ke_i,\,
x(L_0)\subset L_0\right\}
\end{array}$$
Since $L^{(j)}=L_0^{(j)}$ for $i\neq j$, the image of $i_1 \circ \cl$ is the vanishing locus of the obvious map of bundles
\begin{equation}\label{eq:map1}
x: \bigoplus_{j \to i \in E} L_0^{(j)}/tL_0^{(j)} \to (L_0^{(i)}/L^{(i)} )\la 1 \ra \{-1\}
\end{equation}
Recall that for a closed immersion $f:X\to Y$ of smooth schemes the pullback  and the exceptional pullback of quasi-coherent
sheaves are related by the following formula 
$$f^!(-)\cong f^*(-)\otimes \det(N_{X/Y}[-1]).$$
This means that 
\begin{align*}
\cl^! i_1^! (-)
\cong&\cl^* i_1^*(-) \otimes \det 
\Big(\big(\bigoplus_{j \to i \in E} L_0^{(j)}/tL_0^{(j)}\big)^\vee\otimes 
(L_0^{(i)}/L^{(i)} )[-1]\la 1\ra\{-1\}\Big) \\ 
\cong &\cl^* i_1^*(-) \otimes \det((L_0^{(i)}/L^{(i)}) \{-1\})^{\sum_{j \to i \in E} a_j})^{-k}\otimes \bigotimes_{j \to i \in E} \det((L_0^{(j)}/tL_0^{(j)})[1]\la -1 \ra)^{-k}.
\end{align*}
Moreover, by \cite[(3.10)]{CW1} we have
$$\omega_{\Gr_{ke_i}} \cong \O_{\Gr_{ke_i}} \otimes \det(L_0^{(i)}/L^{(i)})^{-a_i} \otimes \det(L_0^{(i)}/tL_0^{(i)})^k [k(a_i-k)].$$
Putting this together we get 
\begin{align*}
\D_1(\cP_{ke_i}) 
\cong & \cl_* \Big(\O_{\calR_{ke_i}^\cl} \otimes \det((L_0^{(i)}/L^{(i)}) \{-1\})^{m_i^+}\Big) 
\otimes\bigotimes_{j \to i \in E} \det((L_0^{(j)}/tL_0^{(j)})[1]\la -1 \ra)^{-k}  \\
& \otimes \det((L_0^{(i)}/tL_0^{(i)}) [1]\la -1 \ra)^{k} [\frac12 k(a_i-k)] [-a_ik] \la a_ik \ra \{-a_ik\} \\
\cong &
 \Big(\phi_i*\bigast_{j \to i \in E} \phi_j^{-1} \Big)^{k} * \cP_{ke_i,m_i^+} [-a_ik] \la a_ik \ra \{ -a_ik \}.
\end{align*}
\end{proof}

\begin{Corollary}\label{cor:duals}
If $k \in\bbN$ and $\ell \in \Z $, then we have  isomorphisms
\begin{align*}
{}^\L \cP_{ke_i,\ell} \dcong (\psi_i^-)^{-k} * \cP_{-ke_i, -\ell - m_i^-}
,\quad
{}^\R\cP_{ke_i,\ell} \dcong \cP_{-ke_i, -\ell + m_i^+} * (\psi_i^+)^k.
\end{align*}
In particular, we have ${}^\R\cP_{ke_i,\ell} \dcong {}^\L \cP_{ke_i,\ell-m_i} *( \psi_i)^k$.
\end{Corollary}
\begin{proof}
This follows from (\ref{eq:duals}) and Lemmas \ref{lem:*} and \ref{lem:D1}. 
\end{proof}

\subsection{Invariants $\Lambda$ and $\b$}
Let us compute the invariants $\Lambda$ and $\b$ between some objects in \eqref{eq:cP} and \eqref{eq:V-phi-psi}.

\begin{Lemma}\label{lem:B1}
If $i\neq j$ in $I$ and $\uk, \ul,\ul'\in\bbZ I$ with $-a_i \le k_i \le a_i$, then we have 
\begin{align*}
& \Lambda(\phi_i\,,\, \cP_{\uk,\ul}) = 2k_i  = - \Lambda(\cP_{\uk,\ul}\,,\, \phi_i) \\
& \Lambda(\psi_i^-\,,\, \cP_{\uk,\ul}) = 2k_i - \sum_{i \to j} 2k_j = - \Lambda(\cP_{\uk,\ul}\,,\,\psi_i^-) \\
& \Lambda(\psi_i^+\,,\, \cP_{\uk,\ul}) = 2k_i - \sum_{j \to i} 2k_j = - \Lambda(\cP_{\uk,\ul}\,,\, \psi_i^+).
\end{align*}
In particular, the objects $\phi_i$ and $\psi_i^\pm$ all $q$-commute with $\cP_{\uk,\ul}$. 

\end{Lemma}
\begin{proof}
For notational convenience assume $\ul=0$ and $k_i \ge 0$. Both $\phi_i * \cP_\uk$ and $\cP_\uk * \phi_i$ are supported over 
$$\calR_\uk^\cl = \{(L,x)\in\calR^\cl\,;\,tL_0^{(i)} \subset L^{(i)} \subset L_0^{(i)}\,,\,\dim(L_0^{(i)}/L^{(i)}) = k_i \}.$$ 
Up to a common shift they are given by line bundles $\det(L_0^{(i)}/tL_0^{(i)})$ and $\det(L^{(i)}/tL^{(i)})$ respectively. 
Now consider the following standard short exact sequences of vector bundles on $\calR_\uk^\cl$
\begin{align*}
& 0 \to L^{(i)}/tL_0^{(i)} \to L_0^{(i)}/tL_0^{(i)} \to L_0^{(i)}/L^{(i)} \to 0 \\
& 0 \to L_0^{(i)}/L^{(i)} \to t^{-1}L^{(i)}/L^{(i)} \to t^{-1}L^{(i)}/L_0^{(i)} \to 0.
\end{align*}
The corresponding determinantal line bundles satisfy 
\begin{align*}
\det(L_0^{(i)}/tL_0^{(i)}) &\cong \det(L^{(i)}/tL_0^{(i)}) \otimes \det(L_0^{(i)}/L^{(i)}) \\
\det(t^{-1}L^{(i)}/L^{(i)}) &\cong \det(L_0^{(i)}/L^{(i)}) \otimes \det(t^{-1}L^{(i)}/L_0^{(i)}).
\end{align*}
Further, we have
$$t^{-1}L^{(i)}/L^{(i)} \cong L^{(i)}/tL^{(i)} \{2\}
,\quad
t^{-1}L^{(i)}/L_0^{(i)} \cong L^{(i)}/tL_0^{(i)} \{2\}.$$
We deduce that
\begin{align*}
\cP_\uk * \phi_i
&\cong\det(t^{-1}L^{(i)}/L^{(i)})\{-2a_i\}\\
&\cong \det(L_0^{(i)}/L^{(i)}) \otimes \det(t^{-1}L^{(i)}/L_0^{(i)})\{-2a_i\}\\
&\cong \det(L_0^{(i)}/L^{(i)}) \otimes \det(L^{(i)}/tL_0^{(i)})\{-2k_i\}\\
&\cong\det(L_0^{(i)}/tL_0^{(i)})\{-2k_i\}\\
&\cong\phi_i * \cP_\uk\{-2k_i\}
\end{align*}
Hence Corollary \ref{cor:real}, Formula \eqref{rMN} and Proposition \ref{prop:homs1} imply that
$$ \Lambda(\phi_i\,,\, \cP_{\uk}) = 2k_i  = - \Lambda(\cP_{\uk}\,,\, \phi_i) $$
The subsequent equalities in the lemma are immediate consequences, by Proposition \ref{prop:Lambda*}. 
If $\ul\neq 0$ and $k_i \le 0$ the computation is similar.
For instance $\ul\neq 0$ contributes to a common power of $\det(L_0/L)$
in $\phi_i * \cP_{\uk,\ul}$ and $\cP_{\uk,\ul} * \phi_i$ which does not change their relative position.
The last claim is Proposition \ref{prop:qcom}.
\end{proof}

\begin{Proposition}\label{prop:B2}
For any $\ell,\ell' \in \Z $ we have 
\begin{enumerate}[label=$\mathrm{(\alph*)}$,leftmargin=8mm]
\item
$\Lambda(\cP_{e_i,\ell}\,, \cP_{e_j,\ell'}) = 
\begin{cases} 
2(\ell-\ell') & \text{ if $i=j$ and $\ell \ge \ell'$ } \\
-2 & \text{ if $i=j$ and $\ell< \ell'$ } \\
- \la \alpha_i,\alpha_j \ra & \text{ if $i \ne j$. }
\end{cases}$ 
\item
$\b(\cP_{e_i,\ell}\,, \cP_{e_j,\ell'}) =             
\begin{cases}
0 & \text{ if $i=j$ and $\ell = \ell'$ } \\
|\ell-\ell'|-1 & \text{ if $i=j$ and $\ell \ne \ell'$ } \\
- \la \alpha_i,\alpha_j \ra & \text{ if $i \ne j$. }
\end{cases}$
\end{enumerate}
\end{Proposition}

\begin{proof}
If $i=j$ and $\ell +1< \ell'$ then Corollary \ref{cor:GrG} yields maps
$$\cP_{e_i, \ell} * \cP_{e_i, \ell'} \to \cQ \{-1\} \to \cP_{e_i, \ell'} * \cP_{e_i, \ell} \{-2\}$$
with $\calQ\neq 0$.
Hence Corollary \ref{cor:real}, Formula \eqref{rMN} and Proposition \ref{prop:homs1} imply that
$$\Lambda(\cP_{e_i,\ell}\,, \cP_{e_i,\ell'})  = -2.$$

If $\ell +1=\ell'$ then Corollary \ref{cor:GrG} yields an isomorphism
$$\cP_{e_i, \ell'} * \cP_{e_i, \ell} \cong \cP_{e_i, \ell} * \cP_{e_i, \ell'} \{2\}$$
Hence for the same reason as above we have
$$\Lambda(\cP_{e_i,\ell}\,, \cP_{e_i,\ell'})  = -2.$$

If $\ell \ge \ell'$ then again from Corollary \ref{cor:GrG} we have a sequence of maps
$$\cP_{e_i, \ell} * \cP_{e_i, \ell'} \to \cP_{e_i, \ell-1} * \cP_{e_j, \ell'+1} \{2\} \to \dots \to \cP_{e_i, \ell'} *
 \cP_{e_i, \ell} \{2(\ell-\ell')\}.$$
Up to the midpoint these maps are surjective and thereafter they are injective. 
This implies that the composition above is nonzero.
We deduce as before that 
$$\Lambda(\cP_{e_i,\ell}\,, \cP_{e_i,\ell'}) = 2(\ell-\ell').$$

If $\la \alpha_i, \alpha_j \ra = -1$ and $ i\to j \in E$, then Lemma \ref{lem:map1} yields maps
\begin{align*}
& \cP_{e_i,\ell} * \cP_{e_j, \ell'} \to \cP_{e_i+e_j, \ell e_i + \ell' e_j} \to \cP_{e_j, \ell'} * \cP_{e_i,\ell} \{1\} \\
&  \cP_{e_j,\ell'}* \cP_{e_i,\ell} \to  \cP_{e_i+e_j,(\ell+1)e_i+(\ell'-1)e_j} [1]\la-1\ra \to \cP_{e_i,\ell} * \cP_{e_j, \ell'} \{1\}
\end{align*}
Both these compositions are nonzero since they are the composition of a surjective with an injective map. 
This implies that $\Lambda(\cP_{e_i,\ell}\,, \cP_{e_j,\ell'}) = 1 = \Lambda(\cP_{e_j,\ell'}\,, \cP_{e_i,\ell})$. 

If $\la \alpha_i, \alpha_j \ra = 0$ then it is easy to see that $\cP_{e_i,\ell} * \cP_{e_j,\ell'} \cong \cP_{e_j,\ell'} * \cP_{e_i,\ell}$. This implies that $\Lambda(\cP_{e_i,\ell}\,, \cP_{e_j,\ell'})=0$. Part (a) is proved.
Part (b) follows immediately from part (a). 
\end{proof}

\begin{Proposition}\label{prop:B3}
The following equalities hold for $i,j \in I$ and $ \ell,\ell'\in\bbZ$. 
\begin{enumerate}[label=$\mathrm{(\alph*)}$,leftmargin=8mm]
\item If $i \ne j$ then
$$\Lambda(\cP_{e_i,\ell}\,,\cP_{-e_j,\ell'}) = - \Lambda(\cP_{-e_j,\ell'}\,,\cP_{e_i,\ell}) = 
\begin{cases}
0 & \text{ if $\la \alpha_i, \alpha_j \ra = 0$ } \\
-1 & \text{ if $i \to j$ } \\
1 & \text{ if $j \to i$ }
\end{cases}$$
In particular, $\b(\cP_{e_i,\ell}\,,\cP_{-e_j,\ell'}) = 0$. 
\item If $m_i < 0$ then 
$$\b(\cP_{e_i,\ell}, \cP_{-e_i,\ell'}) =
\begin{cases}
0 & \text{ if $m_i^+ < \ell+\ell' < -m_i^-$} \\
-\ell - \ell' + m_i^+ + 1 & \text{ if $\ell + \ell' \le m_i^+$ } \\
\ell + \ell' + m_i^- + 1 & \text{ if $- m_i^- \le \ell + \ell'$ }
\end{cases}$$
\item If $m_i \ge 0$ then 
$$\b(\cP_{e_i,\ell}, \cP_{-e_i,\ell'}) =
\begin{cases}
m_i+2 & \text{ if $- m_i^- \le \ell+\ell' \le m_i^+$} \\
-\ell - \ell' + m_i^+ + 1 & \text{ if $\ell + \ell' < - m_i^-$ } \\
\ell + \ell' + m_i^- + 1 & \text{ if $m_i^+ < \ell + \ell'$ }
\end{cases}$$
\end{enumerate}
\end{Proposition}

\begin{proof}
If $i\neq j$, then Corollary \ref{cor:duals}, Lemma \ref{lem:B1} and Propositions \ref{prop:B2}, \ref{prop:Lduals} yield
\begin{align*}
\Lambda(\cP_{e_i,\ell}\,,\cP_{-e_j,\ell'}) 
&= \Lambda(\cP_{e_i,\ell}\,, {}^\R\cP_{e_j,-\ell'+m_j^+} * (\psi_j^+)^{-1}) \\
&= \Lambda(\cP_{e_i,\ell}\, {}^\R\cP_{e_j,-\ell'+m_j^+}) + \Lambda( \cP_{e_i,\ell}\,, (\psi_j^+)^{-1}) \\
&= \Lambda(\cP_{e_j,-\ell'+m_j^+}, \cP_{e_i,\ell}) + \Lambda(\psi_j^+, \cP_{e_i,\ell}) \\
&= -\la \alpha_i,\alpha_j \ra - 2 \delta_{i \to j}
\end{align*}
A similar calculation gives
$$\Lambda(\cP_{-e_j,\ell'}, \cP_{e_i,\ell}) = \Lambda(\psi_j^-,\cP_{e_i,\ell}) + \Lambda(\cP_{e_i,\ell}, \cP_{e_j,-\ell'-m_j^-}) 
= - 2 \delta_{j \to i} - \la \alpha_i, \alpha_j \ra.$$
Together these prove part (a). 

The case $i=j$ is similar. 
By Corollary \ref{cor:duals}, Lemma \ref{lem:B1} and Propositions \ref{prop:B2}, \ref{prop:Lduals}, we have
\begin{align*}
\Lambda(\cP_{e_i,\ell}\,, \cP_{-e_i,\ell'})
&= \Lambda(\cP_{e_i,\ell}\,, {}^\R\cP_{e_i,-\ell' + m_i^+} * (\psi_i^+)^{-1}) \\
&= \Lambda(\cP_{e_i,\ell}\,, {}^\R\cP_{e_i,-\ell' + m_i^+}) + \Lambda(\cP_{e_i,\ell}\,, (\psi_i^+)^{-1}) \\
&= \Lambda(\cP_{e_i,-\ell' + m_i^+}\,, \cP_{e_i,\ell}) + \Lambda(\psi_i^+\,, \cP_{e_i,\ell}) \\
&= \begin{cases}
-2(\ell+\ell'-m_i^+-1) & \text{ if $\ell+\ell' \le m_i^+ $ } \\
0 & \text{ if $\ell+\ell' > m_i^+ $ } 
\end{cases}
\end{align*}
We similarly get
\begin{align*}
\Lambda(\cP_{-e_i,\ell'}\,, \cP_{e_i,\ell}) 
&= \Lambda(\psi_i^- *( {}^\L \cP_{e_i, -\ul' - m_i^-})\,, \cP_{e_i,\ell}) \\
&= \Lambda(\psi_i^-\,, \cP_{e_i,\ell}) + \Lambda({}^\L \cP_{e_i, -\ul' - m_i^-}\,, \cP_{e_i,\ell}) \\
&= \Lambda(\psi_i^-\,, \cP_{e_i,\ell}) + \Lambda(\cP_{e_i,\ell}\,, \cP_{e_i, -\ul' - m_i^-}) \\
&= \begin{cases}
2(\ell + \ell' + m_i^- + 1) & \text{ if $\ell +\ell' \ge - m_i^-$ } \\
0 & \text{ if $\ell +\ell' < - m_i^-$ }
\end{cases}
\end{align*}
Together these identities imply parts (b) and (c). 
\end{proof}

\begin{Corollary}\label{cor:B4}
We have $\Lambda({}^\L \cP_{e_i}\,,\cP_{e_i}) = 0$. If $m_i<0$, then we also have
$$\Lambda(\cP_{e_i}\,, {}^\L \cP_{e_i}) = \Lambda({}^\L \cP_{e_i,-1}\,,\cP_{e_i}) = \Lambda(\cP_{e_i}\,, {}^\L \cP_{e_i,-1}) = 2.$$
\end{Corollary}

\begin{proof}
By \eqref{adjunction} we have the adjunction maps
$${}^\L \cP_{e_i} * \cP_{e_i} \to 1 \to \cP_{e_i} * {}^\L \cP_{e_i}$$ 
which proves $\Lambda({}^\L \cP_{e_i}\,,\,\cP_{e_i}) = 0$ by Proposition \ref{prop:homs1}. On the other hand,  we have
\begin{align*}
\b({}^\L \cP_{e_i}\,,\, \cP_{e_i}) &= \b\Big((\psi_i^-)^{-1} * \cP_{-e_i, -m_i^-}\,,\, \cP_{e_i}\Big) \\
&= \b((\psi_i^-)^{-1}\,,\, \cP_{e_i}) +\b(\cP_{-e_i,-m_i^-}\,,\, \cP_{e_i}) \\
&= 1
\end{align*}
The first equality is Corollary \ref{cor:duals}, the second one is Proposition \ref{prop:Lambda*},
and the third one is Proposition \ref{prop:B3} and the fact that the object 
$\psi_i^-$ $q$-commutes with $\cP_{e_i}$ by lemma \ref{lem:B1}. 
Similarly, Corollary \ref{cor:duals} gives
${}^\L\cP_{e_i,-1}=(\psi_i^-)^{-1} * \cP_{-e_i, -m_i^-+1}$,
hence by Lemma \ref{lem:V} and \eqref{adjunction3} we have nonzero maps 
$${}^\L \cP_{e_i,-1} * \cP_{e_i} \{-1\} \to \cV_i \to \cP_{e_i} * {}^\L \cP_{e_i,-1} \{1\}$$ which 
proves $\Lambda({}^\L \cP_{e_i,-1},\cP_{e_i}) = 2$ by Proposition \ref{prop:bselfdual}. 
The remaining claims follow similarly, since we have
$$\b({}^\L \cP_{e_i,-1}\,,\, \cP_{e_i}) = \b\Big((\psi_i^-)^{-1} * \cP_{-e_i, -m_i^-+1}\,,\, \cP_{e_i}\Big)  = \b(\cP_{-e_i,-m_i^- + 1}\,,\, \cP_{e_i}) = 2$$
As above, the first equality is Corollary \ref{cor:duals}, the second one is Proposition \ref{prop:Lambda*} and the fact that the object $\psi_i^-$ $q$-commutes with $\cP_{e_i}$ by lemma \ref{lem:B1}, and the third one is Proposition \ref{prop:B3}. 
\end{proof}

\subsection{Relations involving $\cV_i$}

We begin by observing the following result. 

\begin{Proposition}\label{prop:ses3}
If $m_i < 0$ there exist short exact sequences
\begin{align*}
& 0 \to 1 \to \cP_{e_i} * {}^\L \cP_{e_i} \to \cS_i \to 0 \\
& 0 \to \cS_i \{-2\} \to {}^\L \cP_{e_i} * \cP_{e_i} \to 1 \to 0 
\end{align*}
where $\cS_i \in \cKP$ is simple and the top left, bottom right maps are by adjunction. 
\end{Proposition}
\begin{proof}
These sequences follow from Proposition \ref{prop:Lb=1} since $\b(\cP_{e_i}, {}^\L \cP_{e_i}) = 1$ by Corollary \ref{cor:B4}.
\end{proof}

The main goal of this subsection is to obtain an analogue of this result, namely Proposition \ref{prop:ses4}, where ${}^\L \cP_{e_i}$ is replaced 
with ${}^\L \cP_{e_i,-1}$. Such a result if crucial for proving that the functor $\F$ is t-exact later in the paper. The preliminaries in this section 
are only used for this purpose. One can imagine establishing Proposition \ref{prop:ses4} by direct computation but this seems technically 
painful and we prefer an indirect argument involving certain K-theory results already established in \cite{FT} in part because some of the 
intermediate results may be of independent interest. 

\begin{Lemma}\label{lem:Q}
For each objects $\cF,\cG\in\cKP$, 
the cone of the $r$-matrix $\bfR_{\cF_z,\cG}^{\ren}$ is also an object in $ \cKP$. 
\end{Lemma}

\begin{proof}
Let $\cQ$ be the cone of $\bfR_{\cF_z,\cG}^{\ren}$. 
It is set theoretically supported over $z=0$ because $\bfR_{\cF_z,\cG}^{\ren}$ is an isomorphism for $z\neq 0$. 
Thus, for $k \gg 0$, we have
$$\cQ \otimes_{\C[z]} \C[z]/z^k \cong \cQ \oplus \cQ[1]\{-2k\}.$$
On the other hand, the restriction of the objects $\cF_z * \cG$ and $\cG * \cF_z$ to $\Spec( \C[z]/z^k)$ belong to $\cKP$,
because the heart is preserved under finite extensions.
It follows that the cone of the restriction of $\bfR_{\cF_z,\cG}^{\ren}$ to $\Spec( \C[z]/z^k)$, 
which we noted is isomorphic to $\cQ \oplus \cQ[1]\{-2k\}$, 
must be supported in homological degrees $-1$ and $0$ with respect to the Koszul-perverse t-structure. 
This means that $\cQ$ must be supported in homological degree $0$.  
\end{proof}

\begin{Lemma}\label{lem:FT}
The map $\overline{\Phi} = \overline{\Phi}^{\underline{\lambda}}_{-\alpha}$ in \cite[Thm.~8.1]{FT} from the $(0,-\alpha)$-shifted quantum group 
of affine type $Q_\af$ to the Grothendieck group $K_0(\cKP)|_{t=1}$ takes 
\begin{align}
\label{eq:1} f_{i,r} &\mapsto \pm [(\psi_i^-)^{-1} * \phi_i]^{\frac12} * [(\cP_{e_i,r-a_i-m_i^-})_z] \\
\label{eq:2} e_{i,-r} &\mapsto \pm [\psi_i^- * \phi_i^{-1} * \psi_i]^{\frac12} * [({}^\L \cP_{e_i,r-a_i-m_i^-})_z] \\
\label{eq:3} A_{i,1}^+ &\mapsto - [\phi_i]^{-\frac12} * [\cV_i] \\
\label{eq:3A} \phi_i^+ &\mapsto [\phi_i]^{\frac12} \\
\label{eq:4} \psi_{i,0}^+ &\mapsto [\psi_i]^{\frac12} \\
\label{eq:5} h_{i,1} &\mapsto \pm (-(q^2+1)[(\cV_i)_z] +\sum_{j \leftrightarrow i \in E} q^{c_j} [(\cV_j)_z])
\end{align}
for some $c_j \in \Z$ and up to powers of $q$.
\end{Lemma}

\begin{Remark}
It is possible to keep track of the signs and powers of $q$ that appear on the right side of the equations above by 
following \cite{FT} more carefully. This is tedious and prone to errors. We can deduce the signs and 
powers of $q$ later on, given the existence of the Koszul-perverse t-structure and our calculation of certain $\La(-,-)$ pairings. 
We find this 
preferable.
\end{Remark}

\begin{proof}
The notation of \cite{FT} translates as follows
\begin{itemize}[leftmargin=8mm]
\item $\calR_{\varpi_{i,1}} = \calR^\cl_{e_i}$ and $\calR_{\varpi^*_{i,1}} = \calR_{-e_i}^\cl$,
\item $(\cQ_i)^{\otimes n} \otimes \O_{\calR_{\varpi_{i,1}}} \dcong \cP_{e_i,n}$ and 
$(\cS_i)^{\otimes n} \otimes \O_{\calR_{\varpi^*_{i,1}}} \dcong \cP_{-e_i,n}$ for $n \in \Z$,
\item $\prod_{t=1}^{a_i} {\bf w}_{i,t} = [\phi_i]$ and $e_1(\{{\bf w}_{i,t}\}_{t=1}^{a_i}) = [\cV_i]$  in K-theory, 
up to powers of ${\bf v} = q$.
\end{itemize}
By Corollary \ref{cor:duals} we have
$${}^\L \cP_{e_i,r-a_i-m_i^-} \dcong (\psi_i^-)^{-1} \cP_{-e_i,-r+a_i}$$
Moreover, we have $[\cF] = (1-q^2)[(\cF)_z]$ for any $\cF \in \cKP$. 
Together this yields \eqref{eq:1}, \eqref{eq:2}, \eqref{eq:3} and \eqref{eq:3A} from the first three maps in \cite[Thm.~8.1]{FT}.
Next from \cite[Eq. 6.1]{FT} we have
\begin{equation}\label{eq:local10}
z^{-b_i^+} \psi_i^+(z) = \frac{\prod_{j \leftrightarrow i} A_j^+(qz)}{A_i^+(z)A_i^+(q^{-2}z)}
\end{equation}
where $A_{i,0}^+ = (\phi_i^+)^{-1}$ and
\begin{equation}\label{eq:local11}
A_i^+(z) = \sum_{r \ge 0} A_{i,r}^+ z^{-r}
,\quad
\psi_i^+(z) = \sum_{r \ge 0} \psi_{i,r}^+ z^{-r}
= \psi^+_{i,0} \exp \left( (q-q^{-1}) \sum_{r > 0} h_{i,r}z^{-r} \right).
\end{equation} 
Note that we are in the case, see \eqref{m}, 
$$b_i^+=0,\quad b_i^-=m_i=-\langle\alpha_i,\alpha\rangle.$$
Equating the right hand sides of \eqref{eq:local10} and \eqref{eq:local11} gives
\begin{equation}\label{eq:6}
\psi_{i,0}^+ = (\phi_i^+)^2 \prod_{j \leftrightarrow i} (\phi_j^+)^{-1}
\end{equation}
when considering the coefficient of $z^0$ and
\begin{equation}\label{eq:7}
(1-q^2)h_{i,1} = \pm q^{c_i} (1+q^2)\phi_i^+ A_{i,1}^+ \pm \sum_{j \leftrightarrow i} q^{c_j} \phi_j^+ A_{j,1}^+
\end{equation}
for some $c_j \in \Z$, when considering the coefficient of $z^{-1}$. Using \eqref{eq:3} and \eqref{eq:3A}, equalities \eqref{eq:6} and \eqref{eq:7} recover \eqref{eq:4} and \eqref{eq:5} respectively.
\end{proof}

\begin{Corollary}\label{cor:FT}
In K-theory with $t=1$ we have
\begin{equation}\label{eq:Kcommutator}
[{}^\L \cP_{e_i,-1} * (\cP_{e_i})_z] - q^2 [(\cP_{e_i})_z * {}^\L \cP_{e_i,-1}] = q^b \left((q^2+1)[\cV_i] + \sum_{j \leftrightarrow i \in E} q^{c_j} [\cV_j] \right)
\end{equation}
for some $b, c_j \in \Z$.
\end{Corollary}
\begin{proof}
Let $r = a_i+m_i^-$ and $s = -a_i-m_i^-+1$. Then by \cite[Eq. ${\rm \hat{U}}6$]{FT} one has
\begin{equation}\label{eq:local13}
[e_{i,r},f_{i,s}] = \psi_{i,0}^+ h_{i,1}.
\end{equation}
On the other hand, by \eqref{eq:1}, \eqref{eq:2}, \eqref{eq:4} and \eqref{eq:5} the map $\Phi$ takes
\begin{align*}
e_{i,r}f_{i,s} &\mapsto \pm [\psi_i]^{\frac12} * [({}^\L \cP_{e_i,-1})_z * (\cP_{e_i})_z] \\
f_{i,s}e_{i,r} &\mapsto \pm [\psi_i]^{\frac12} * [(\cP_{e_i})_z * ({}^\L \cP_{e_i,-1})_z] \\
\psi_{i,0}^+ h_{i,1} &\mapsto [\psi_i]^{\frac12} \left(\pm (q^2+1)[(\cV_i)_z] \pm \sum_{j \leftrightarrow i} q^{c_j} [(\cV_j)_z] \right)
\end{align*}
up to powers of $q$. Since $\Phi$ is injective \eqref{eq:local13} relates these K-theory classes which, after cancelling the common 
$[\psi_i]^{\frac12}$ factors and dividing by $1-q^2$, gives  
\begin{equation}\label{eq:local12}
[{}^\L \cP_{e_i,-1} * (\cP_{e_i})_z] \pm q^a [(\cP_{e_i})_z * {}^\L \cP_{e_i,-1}] = q^b \left(\pm (q^2+1)[\cV_i]
 \pm \sum_{j \leftrightarrow i} q^{c_j} [\cV_j] \right)
\end{equation}
for some $a,b \in \Z$. Now the head of $\cP_{e_i} * {}^\L \cP_{e_i,-1}$ cannot be $\cV_j$ for any $j \in I$ and thus its class in K-theory must 
cancel out on left side of \eqref{eq:local12}. This means that the first $\pm$ in \eqref{eq:local12} must be a minus and that $a=2$ since 
$\La(\cP_{e_i}, {}^\L \cP_{e_i,-1})=2$. Finally, by Lemma \ref{lem:Q}, the right hand side of \eqref{eq:local12} represents the class of an object 
in $\cKP$. Thus every $\pm$ on the right side of \eqref{eq:local12} must be a plus.
\end{proof}

\begin{Proposition}\label{prop:ses4}
If $m_i < 0$ there is a simple object $\cK_i \in \cKP$ and sequences
\begin{align}
\label{eq:seq9} & 0 \to \cV_i \{-1\} \to \cP_{e_i} * {}^\L \cP_{e_i,-1} \to \cK_i \{1\} \to 0 \\
\label{eq:seq10}& 0 \to \cK_i \{-1\} \to {}^\L \cP_{e_i,-1} * \cP_{e_i} \to \cV_i \{1\} \to 0
\end{align}
which are exact on the left and right but not necessarily in the middle. 
We have
$$ \cP_{e_i} \d {}^\L \cP_{e_i,-1}\cong\cV_i \{-1\}
,\quad
 \cP_{e_i} \n {}^\L \cP_{e_i,-1}\cong \cK_i \{1\}
,\quad
 {}^\L \cP_{e_i,-1} \d \cP_{e_i}\cong\cK_i \{-1\} 
 ,\quad
 {}^\L \cP_{e_i,-1} \n \cP_{e_i}\cong\cV_i \{1\}
$$
Further, the composition factors of $\cP_{e_i} * {}^\L \cP_{e_i,-1}$ and ${}^\L \cP_{e_i,-1}*\cP_{e_i}$ are $\{\cK_i, \cV_i, \cV_j\,;\,j \leftrightarrow i \in E\}$ with multiplicity one,
up to Koszul and loop shifts. 
\end{Proposition}

\begin{proof}
By adjunction and Lemma \ref{lem:V} we have nonzero maps
$$\cV_i \{-1\} \to \cP_{e_i} * {}^\L \cP_{e_i,-1} ,\quad {}^\L \cP_{e_i,-1} * \cP_{e_i} \to \cV_i \{1\}.$$
We also know that $\cP_{e_i} * {}^\L \cP_{e_i,-1}$ and ${}^\L \cP_{e_i,-1} * \cP_{e_i}$ have simple heads and socles 
by Propositions \ref{prop:qcom}, \ref{prop:lsubquotient}, because the factors are all simple and real.  This yields the sequences \eqref{eq:seq9} and \eqref{eq:seq10} along with their exactness on the left and right. 

The composition factors of $\cP_{e_i} * {}^\L \cP_{e_i,-1}$ and ${}^\L \cP_{e_i,-1}*\cP_{e_i}$ are the same up to shifts.
To compute them, we consider the triangle
\begin{equation}\label{eq:local6}
(\cP_{e_i})_z * {}^\L \cP_{e_i,-1} \{-2\} \xrightarrow{\bfR^\ren} {}^\L \cP_{e_i,-1} * (\cP_{e_i})_z \to \cQ_i.
\end{equation}
Since, by Lemma \ref{lem:Q}, we have $\cQ_i \in \cKP_X$, we can use Corollary \ref{cor:FT} to conclude that the composition factors of $\cQ_i$ are $\{(\cV_i)^{\oplus 2}, \cV_j; j \leftrightarrow i \in E \}$ up to shifts, meaning that $\cV_i$ has multiplicity two and $\calV_j$ multiplicity one.

On the other hand, restricting \eqref{eq:local6} to $z=0$ recovers the triangle
$$\cP_{e_i} * {}^\L \cP_{e_i,-1} \{-2\} \xrightarrow{\r} {}^\L \cP_{e_i,-1} * \cP_{e_i} \to \cQ_i'$$
where $\cQ_i' \cong \Cone(\cQ_i \{-2\} \xrightarrow{z} \cQ_i)$. Since the image of $\r$ is the socle of ${}^\L \cP_{e_i,-1} * \cP_{e_i}$ which is $\cK_i \{-1\}$ it suffices to show that the composition factors of $\coker(z)$ are $\{\cV_i, \cV_j; j \leftrightarrow i\}$. Now, for degree reasons, the image $\Im(z) \subset \cQ_i$ is either zero or $\cV_i$. Thus the composition factors of $\cP_{e_i} * {}^\L \cP_{e_i,-1}$ are either $\{(\cV_i)^{\oplus 2}, \cV_j: j \leftrightarrow i \in E \}$ or $\{\cV_i, \cV_j: j \leftrightarrow i \in E\}.$ But $\cV_i$ is the head of $\cP_{e_i} * {}^\L \cP_{e_i,-1}$.  So, by Proposition \ref{prop:lsubquotient}, the object $\cV_i$ cannot occur more than once in the composition series. 
\end{proof}

\begin{Proposition}\label{prop:dimhom}
For $i,j \in I$ we have
\begin{equation}\label{eq:dimhom}
\Hom_{\cKP}({}^\L \cP_{e_j,-1} * \cP_{e_j}, \cP_{e_i} * {}^\L \cP_{e_i,-1} [a]\la -a \ra \{b\}) \cong 
\begin{cases}
\C & \text{ if $i \to j$ and $a=1$, $b=1$ } \\
\C & \text{ if $j \to i$ and $a=-1$, $b=1$ } \\
0&\text{if $\la\alpha_i,\alpha_j\ra=0$}
\end{cases}  
\end{equation}
If $j \to i$ let $\Gamma_{ji}$ denote the nonzero map. 
The composition factors of its image are $\{\cV_i, \cV_j\}$, up to shifts. 
\end{Proposition}

\begin{proof}
By adjunction 
$$\Hom_{\cKP}({}^\L \cP_{e_j,-1} * \cP_{e_j}, \cP_{e_i} * {}^\L \cP_{e_i,-1} [a]\la -a \ra \{b\}) 
\cong \Hom_{\cKP}(\cP_{e_j} * \cP_{e_i,-1}, \cP_{e_j,-1} * \cP_{e_i} [a]\la -a \ra \{b\}).$$

If $i \to j$ then, by \eqref{eq:ses7}, we have a sequence of maps
$$\cP_{e_j} * \cP_{e_i,-1} \to \cP_{e_i+e_j, -e_j} [1]\la -1 \ra \to \cP_{e_j,-1} * \cP_{e_i} [1]\la-1\ra \{1\}$$
whose composition is nonzero because the first map is surjective and the second injective. 
Since $\cP_{e_j} * \cP_{e_i,-1}$ and $\cP_{e_j,-1} * \cP_{e_i}$ have length two by \eqref{eq:ses7},
this implies the $i \to j$ case of \eqref{eq:dimhom}.

If $j \to i$ then, by \eqref{eq:ses8} with $i$ and $j$ interchanged,  we have the sequence of maps 
$$\cP_{e_j} * \cP_{e_i,-1} \to \cP_{e_i+e_j, -e_i} \to \cP_{e_j,-1} * \cP_{e_i} [-1]\la 1 \ra \{1\}$$
which similarly gives the $j \to i$ case of \eqref{eq:dimhom}. 

Now, we compute the composition factors of the image.
By Proposition \ref{prop:ses4}, the composition factors of ${}^\L \cP_{e_j,-1} * \cP_{e_j}$ are
$\{\cK_j, \cV_j, \cV_k\,;\,j\leftrightarrow k\}$, up to shifts,
while the composition factors of $\cP_{e_i} * {}^\L \cP_{e_i,-1}$ are 
$\{\cK_i, \cV_i, \cV_k\,;\, i\leftrightarrow k\}.$
Thus the composition factors of the image of any map
$\gamma:{}^\L \cP_{e_j,-1} * \cP_{e_j}\to\cP_{e_i} * {}^\L \cP_{e_i,-1} [a]\la -a \ra \{b\}$ is a subset of $\{\cV_i, \cV_j\}$. On the other hand, it cannot be just $\cV_i$ because then 
$\cV_i$ is in the head of ${}^\L \cP_{e_j,-1} * \cP_{e_j}$ which we know to be $\cV_j$. Likewise it cannot be just $\cV_j$ because then 
$\cV_j$ is in the socle of $\cP_{e_i} * {}^\L \cP_{e_i,-1}$ which we know to be $\cV_i$. The result follows. 
\end{proof}

\section{Root objects}\label{sec:root-objects}

The goal of this section is to define certain additional objects associated to affine roots and to study their properties.

\subsection{The object $\scrM(\beta-\delta)$}

Given $\beta \in \Phi^+$ choose a minimal expression $\beta = s_{i_k} \dots s_{i_1}(\alpha_i)$ and define
\begin{equation}\label{eq:b-d}
\scrM(\beta-\delta) = \cP_{e_{i_k}} \n (\cP_{e_{i_{k-1}}} \n \dots (\cP_{e_{i_1}} \n \cP_{e_i,-1}) \dots)
\end{equation}

\begin{Proposition}\label{prop:consistent}
The object $\scrM(\beta - \delta)$ is simple and real. Up to Koszul shifts it does not depend on the choice of minimal expression. 
\end{Proposition}
\begin{proof}
We proceed by induction on the height of $\beta$. The base case, where $\beta = \alpha_i$, is clear. 
For the induction step, suppose we have two minimal expressions
$$s_{i_k} \dots s_{i_1}(\alpha_i) = \beta =  s_{j_k} \dots s_{j_1}(\alpha_j).$$
If $i_k = j_k$ then by induction we are done. Otherwise, we have two cases. 

If $\la \alpha_{i_k}, \alpha_{j_k} \ra = 0$ then $\beta' = \beta - \alpha_{i_k} - \alpha_{j_k}$ is a root and, by induction, we have
$$\scrM(\beta-\alpha_{i_k}-\delta) \cong \cP_{e_{j_k}} \n \scrM(\beta'-\delta), \ \ \ \scrM(\beta-\alpha_{j_k}-\delta) \cong \cP_{e_{i_k}} \n \scrM(\beta'-\delta).$$
It follows that 
\begin{align*}
\cP_{e_{i_k}} \n \scrM(\beta-\alpha_{i_k}-\delta)  
&\cong \cP_{e_{i_k}} \n (\cP_{e_{j_k}} \n \scrM(\beta'-\delta))  \\
&\cong \cP_{e_{j_k}} \n (\cP_{e_{i_k}} \n \scrM(\beta'-\delta)) \\
&\cong \cP_{e_{j_k}} \n \scrM(\beta-\alpha_{j_k}-\delta).
\end{align*}
Here the second isomorphism uses that 
$$\cP_{e_{i_k}} * \cP_{e_{j_k}} * \scrM(\beta'-\delta) \cong \cP_{e_{j_k}} * \cP_{e_{i_k}} * \scrM(\beta'-\delta)$$
has a simple head since $\cP_{e_{i_k}} * \cP_{e_{j_k}} \cong \cP_{e_{j_k}} * \cP_{e_{i_k}}$ is simple, real. 
This shows that two minimal expressions produce the same $\scrM(\beta-\delta)$. 

If $\la  \alpha_{i_k}, \alpha_{j_k} \ra = -1$ then, using $\la \beta, \alpha_{i_k} \ra = 1$, we get $\la \beta-\alpha_{j_k}, \alpha_{i_k} \ra = 2$ which means 
$\beta = \alpha_{i_k} + \alpha_{j_k}$. The result then follows from Corollary \ref{cor:map2} which implies that $\cP_{e_{i_k}} \n \cP_{e_{j_k},-1}$ and $\cP_{e_{j_k}} \n \cP_{e_{i_k},-1}$ are isomorphic, up to Koszul shifts.
\end{proof}

The following is an immediate consequence of Proposition \ref{prop:consistent}. 

\begin{Corollary}\label{cor:consistent} 
For $i\in I$ we have, up to Koszul shifts, 
\begin{enumerate}[label=$\mathrm{(\alph*)}$,leftmargin=8mm]
\item $\scrM(\alpha_i-\delta) \cong \cP_{e_i,-1}$,
\item $\la \beta, \alpha_i \ra = -1 \Rightarrow\cP_{e_i} \n \scrM(\beta-\delta) \cong \scrM(\beta+\alpha_i-\delta).$
\end{enumerate}
\qed
\end{Corollary}

\begin{Proposition}\label{prop:lambdas}
If $\alpha \in \P^{++}$ then for each $\beta \in \Phi^+$ and $i \in I$ the following hold: 
\begin{align*}
\Lambda(\cP_{e_i}, \scrM(\beta - \delta)) &= |\la \beta, \alpha_i \ra|,\\
\Lambda(\scrM(\beta - \delta), \cP_{e_i}) &= -\la \beta, \alpha_i \ra,\\
\Lambda({}^\R\cP_{e_i}, \scrM(\beta - \delta)) &= \la \beta, \alpha_i \ra,\\
\Lambda(\scrM(\beta - \delta), {}^\R\cP_{e_i}) &= |\la \beta, \alpha_i \ra|,\\
\b(\cP_{e_i}, \scrM(\beta - \delta)) &= \max\{-\la \beta,\alpha_i \ra, 0\},\\
\b({}^\R\cP_{e_i}, \scrM(\beta - \delta)) &= \max\{\la \beta,\alpha_i \ra, 0 \}.
\end{align*}
\end{Proposition}
\begin{proof}
We must prove that
\begin{align*}
\Lambda(\cP_{e_j}, \scrM(\beta - \delta)) &=
\begin{cases}
1& \text{ if $\la \beta, \alpha_j \ra = \pm 1$ } \\
0 &\text{ if $\la \beta, \alpha_j \ra = 0$ } \\
2 &\text{ if $\la \beta, \alpha_j \ra = 2$ }
\end{cases} \\
\Lambda(\scrM(\beta - \delta), \cP_{e_j}) &=
\begin{cases}
\mp 1&\text{ if $\la \beta, \alpha_j \ra = \pm 1$ } \\
0 &\text{ if $\la \beta, \alpha_j \ra = 0$ } \\
-2 &\text{ if $\la \beta, \alpha_j \ra = 2$ }
\end{cases}\\
\Lambda({}^\R\cP_{e_j}, \scrM(\beta - \delta)) &=
\begin{cases}
\pm 1& \text{ if $\la \beta, \alpha_j \ra = \pm 1$} \\
0 &\text{ if $\la \beta, \alpha_j \ra = 0$} \\
2 &\text{ if $\la \beta, \alpha_j \ra = 2$}
\end{cases}\\
\Lambda(\scrM(\beta - \delta), {}^\R\cP_{e_j}) &=
\begin{cases}
1&\text{ if $\la \beta, \alpha_j \ra = \pm 1$}\\
0 &\text{ if $\la \beta, \alpha_j \ra = 0$}\\
2 &\text{ if $\la \beta, \alpha_j \ra = 2$}
\end{cases}
\end{align*}
We prove all equalities by induction on the height of $\beta$. 
The height one case corresponds to $\beta = \alpha_i$ for some $i \in I$ in which 
case $\scrM(\beta - \delta) = \cP_{e_i,-1}$ by Corollary \ref{cor:consistent}.
The top two equalities then follow from Proposition \ref{prop:B2} and the bottom two from 
Proposition \ref{prop:B3} and Corollary \ref{cor:B4}.

For the induction step, choose $i \in I$ with $\la \beta,\alpha_i \ra = -1$. 
Then Corollary \ref{cor:consistent} yields
\begin{align}\label{induction}
\scrM(\beta+\alpha_i-\delta) \ddcong \cP_{e_i} \n \scrM(\beta-\delta)
\end{align}
There are several cases to consider. 

Case 1: $j=i$. By induction, we have $\b(\cP_{e_i}\,,\, \scrM(\beta-\delta))=1$. 
Thus, Proposition \ref{prop:Lb=1} and \eqref{induction} yield
\begin{align*}
& \Lambda(\cP_{e_i}, \scrM(\beta+\alpha_i-\delta)) = \Lambda(\cP_{e_i}, \scrM(\beta-\delta)) = 1 \\
& \Lambda(\scrM(\beta+\alpha_i-\delta), \cP_{e_i}) = - \Lambda(\cP_{e_i}, \scrM(\beta-\delta)) = - 1.
\end{align*}

Case 2a: $\la \alpha_j,\alpha_i \ra = -1$ and $\la \beta, \alpha_j \ra = 0$. 
By induction, the objects $\cP_{e_j}$ and $\scrM(\beta-\delta)$ $q$-commute. 
Using Proposition \ref{prop:B2}, \ref{prop:Lambda*} this implies
$$\Lambda(\scrM(\beta+\alpha_i-\delta), \cP_{e_j}) = \Lambda(\cP_{e_i}, \cP_{e_j}) + \Lambda(\scrM(\beta-\delta), \cP_{e_j}) = 1+0=1.$$
On the other hand, by induction we know that the objects ${}^\R\cP_{e_i}$ and $\scrM(\beta-\delta)$ $q$-commute. This implies
$$\Lambda(\scrM(\beta+\alpha_i-\delta), {}^\R\cP_{e_j}) = \Lambda(\cP_{e_i},{}^\R\cP_{e_j}) + \Lambda(\scrM(\beta-\delta),{}^\R\cP_{e_j}) = 1 + 0 = 1.$$
By Proposition \ref{prop:Lduals} this means $\Lambda(\cP_{e_j}, \scrM(\beta+\alpha_i-\delta)) = 1$ as required. 

Case 2b: $\la \alpha_j, \alpha_i \ra = -1$ and $\la \beta, \alpha_j \ra = 1$. 
By induction, the objects $\cP_{e_j}$ and $\scrM(\beta-\delta)$ $q$-commute.
So again we get 
$$\Lambda(\scrM(\beta+\alpha_i-\delta), \cP_{e_j}) = \Lambda(\cP_{e_i}, \cP_{e_j}) + \Lambda(\scrM(\beta-\delta), \cP_{e_j}) = 1-1=0.$$
On the other hand, since $\la \beta, \alpha_j \ra = 1$ we have 
$$\scrM(\beta-\delta) \ddcong \cP_{e_j} \n \scrM(\beta-\alpha_j-\delta).$$ 
Thus 
$$\cP_{e_i} \n (\cP_{e_j} \n \scrM(\beta-\alpha_j-\delta)) \cong \scrM(\beta+\alpha_i-\delta).$$
Moreover, using induction, we know 
$$\Lambda(\cP_{e_i}, \cP_{e_j} \n \scrM(\beta-\alpha_j-\delta)) = 1 = \Lambda(\cP_{e_i},\cP_{e_j}) + \Lambda(\cP_{e_i},\scrM(\beta-\alpha_j-\delta)).$$
By Proposition \ref{prop:normal} this means that 
$$\cP_{e_i} \n (\cP_{e_j} \n \scrM(\beta-\alpha_j-\delta)) \cong (\cP_{e_i} \n \cP_{e_j}) \n \scrM(\beta-\alpha_j-\delta).$$
Finally, using Proposition \ref{prop:Lb=1}, we know that the objects $\cP_{e_j}$ and $\cP_{e_i} \n \cP_{e_j}$ $q$-commute. 
It follows by Proposition \ref{prop:Lambda*} that
$$\Lambda(\cP_{e_j}, \scrM(\beta+\alpha_i-\delta)) = \Lambda(\cP_{e_j}, \cP_{e_i} \n \cP_{e_j}) + \Lambda(\cP_{e_j}, \scrM(\beta-\alpha_j-\delta)) = -1 + 1 = 0$$
where we use Proposition \ref{prop:Lb=1} and induction to obtain the second equality. 

Case 2c: $\la \alpha_j,\alpha_i \ra = -1$ and $\la \beta, \alpha_j \ra = 2$. In this case $\beta = \alpha_j$ and $\scrM(\beta-\delta) = \cP_{e_j,-e_j}$ which $q$-commutes with $\cP_{e_j}$. Thus 
$$\Lambda(\scrM(\alpha_j+\alpha_i-\delta), \cP_{e_j}) = \Lambda(\cP_{e_i}, \cP_{e_j}) + \Lambda(\cP_{e_j,-e_j}, \cP_{e_j}) = 1-2=-1.$$
On the other hand, using Corollary \ref{cor:map2}, 
$$\Lambda(\cP_{e_j}, \scrM(\alpha_j+\alpha_i-\delta)) = \Lambda(\cP_{e_j}, \cP_{e_j} \n \cP_{e_i,-1}) = \Lambda(\cP_{e_j}, \cP_{e_i,-1}) = 1$$
as required.

This covers all the possible cases if $\la \alpha_j,\alpha_i \ra = -1$. 
Notice that we cannot have $\la \alpha_j,\alpha_i \ra = \la \beta,\alpha_j \ra = -1$, 
since in this case $\la \beta+\alpha_i,\alpha_j \ra = -2$ which implies $\beta+\alpha_i \not\in \Phi^+$.

Case 3: $\la \alpha_j,\alpha_i \ra = 0$. In this case the objects $\cP_{e_i}$ and $\cP_{e_j}$ $q$-commute.
Hence, we have $\Lambda(\cP_{e_j}, \cP_{e_i})=0$. So 
\begin {align*}
\Lambda(\cP_{e_j}, \scrM(\beta+\alpha_i-\delta)) &= \Lambda(\cP_{e_j}, \cP_{e_i}) + \Lambda(\cP_{e_j}, \scrM(\beta-\delta)) \\
&= \Lambda(\cP_{e_j}, \scrM(\beta-\delta))
\end{align*}
as required. 
On the other hand, by adjunction we have a map $\scrM(\beta-\delta) \to {}^\R\cP_{e_i} * \scrM(\beta+\alpha_i-\delta)$ 
which means that $\scrM(\beta+\alpha_i-\delta) \n {}^\R\cP_{e_i}$ and $\scrM(\beta-\delta)$ are isomorphic up to shifts. 
Since $\cP_{e_j}$ $q$-commutes with ${}^\R\cP_{e_i}$ this means that 
$$\Lambda(\scrM(\beta+\alpha_i-\delta), \cP_{e_j}) + \Lambda({}^\R\cP_{e_i}, \cP_{e_j}) = \Lambda(\scrM(\beta-\delta),\cP_{e_j}).$$
Since $\Lambda({}^\R\cP_{e_i}, \cP_{e_j}) = 0$ this implies 
$$\Lambda(\scrM(\beta+\alpha_i-\delta), \cP_{e_j}) = \Lambda(\scrM(\beta-\delta),\cP_{e_j}).$$
This completes the induction step for the top two equalities. Since 
$$\Lambda(\cP_{e_j}, \scrM(\beta - \delta)) = \Lambda(\scrM(\beta - \delta), {}^\R\cP_{e_j})$$
the bottom right equality follows from the top left. It remains to prove the induction step for the bottom left equality. 

If $i \ne j$ then the objects ${}^\R\cP_{e_j}$ and  $\cP_{e_i}$ $q$-commute, and hence
$$\Lambda({}^\R\cP_{e_j}, \scrM(\beta+\alpha_i-\delta)) = \Lambda({}^\R\cP_{e_j}, \cP_{e_i}) + \Lambda({}^\R\cP_{e_j}, \scrM(\beta-\delta)).$$
If $\la \alpha_i,\alpha_j \ra = 0$ then $\Lambda({}^\R\cP_{e_j}, \cP_{e_i}) = 0$ and so 
$$\Lambda({}^\R\cP_{e_j}, \scrM(\beta+\alpha_i-\delta)) = \Lambda({}^\R\cP_{e_j}, \scrM(\beta-\delta)).$$
If $\la \alpha_i,\alpha_j \ra = -1$ then $\Lambda({}^\R\cP_{e_j}, \cP_{e_i}) = -1$ and so
$$\Lambda({}^\R\cP_{e_j}, \scrM(\beta+\alpha_i-\delta)) = \Lambda({}^\R\cP_{e_j}, \scrM(\beta-\delta)) - 1.$$
It is easy to see that this implies the induction step. 

Finally, when $i=j$ it remains to show that $\Lambda({}^\R\cP_{e_i}, \scrM(\beta+\alpha_i-\delta)) = 1$. Notice that 
$$\b({}^\R\cP_{e_i}, \scrM(\beta+\alpha_i-\delta)) \le \b({}^\R\cP_{e_i},\cP_{e_i}) + \b({}^\R\cP_{e_i}, \scrM(\beta+\alpha_i)) = 1 + 0 = 1.$$
Since we already know $\Lambda(\scrM(\beta+\alpha_i-\delta), {}^\R\cP_{e_i})=1$ it suffices to prove that the inequality above is an equality. 
Equivalently, it suffices to show that ${}^\R\cP_{e_i}$ does not $q$-commute with $\scrM(\beta+\alpha_i-\delta)$. 
But if this were the case then ${}^\R\cP_{e_i} * \scrM(\beta+\alpha_i-\delta)$ would be simple and the adjunction map 
$$\scrM(\beta-\delta) \rightarrow {}^\R\cP_{e_i} * \scrM(\beta+\alpha_i-\delta)$$ would be an isomorphism. 
This is not possible since ${}^\R\cP_{e_i}$ is supported over the negative part of the affine Grassmannian while $\scrM(\beta-\delta)$ is 
supported over the positive part (because it is obtained by repeatedly taking heads of products of objects supported over the positive part).
\end{proof}

\begin{Corollary}\label{cor:real2}
If $\alpha \in \P^{++}$ and $\beta \in \Phi^+$ then the object $\scrM(\beta-\delta)$ is simple and real. 
\end{Corollary}

\begin{proof}
We proceed by induction on the height of $\beta$. The base case, when $\beta = \alpha_i$, follows from Corollary \ref{cor:real}. 
For the induction step, we may assume that $\la \beta, \alpha_i \ra = -1$.
Then Corollary \ref{cor:consistent} yields
$$\scrM(\beta+\alpha_i-\delta) \ddcong \cP_{e_i} \n \scrM(\beta-\delta)$$
 By induction both $\cP_{e_i}$ and $\scrM(\beta-\delta)$ are simple and real. 
Moreover, by Proposition \ref{prop:lambdas} we have  $\b(\cP_{e_i}, \scrM(\beta-\delta)) = 1$.
Using part (b) of Proposition \ref{prop:Lb=1} this implies that $\scrM(\beta+\alpha_i-\delta)$ is also simple and real. 
This completes the induction. 
\end{proof}

\begin{Remark}
Recall that $\beta \in \Phi^+$ and that we have choosen a minimal expression $\beta = s_{i_k} \dots s_{i_1}(\alpha_i)$.
For each $l=1,\dots, k$ we set $\beta_l = s_{i_l} \dots s_{i_1}(\alpha_i)$.
We have $(\alpha_{i_{l+1}},\beta_l)=-1$. Hence Corollary \ref{cor:consistent} and Proposition \ref{prop:lambdas} 
imply that
$$\cP_{e_{i_{l+1}}} \n \scrM(\beta_l-\delta) \cong \scrM(\beta_{l+1}-\delta),\quad l<k$$
\begin{align*}
\Lambda(\cP_{e_{i_{l+1}}}, \scrM(\beta_l - \delta)) =
\Lambda(\scrM(\beta_l - \delta), \cP_{e_{i_{l+1}}})=\b(\scrM(\beta_l - \delta), \cP_{e_{i_{l+1}}})=1
\end{align*}
By Proposition \ref{prop:homs1} we have
\begin{align*}
\scrM(\beta_{l+1}-\delta)\cong\cP_{e_{i_{l+1}}} \n \scrM(\beta_l-\delta)&\cong\scrM(\beta_l-\delta) \d \cP_{e_{i_{l+1}}}\{1\},\\
 \scrM(\beta_l-\delta)\n\cP_{e_{i_{l+1}}} &\cong \cP_{e_{i_{l+1}}}\d\scrM(\beta_l-\delta) \{1\}
\end{align*}
By Proposition \ref{prop:Lb=1} we have the exact sequences
\begin{gather*}
0\to \cP_{e_{i_{l+1}}} \d \scrM(\beta_l-\delta)\to \cP_{e_{i_{l+1}}} * \scrM(\beta_l-\delta)\to\scrM(\beta_{l+1}-\delta)\to 0,\\
0\to  \scrM(\beta_{l+1}-\delta)\{-2\}\to  \scrM(\beta_l-\delta)*\cP_{e_{i_{l+1}}}\{-1\} \to  \scrM(\beta_l-\delta)\n\cP_{e_{i_{l+1}}}\{-1\} \to 0
\end{gather*}
Hence, we have
$$(1-q^2)\,[\scrM(\beta_{l+1}-\delta)]=[\,[\cP_{e_{i_{l+1}}}]\,,\,[\scrM(\beta_l-\delta)]\,]_q$$
where $[A,B]=AB-qBA$ is the $q$-commutator.
We deduce that the K-theoretic analogue of  \eqref{eq:b-d} is
$$(1-q^2)^k\,[\scrM(\beta-\delta)]=[\,[\cP_{e_{i_k}}]\,,\,[\,[\cP_{e_{i_{k-1}}}]\,,\,\dots[\,[\cP_{e_{i_1}}]\,,\,[\cP_{e_{i,-1}}]\,]_q\dots]_q\,]_q$$
\end{Remark}

\subsection{The object $\scrM(\alpha_i)$}

For $\beta \in \Phi^+$ we set
\begin{equation}\label{eq:-b+d}
\scrM(-\beta+\delta) = {}^\L \scrM(\beta - \delta) \in \cKP_{-\beta+\delta}.
\end{equation}
Let $\theta \in \Phi^+$ be the highest root. 
We define the object $\scrM(\alpha_i)$ in $\cKP_{\alpha_i}$ to be
\begin{equation}\label{eq:Malphai}
\scrM(\alpha_i)  = 
\begin{cases}
\cP_{e_i} & \text{ if $i \in I$ } \\
\scrM(-\theta+\delta) & \text{ if $i=0$ }
\end{cases}
\end{equation}
Notice that technically $\scrM(\alpha_0)$ is only defined up to Koszul shifts. 

\begin{Proposition}\label{prop:alpha0}
Assume that $\alpha \in \P^{++}$. 
\hfill
\begin{enumerate}[label=$\mathrm{(\alph*)}$,leftmargin=8mm]
\item
The object $\scrM(\alpha_i)$ is simple and real for all $i\in I_\af$.
\item
For $i \ne j \in I_\af$ we have 
$\Lambda(\scrM(\alpha_i)\,,\,\scrM(\alpha_j)) = - \la \alpha_i, \alpha_j \ra.$
\end{enumerate}
\end{Proposition}

\begin{proof}
Part (a) follows from Corollary \ref{cor:real} and \eqref{eq:Malphai} if $i\in I$, and from 
Corollary \ref{cor:real2} and \eqref{eq:Malphai} if $i=0$.
Part (b), for $i,j \ne 0$, follows from Proposition \ref{prop:B2}. 
If $i \ne 0$, using Proposition \ref{prop:Lduals} and Proposition \ref{prop:lambdas}, we get
\begin{align*}
\Lambda(\scrM(\alpha_i)\,,\, \scrM(\alpha_0)) &= \Lambda(\scrM(\alpha_i)\,,\, {}^\L \scrM(\theta-\delta)) 
 =\Lambda({}^\R\scrM(\alpha_i)\,,\,  \scrM(\theta-\delta)) 
 = \la \theta, \alpha_i \ra 
 = - \la \alpha_i,\alpha_0 \ra
 \end{align*}
and similarly $\Lambda(\scrM(\alpha_0),\scrM(\alpha_i)) = \Lambda(\scrM(\theta-\delta), {}^\R\scrM(\alpha_i)) = -\la \alpha_0,\alpha_i \ra$.
\end{proof}

\section{The functor $\F$}\label{sec:Fprops}

Unless specified otherwise we assume that the dimension  vector $\alpha$ in \eqref{alpha} is strictly dominant, i.e., 
we have $\alpha \in \P^{++}$. We consider the Koszul perverse category $\cKP_{X^{(\bullet)}}$ in \S\ref{sec:twisted product}. 
From Proposition \ref{prop:alpha0} we have simple, real objects $\{\scrM(\alpha_i)\,;\, i \in I_\af\}$ satisfying $\La(\scrM(\alpha_i)\,,\,\scrM(\alpha_j)) = - \la \alpha_i, \alpha_j \ra$ for $i\neq j$.  This duality datum induces a monoidal functor $\F: \gMod(R) \to \calD^-_{X^{(\bullet)}}$ as in \eqref{eq:FMod}.  One of the main goals of this section is Proposition \ref{prop:Fexact} which proves that $\F$ is t-exact, giving us an exact monoidal functor 
\begin{equation}\label{eq:F2}
\F: \gmod(R)^\heartsuit \to \cKP_{X^{(\bullet)}}.
\end{equation}
If the duality datum was of finite type, so that $R = \oplus_\beta R(\beta)$ is of finite type, then this t-exactness would follow from the fact that the $R(\beta)$ have finite homological dimension, see \cite[Thm.~3.8]{KKK}. However, in the affine case these $R(\beta)$ need not have finite homological dimension and the boundedness of $\F$ is difficult to prove. 

This section is organized as follows. Section  \ref{sec:preliminaries} contains some preliminaries which are used to identify the images $\F(L(-\alpha_i+\delta))$ in Lemma \ref{lem:F2} and $\F(L(\delta_i))$ in Proposition \ref{prop:delta5}. This second identification is the most technically difficult step. This allows us to prove that $\F$ is t-exact in Proposition \ref{prop:Fexact}.  Finally, we show that $\F$ takes a simple module to either zero or a simple object in Proposition \ref{prop:simple2simple}.

\subsection{Some preliminaries}\label{sec:preliminaries}

\begin{Lemma}\label{lem:F1}
Suppose $L,L' \in \gmod(R)^\heartsuit$ are simple with at least one real and that $\scrM = \F(L)$, $\scrM' = \F(L')\in\cKP$ are also simple with 
at least one real. Then there are two cases
\begin{enumerate}[label=$\mathrm{(\alph*)}$,leftmargin=8mm]
\item $\La(L,L') = \La(\scrM,\scrM')$ with $\F(\r_{L,L'}) \in\bbC^\times \cdot \r_{\scrM,\scrM'}$ or
\item $\La(L,L') > \La(\scrM,\scrM')$ with $\F(\r_{L,L'}) = 0$.
\end{enumerate}
\end{Lemma}
\begin{proof}
By Corollary \ref{cor:FLz} we have $\F(L_z) \cong \scrM_z$ and $\F(L'_z) \cong \scrM'_z$.
Using Proposition \ref{prop:homs2} we know that
$$\F(\bfR_{L_z,L'}^{\ren}) = a z^c \bfR_{\scrM_z,\scrM'}^{\ren}
,\quad
 \F(\bfR_{L',L_z}^{\ren}) = b z^d \bfR_{M',M_z}^{\ren}$$
for some scalars $a,b \in \C$ and $c,d \in \N$, and that
$$\bfR_{L',L_z}^{\ren} \circ \bfR_{L_z,L'}^{\ren} = z^{\b(L,L')} \cdot \id_{L_z \circ L'} 
,\quad
 \bfR_{\scrM',\scrM_z}^{\ren} \circ \bfR_{\scrM_z,\scrM'}^{\ren} = z^{\b(\scrM,\scrM')} \cdot \id_{\scrM_z \circ \scrM'}.$$
This implies that $a,b\neq 0$ and $\b(L,L') = \b(\scrM,\scrM')+c+d$. 
Thus, either $\b(L,L') = \b(\scrM,\scrM')$ with $c=d=0$, yielding the first case, 
or $\b(L,L') > \b(\scrM,\scrM')$ with either $c > 0$ or $d > 0$ yielding the second case. 
\end{proof}

\begin{Lemma}\label{lem:F2}
Suppose $L,L' \in \gmod(R)^\heartsuit$ are simple with at least one real, and that 
$\scrM = \F(L)$, $\scrM' = \F(L') \in \cKP$ are also simple with at least one real. If $\b(L,L') = \b(\scrM,\scrM')$ then for any submodule 
$S \subset L \circ L'$ the induced map 
$H^0_{\cKP}(\F(S)) \to \scrM \circ \scrM'$
is nonzero and its image contains $\scrM \d \scrM'$. 
\end{Lemma}

\begin{proof}
By Proposition \ref{prop:lsubquotient}, the module $L \d L'$ is simple. So we have sequence of inclusions 
$$L \d L' \to S \to L \circ L'.$$
Applying $\F$ to this sequence it suffices to show that the induced composition
$$H^0_{\cKP}(\F(L \d L')) \to H^0_{\cKP}(\F(S)) \to \F(L \circ L') \cong \scrM \circ \scrM'$$
is nonzero and that its image contains $\scrM \d \scrM'$. 
To do this consider the composition 
$$L' \circ L \to L' \n L \dcong L \d L' \to L \circ L'.$$
Up to shifts, the composition is equal to $\r_{L',L}$. Applying $\F$ we get 
\begin{equation}\label{eq:local9}
\scrM' \circ \scrM \to H^0_{\cKP}(\F(L \n L')) \dcong H^0_{\cKP}(\F(L \d L')) \to \scrM \circ \scrM'
\end{equation}
which, by Lemma \ref{lem:F1}, is (up to scaling) equal to $\r_{\scrM', \scrM}$. Moreover, the left map \eqref{eq:local9} is surjective. This means that the image of the right map in \eqref{eq:local9} is the same as the image of $\r_{\scrM',\scrM}$ which is $\scrM \d \scrM'$. 
\end{proof}

\begin{Corollary}\label{cor:F1}
Suppose $L,L' \in \gmod(R)^\heartsuit$ are simple with at least one real, $\scrM = \F(L)$, $\scrM' = \F(L') \in \cKP$ are also simple with at least one real. Suppose $L \circ L'$ and $\scrM * \scrM'$ have length $2$. Then $\b(L,L') \ge \b(\scrM,\scrM')$ and, if equality holds, we have $\F(L \n L') \cong \scrM \n \scrM'$ and $\F(L \d L') \cong \scrM \d \scrM'$. 
\end{Corollary}

\begin{proof} 
By Lemma \ref{lem:F1} we have $\b(L,L') \ge \b(\scrM,\scrM')$. 
Suppose now $\b(L,L') = \b(\scrM,\scrM')$. 
Applying Lemma \ref{lem:F2} with $S = L \d L' \subset L \circ L'$ we get that the image of 
\begin{align}\label{F(i)} 
H^0_{\cKP}(\F(L \d L')) \to \F(L \circ L') \cong \scrM \circ \scrM'
\end{align}
contains $\scrM \d \scrM'$. In particular, $H^0_{\cKP}(\F(L \d L')) \ne 0$. 
Interchanging $L$ and $L'$ we also get that $H^0_{\cKP}(\F(L' \d L)) \ne 0$. 
Since $L \n L' \dcong L' \d L$ this means that $H^0_{\cKP}(\F(L' \n L)) \ne 0$.

But now apply $\F$ to the short exact sequence
$$0 \to L \d L' \xrightarrow{i} L \circ L' \xrightarrow{p} L \n L' \to 0$$
Since $\F$ is right t-exact, we get a long exact sequence
\begin{align}\label{eq:LES}
\dots \to H^0_{\cKP}(\F(L \d L')) \xrightarrow{\F(i)}  \scrM \circ \scrM' \xrightarrow{\F(p)} H^0_{\cKP}(\F(L \n L')) \to 0
\end{align}
where the left and right terms are nonzero and $\F(i)$ is nonzero. Since $\scrM \circ \scrM'$ has length $2$ this means that 
$H^0_{\cKP}(\F(L \d L')) \cong \scrM \d \scrM'$, $H^0_{\cKP}(\F(L \n L')) \cong \scrM \n \scrM'$ and moreover $\F(i)$ must be injective. 
Since $H^n_{\cKP}(\scrM \circ \scrM') = 0$ for $n \ne 0$ it follows from \eqref{eq:LES} that 
$$H^n_{\cKP}(\F(L \d L')) = 0 = H^n_{\cKP}(\F(L \n L'))$$
for $n \ne 0$. The result follows. 
\end{proof}

\subsection{The image of $L(-\alpha_i+\delta)$}

Hereforth, we fix a balanced convex order as in Remark \ref{rem:minimalpair}.  The precise properties of this preorder are only used in the proof of Lemma \ref{lem:F4}. However, since this lemma is frequently used in subsequent sections we must require that our preorder satisfy this condition going forward. Recall that for each $\beta \in \bbN\Phi_\af^+$ the simple $R(\beta)$-modules are $L(\pi)$ where $\pi\in\Pi(\beta)$ is a root partition.
Further, by definition of the functor $\F$, for all $i\in I_\af$ we have 
\begin{align}\label{eq:Falphai1}\F(L(\alpha_i)) \cong \scrM(\alpha_i).\end{align}
where $\scrM(\alpha_i)$ is as in \eqref{eq:Malphai}. In particular, if $i\in I$ then
\begin{align}\label{eq:Falphai2}\F(L(\alpha_i)) \cong \scrM(\alpha_i)=\calP_{e_i}.\end{align}

\begin{Lemma}\label{lem:F4}
If $\beta \in \Phi^+$ then $\F(L(-\beta + \delta)) \ddcong \scrM(-\beta+\delta)$. 
Hence, we have $\F(L(-\alpha_i+\delta)) \ddcong {}^\L \cP_{e_i,-1}$.
\end{Lemma}

\begin{proof}
We prove this by reverse induction on the height of $\beta$.  The base case $\beta = \theta$ holds because $\F(L(\alpha_0)) \cong \scrM(\alpha_0)$.  Assume that $\beta \ne \theta$.  By Proposition \ref{prop:minimalpair} there exists $i \in I$ such that $(\alpha_i,-\beta-\alpha_i+\delta)$ is a minimal pair and $\la \beta, \alpha_i \ra = -1$.  From \eqref{eq:Lminpair} we find that $L(\alpha_i) \circ L(-\beta-\alpha_i+\delta)$ has length $2$.  Using \eqref{eq:Falphai2}, by induction we get
$$\F(L(-\beta-\alpha_i+\delta)) \ddcong \scrM(-\beta-\alpha_i+\delta).$$
On the other hand, Propositions \ref{prop:Lduals} and \ref{prop:lambdas} yield
\begin{align*}
& \Lambda(\scrM(-\beta-\alpha_i+\delta)\,,\,\scrM(\alpha_i)) = \Lambda(\scrM(\alpha_i)\,,\, \scrM(\beta+\alpha_i-\delta)) = 1 \\
& \Lambda(\scrM(\alpha_i)\,,\, \scrM(-\beta-\alpha_i+\delta)) = \Lambda(\scrM(\beta+\alpha_i-\delta)\,,\, \scrM(\alpha_i)) = 1
\end{align*}
Thus we have
$$\b(\scrM(\alpha_i),\scrM(-\beta-\alpha_i+\delta)) = 1.$$
This means that $\scrM(\alpha_i)*\scrM(-\beta-\alpha_i+\delta)$ has length $2$ by Proposition \ref{prop:Lb=1}.
Moreover, by Corollary \ref{cor:consistent} and \eqref{eq:Malphai} we have 
\begin{align*} 
\scrM(\alpha_i) \n \scrM(\beta-\delta) \ddcong \scrM(\beta+\alpha_i-\delta).
\end{align*}
Using Proposition \ref{prop:Lduals} and \eqref{eq:-b+d}, we deduce by adjunction that 
$$\scrM(-\beta-\alpha_i+\delta) \n \scrM(\alpha_i) \ddcong \scrM(-\beta+\delta).$$
Similarly, since $(\alpha_i,-\beta-\alpha_i+\delta)$ is a minimal pair, we have
$$L(-\beta-\alpha_i+\delta) \n L(\alpha_i) \cong L(-\beta+\delta).$$
It follows that 
\begin{align*}
\F(L(-\beta+\delta)) 
&\cong \F(L(-\beta-\alpha_i+\delta) \n L(\alpha_i)) \\
&\cong \F(L(-\beta-\alpha_i+\delta)) \n \F(L(\alpha_i)) \\
&\ddcong \scrM(-\beta-\alpha_i+\delta) \n \scrM(\alpha_i)\\
&\ddcong \scrM(-\beta+\delta)
\end{align*}
where the second isomorphism follows from Corollary \ref{cor:F1}. This completes the induction step.
\end{proof}

\subsection{The image of $L(\delta_i)$}

The most difficult part of proving that $\F$ is t-exact is showing that $\F(L(\delta_i)) \ddcong \cV_i$. We achieve this in a series of steps culminating in Proposition \ref{prop:delta5}. When discussing composition factors we will ignore loop and Koszul shifts. 

\begin{Lemma}\label{lem:delta1}
For $i \in I$ the composition factors of $H^0_{\cKP}(\F(L(\delta_i)))$ is a subset of 
\begin{equation}\label{eq:local8}
\{\cV_i, \cV_j; j \leftrightarrow i\}
\end{equation}
which contains $\cV_i$ while those of $H^0_{\cKP}(\F(L(\alpha_i, -\alpha_i+\delta)))$ is a subset of 
\begin{equation}
\{\cK_i, \cV_j; j \leftrightarrow i\}.
\end{equation}
which contains $\cK_i$. 
\end{Lemma}

\begin{proof}
Consider the short exact sequence 
\begin{align}\label{exact-sequence1}0 \to \Ker \xrightarrow{h} L(-\alpha_i+\delta) \circ L(\alpha_i) \to L(\delta_i) \to 0
\end{align} 
where $\Ker$ is by definition the kernel. 
By Lemma \ref{lem:F4} and \eqref{eq:Falphai2}, 
we have $\F(L(-\alpha_i + \delta) \circ L(\alpha_i)) \ddcong {}^\L \cP_{e_i,-1} \circ \cP_{e_i}$.
Since the functor $\F$ is right t-exact, applying $\F$ gives an exact sequence 
$$H^0_{\cKP}(\F(\Ker)) \to {}^\L \cP_{e_i,-1} \circ \cP_{e_i} \to H^0_{\cKP}(\F(L(\delta_i))) \to 0.$$
By Lemma \ref{lem:F2} the image of the first map contains at least $\cK_i \{-2\} = {}^\L \cP_{e_i,-1} \d \cP_{e_i}$. 
Since the second map above is surjective it follows from Proposition \ref{prop:ses4} that the composition factors of 
$H^0_{\cKP}(\F(L(\delta_i)))$ is a subset of \eqref{eq:local8}.

On the other hand, if we consider $L(\delta_i) \{-2\} \subset L(\alpha_i) \circ L(-\alpha_i + \delta)$ then by 
Lemma \ref{lem:F2} the image of the induced map 
$$H^0_{\cKP}(\F(L(\delta_i) \{-2\})) \to \F(L(\alpha_i) \circ L(-\alpha_i + \delta)) \cong \cP_{e_i} \circ {}^\L \cP_{e_i,-1}$$
contains $\cV_i \{-2\} \cong \cP_{e_i} \d {}^\L \cP_{e_i,-1}$. This means that $\cV_i$ must appear as a subquotient of 
$H^0_{\cKP}(\F(L(\delta_i)))$. This proves the first claim. The second claim follows similarly with $L(\alpha_i, -\alpha_i+\delta)$ 
in place of $L(\delta_i)$. 
\end{proof}

\begin{Lemma}\label{lem:delta2}
For $i \in I$ we have $H^0_{\cKP}(\F(L(\delta_i))) \ddcong \cV_i$. 
\end{Lemma}

\begin{proof}
Since $L(\delta_{ij}) \subset L(\alpha_i) \circ L(-\alpha_i+\delta)$ by Corollary \ref{cor:gamma}, 
it follows from Lemma \ref{lem:F2} applied to $S=L(\delta_{ij})$, Lemma \ref{lem:F4} and \eqref{eq:Falphai2} that 
$\cV_i \dcong \cP_{e_i} \d {}^\L \cP_{e_i,-1}$ is a composition factor of $H^0_{\cKP}(\F(L(\delta_{ij})))$.
On the other hand, applying the right exact functor $\F$ to the surjective map 
$L(-\alpha_j+\delta) \circ L(\alpha_j) \to L(\delta_{ij})$ gives a surjective map 
$${}^\L \cP_{e_j,-1} \circ \cP_{e_j} \to H^0_{\cKP}(\F(L(\delta_{ij}))).$$
It follows from Proposition \ref{prop:ses4} that the composition factors of $H^0_{\cKP}(\F(L(\delta_{ij})))$ are contained in 
$$\{\cK_j, \cV_j, \cV_k\,;\,k \leftrightarrow j\}.$$
Finally, applying $\F$ to the short exact sequence 
$$0 \to L(\delta_i) \{-1\} \to L(\delta_{ij}) \to L(\delta_j) \to 0.$$
gives an exact sequence
$$H^0_{\cKP}(\F(L(\delta_i)))\{-1\} \to H^0_{\cKP}(\F(L(\delta_{ij}))) \to H^0_{\cKP}(\F(L(\delta_j))) \to 0.$$
From the above discussion, using Lemma \ref{lem:F2} with $S=L(\delta_i)\{-1\}$, we deduce that
 the image of the first map contains $\cV_i$ as a subquotient.
It follows that the composition factors of $H^0_{\cKP}(\F(L(\delta_j)))$ cannot contain $\cV_i$. Since this is true for any $i \leftrightarrow j$ it 
follows by Lemma \ref{lem:delta1} that the composition factors of $H^0_{\cKP}(\F(L(\delta_j)))$ is a subset of $\{\cV_j\}$. The result follows. 
\end{proof}

Recall that for any $i,j \in I$ with $\la \alpha_i, \alpha_j \ra = -1$ we have nonzero maps 
\begin{align*}
\gamma_{ji}:\,& L(-\alpha_j+\delta) \circ L(\alpha_j) \to L(\alpha_i) \circ L(-\alpha_i + \delta) \{1\} \\
\Gamma_{ji}:\,& {}^\L \cP_{e_j,-1} \circ \cP_{e_j} \to \cP_{e_i} \circ {}^\L \cP_{e_i,-1} [1] \la -1 \ra \{1\}
\end{align*}
from \eqref{eq:gamma} and \eqref{eq:dimhom} respectively. 

\begin{Lemma}\label{lem:delta3}
Under the isomorphisms
$$\F(L(-\alpha_j+\delta) \circ L(\alpha_j)) \ddcong {}^\L \cP_{e_j,-1} \circ \cP_{e_j} \text{ and } 
\F(L(\alpha_i) \circ L(-\alpha_i + \delta)) \ddcong \cP_{e_i} \circ {}^\L \cP_{e_i,-1}$$
induced by Lemma \ref{lem:F4} the maps $\F(\gamma_{ji})$ and $\Gamma_{ji}$ agree up to scaling.  
\end{Lemma}

\begin{proof}
Using \eqref{eq:dimhom} it suffices to show that $\F(\gamma_{ji})$ is nonzero. Recall that $\gamma_{ij}$ factors as
\begin{equation}\label{eq:local5}
L(-\alpha_j+\delta) \circ L(\alpha_j) \to L(\delta_{ij}) \to L(\alpha_i) \circ L(-\alpha_i+\delta) \{2\}.
\end{equation}
Applying $\F$ we get the sequence of maps 
\begin{equation}\label{eq:local7}
{}^\L \cP_{e_j,-1} \circ \cP_{e_j} \to H^0_{\cKP}(\F(L(\delta_{ij}))) \to \cP_{e_i} \circ {}^\L \cP_{e_i,-1} \{2\}.
\end{equation}
The first map above is surjective (since $\F$ is right t-exact) so it remains to show that the second map in \eqref{eq:local7} is nonzero. To do this consider the composition 
$$L(-\alpha_i+\delta) \circ L(\alpha_i) \to L(\delta_i) \to L(\delta_{ij}) \to L(\alpha_i) \circ L(-\alpha_i+\delta) \{2\}$$
which is just the map $\r$. Since $\b(L(-\alpha_i+\delta), L(\delta_i)) = 2 = \b({}^\L \cP_{e_i,-1}, \cP_{e_i})$ it follows by Lemma \ref{lem:F1} that $\F(\r) \ne 0$. Thus the second map in \eqref{eq:local7} must also be nonzero. 
\end{proof}

\begin{Corollary}\label{cor:delta4}
For $i \in I$ we have 
$$H^n_{\cKP}(\F(L(\alpha_i,-\alpha_i+\delta))) \ddcong
\begin{cases}
\cK_i & \text{ if $n = 0$ } \\
0 & \text{ if $n = -1$ }
\end{cases}$$
\end{Corollary}
\begin{proof}
We omit the loop and Koszul shifts for convenience. Consider the exact short exact sequence 
$$0 \to \Ker \xrightarrow{h} L(\alpha_i) \circ L(-\alpha_i+\delta) \to L(\alpha_i,-\alpha_i+\delta) \to 0$$
where $\Ker$ is by definition the kernel. Applying $\F$ we get a long exact sequence
$$0 \to H^{-1}_{\cKP}(\F(L(\alpha_i,-\alpha_i+\delta))) \to H^0_{\cKP}(\F(\Ker)) \xrightarrow{\F(h)} \cP_{e_i} \circ {}^\L \cP_{e_i,-1} \to H^0_{\cKP}(\F(L(\alpha_i,-\alpha_i+\delta))) \to 0.$$
By Lemma \ref{lem:F2} the image of $\F(h)$ contains $\cV_i$ as a subquotient. 

On the other hand, for any $j \leftrightarrow i$ the map $\gamma_{ij}$ factors as in \eqref{eq:local5} where the first map is surjective and the second injective. Applying $\F$ we get 
$${}^\L \cP_{e_j,-1} \circ \cP_{e_j} \to H^0_{\cKP}(\F(L(\delta_{ij}))) \to \cP_{e_i} \circ {}^\L \cP_{e_i,-1}$$
where the first map is surjective. Since $\F(\gamma_{ji})$ and $\Gamma_{ji}$ agree up to scaling it follows that the image of the second map 
in \eqref{eq:local5} is the same as the image of $\Gamma_{ji}$. 

Moreover, the map $L(\delta_{ij}) \to L(\alpha_i) \circ L(-\alpha_i+\delta)$ factors through $h$. This means that the image of $\F(h)$ contains 
the image of $\Gamma_{ji}$ and thus also contains $\cV_j$ as a subquotient. Thus the image of $\F(h)$ has composition factors which include 
$\{\cV_j: j \leftrightarrow i\}$. By Proposition \ref{prop:ses4} and Lemma \ref{lem:delta1} it follows that $H^0_{\cKP}(\F(L(\alpha_i,-\alpha_i+\delta))) \ddcong \cK_i$. 

Finally, $\Ker$ has composition factors $\{L(\delta_i), L(\delta_j)\,;\,j \leftrightarrow i\}$. It follows by Lemma \ref{lem:delta2} and induction on 
length that the composition factors of $\F(\Ker)$ are contained in $\{\cV_i, \cV_j\,;\, j \leftrightarrow i\}$. Since, by Proposition \ref{prop:ses4}, the 
composition factors of $\cP_{e_i} \circ {}^\L \cP_{e_i,-1}$ are $\{\cK_i, \cV_i, \cV_j\,;\, j \leftrightarrow i\}$ this implies that $\F(h)$ is injective. This 
in turn means that $H^{-1}_{\cKP}(\F(L(\alpha_i,-\alpha_i+\delta))) = 0$ which completes the proof. 
\end{proof}

\begin{Proposition}\label{prop:delta5}
For $i \in I$ we have $\F(L(\delta_i)) \ddcong \cV_i$ and $\F(L(\alpha_i,-\alpha_i+\delta)) \ddcong \cK_i$. 
\end{Proposition}
\begin{proof}
As usual we ignore loop and Koszul shifts. We first show that $H^{-1}_{\cKP}(\F(L(\delta_i))) = 0$ by arguing like in the proof of Corollary \ref{cor:delta4}. 
To do so, we consider the short exact sequence \eqref{exact-sequence1}.
Since the functor $\F$ is right exact, we have the long exact sequence 
$$0 \to H^{-1}_{\cKP}(\F(L(\delta_i))) \to H^0_{\cKP}(\F(\Ker)) \xrightarrow{\F(h)} {}^\L \cP_{e_i,-1} \circ \cP_{e_i} 
\to H^0_{\cKP}(\F(L(\delta_i))) \to 0.$$
By \eqref{eq:Ldelta2} the composition factors of $\Ker$ are
$$\{L(\alpha_i,-\alpha_i+\delta), L(\delta_j)\,;\,j \leftrightarrow i\}.$$
So by Lemma \ref{lem:delta2} and Corollary \ref{cor:delta4} the composition factors of 
$H^0_{\cKP}(\F(\Ker))$ is a subset of 
$$\{\cK_i, \cV_j\,;\,j \leftrightarrow i\}.$$
Since, by Proposition \ref{prop:ses4}, the composition factors of ${}^\L \cP_{e_i,-1} \circ \cP_{e_i}$ are
$$\{\cK_i, \cV_i, \cV_j\,;\,j \leftrightarrow i\}$$
Finally, by Lemma \ref{lem:delta2} again we have $H^0_{\cKP}(\F(L(\delta_i))) \ddcong \cV_i$.  Hence the map $\F(h)$ is injective, and the long exact sequence above implies that $H^{-1}_{\cKP}(\F(L(\delta_i)))=0$. 

To finish the argument we need to show that for $i \in I$ and $n \le -2$ we have
$$H^n_{\cKP}(\F(L(\delta_i))) = 0 = H^n_{\cKP}(\F(L(\alpha_i,-\alpha_i+\delta)))$$
Assuming this is not the case, let $m < 0$ be the largest integer such that one of these quantities is nonzero. 
Suppose it is $H^m_{\cKP}(\F(L(\delta_i))) \ne 0$. The other cases are similar. 
Consider again the short exact sequence \eqref{exact-sequence1}.
Applying $\F$ we get a long exact sequence 
$$\dots \to 0 \to H^m_{\cKP}(\F(L(\delta_i))) \to H^{m+1}_{\cKP}(\F(\Ker)) \to \dots$$
By assumption $H^{m+1}_{\cKP}(\F(\Ker)) = 0$ giving $H^m_{\cKP}(\F(L(\delta_i)))=0$ which is a contradiction. 
\end{proof}

\subsection{Exactness of $\F$}

\begin{Lemma}\label{lem:F3}
For $i \in I$ we have 
\begin{enumerate}[label=$\mathrm{(\alph*)}$,leftmargin=8mm]
\item $\F(L(\alpha_i + n \delta)) \ddcong \cP_{e_i,n}$ for  $n \ge 0$,
\item $\F(L(-\alpha_i + n \delta)) \ddcong {}^\L\cP_{e_i,-n}$ for  $n \ge 1$.
\end{enumerate}
\end{Lemma}

\begin{proof}
We prove (a) by induction on $n$. The base case $n=0$ follows from \eqref{eq:Falphai2}.
For $n\geqslant 0$ we know by \eqref{eq:Lal1} that $L(\delta_i) \circ L(\alpha_i+n\delta)$ has length $2$, $\b(L(\delta_i), L(\alpha_i+n\delta))=1$ and 
$$L(\delta_i) \n L(\alpha_i+n\delta) \dcong L(\alpha_i+(n+1)\delta).$$
By induction $\F(L(\alpha_i+n\delta)) \cong \cP_{e_i, n}$ and from the above we have 
$\F(L(\delta_i)) \cong \cV_i$.
Moreover, using Lemma \ref{lem:V}, we know that $\cV_i * \cP_{e_i,n}$ has length $2$, $\b(\cV_i, \cP_{e_i,n})=1$ and 
$$\cV_i \n \cP_{e_i,n} \dcong \cP_{e_i,n+1}.$$
It follows that 
\begin{align*}
\F(L(\alpha_i+(n+1)\delta)) 
&\dcong \F(L(\delta_i) \n L(\alpha_i+n\delta)) \\
&\cong \F(L(\delta_i)) \n \F(L(\alpha_i+n\delta)) \\
&\ddcong \cV_i \n \cP_{e_i, n} \\
&\dcong \cP_{e_i, n+1} 
\end{align*}
where the second isomorphism uses Corollary \ref{cor:F1}. This proves (a).  

Claim (b) follows similarly, with the base case of the induction $\F(L(-\alpha_i+\delta)) \ddcong {}^\L \cP_{e_i,-1}$ having been proven in Lemma \ref{lem:F4}. 
\end{proof}

\begin{Corollary}\label{cor:F2}
For every $\rho \in \Phi_\re^+$ we have $\F(L(\rho)) \in \cKP$.  It is either simple and real or zero.
\end{Corollary}

\begin{proof}
If $\rho \in \Phi_\re^+$ has no real mimimal pair then $\rho$ is of the form $\pm \alpha_i + n \delta$ by Proposition \ref{prop:minimalpair}, and the result follows from Lemma \ref{lem:F3}. 
Otherwise, suppose $(\beta, \gamma)$ is a real minimal pair for $\rho$. 
We have $\b(L(\beta)\,,\,L(\gamma))=1$ by \eqref{eq:Lminpair} and Proposition \ref{prop:bselfdual}.
By induction on height we can assume $\F(L(\beta)), \F(L(\gamma)) \in \cKP$ and either they are both simple real, or 
one of them is simple real and the other is zero. In the latter case clearly $\F(L(\rho))=0$. 

In the former case, by Lemma \ref{lem:F4}, either $\b(\F(L(\gamma)), \F(L(\beta))) < \b(L(\beta)\,,\,L(\gamma))=1$ or we have equality. In the first case $\F(L(\gamma)) \circ \F(L(\beta))$ is simple, real and either $F(L(\rho)) = 0$ or $F(L(\rho)) \cong \F(L(\gamma)) \circ \F(L(\beta))$. 

In the second case $\b(\F(L(\gamma)), \F(L(\beta))) = \b(L(\beta)\,,\,L(\gamma))$ and 
$$\F(L(\rho)) \cong \F(L(\gamma) \n L(\beta)) \cong  \F(L(\gamma)) \n \F(L(\beta)) \in \cKP$$
where the second isomorphism is by Corollary \ref{cor:F1}.  In this case, by Propositions \ref{prop:Lb=1} and \ref{prop:qcom}, $\F(L(\rho))$ is simple and real.  This completes the induction step.  
\end{proof}

\begin{Proposition}\label{prop:Fexact}
The functor $\F$ is t-exact.
\end{Proposition}
\begin{proof}
To show t-exactness it suffices to show that $\F(L(\pi)) \in \cKP$ for every root partition $\pi$. 
We proceed by induction on the length of $\pi$. If $\rho \in \Phi_\re^+$ then $\F(L(\rho)) \in \cKP$ by Corollary \ref{cor:F2}. 
If $\mu$ is an $\ell$-multipartition then $L(\mu)$ is a direct summand of a product of $L(\delta_i)$ by \S\ref{sec:quiverHecke}. 
Since $\F(L(\delta_i)) \in \cKP$ by Lemma \ref{lem:F3} it follows that $\F(L(\mu)) \in \cKP$. This proves the base case. 
For a general root partition $\pi$, Proposition \ref{prop:hereditary} yields short exact sequences 
\begin{align}
\label{eq:ses1} & 0 \to K \to \bd(\pi) \to L(\pi) \to 0 \\
\label{eq:ses2} & 0 \to L(\pi) \to \bn(\pi) \{s\} \to K' \to 0 
\end{align} 
where $s \in \Z$ and $K$ and $K'$ both have simple subquotients $L(\pi')$ with $\pi' < \pi$. By induction we can assume that 
$\F(K), \F(K') \in \cKP$. By Corollary \ref{cor:F2} and Proposition \ref{prop:delta5} we know that $\F(\bd(\pi)), \F(\bn(\pi)) \in \cKP$. 
It follows from (\ref{eq:ses1}) that $\F(L(\pi))$ is 
supported in t-homological degrees $[-1,0]$ while from (\ref{eq:ses2}) the object
$\F(L(\pi))$ is supported in degrees $[0,1]$. Thus $\F(L(\pi)) \in \cKP$ 
which completes the induction. 
\end{proof}

\section{Further properties of $\F$}\label{sec:Fmore}

\subsection{Simples to simples}

Recall the simple objects $L(\mu)$ indexed by $\ell$-multipartition $\mu$. Such an object is obtained as the image of an idempotent $P(\mu)$ acting on an appropriate product of $L(\delta_i)$'s. 

On the other hand, recall from \eqref{eq:V-phi-psi} that $\cV_i$ is a vector bundle over $\calR_\u0^\cl$ with fibers $(L_0^{(i)}/tL_0^{(i)})$ (up to shifts). Moreover, the object $(\cV_i)^{*n}$ is just the $n$th tensor product of these bundles so there is a natural surjective map $\C[S_n] \to \End((\cV_i)^{*n})$. We denote by $\cV(\mu)$ the image of $P(\mu)$ acting on an appropriate product of $\cV_i$'s (the analogue of $L(\mu)$). Note that many of these $\cV(\mu)$ will be zero.

\begin{Lemma}\label{lem:F5}
For every $\ell$-multipartition $\mu$ we have $\F(L(\mu)) \ddcong \cV(\mu)$. 
In particular the object $\F(L(\mu))$ is either simple or zero. 
\end{Lemma}

\begin{proof}
Since, by Proposition \ref{prop:delta5}, we have $\F(L(\delta_i)) \cong \cV_i$, it suffices to show that $\F$ identifies the actions of $\frakS_n$ on 
$L(\delta_i^{\otimes n})$ and $\cV_i^{*n}$. More precisely, if $s \in \End(L(\delta_i)^{\circ 2})$ is the generator of the $\frakS_2$ action, we need 
to show that $\F(s) = \pm s'$ where $s' \in \End((\cV_i)^{*2})$ is the generator of the other $\frakS_2$-action. 
Since $s^2=1$ it follows that $\F(s)$ is either $\pm s'$ or $\pm 1$. It remains to rule out the second possibility. 

Let us denote the two summands of $L(\delta_i)^{\circ 2}$, namely the images of $s+1$ and $s-1$, by $L(2)$ and $L(1^2)$.  If $F(s) = \pm 1$ then one of these summands is mapped to zero and, since $\F(L(\delta_i)^{\circ 2}) \cong (\cV_i)^{*2}$, the other summand is mapped to $(\cV_i)^{*2}$. In particular

Now consider the product $L(-\alpha_i + \delta) \circ L(\alpha_i) \circ L(\delta_i)$. 
Since $(\alpha_i, \delta_i)$ is a minimal pair we have a short exact sequence 
$$0 \to  L_1 \to L(\alpha_i) \circ L(\delta_i) \to L_2 \to 0$$
where $L_1$ and $L_2$ are simple. Since $L(-\alpha_i+\delta)$ is simple and real it follows that $L(-\alpha_i+\delta) \n L_1$ and 
$L(-\alpha_i+\delta) \n L_2$ are simple by Propositions \ref{prop:qcom}, \ref{prop:lsubquotient}. 
Subsequently, the head of $L(-\alpha_i + \delta) \circ L(\alpha_i) \circ L(\delta_i)$ 
has at most two summands. On the other hand, since $L(-\alpha_i + \delta) \n L(\alpha_i) \cong L(\delta_i)$, 
it follows that the head of $L(-\alpha_i + \delta) \circ L(\alpha_i) \circ L(\delta_i)$ must be $L(\delta_i)^{\circ 2}$, with 
$$\{L(-\alpha_i+\delta) \n L_1\,,\,L(-\alpha_i+\delta) \n L_2 \} = \{L(2)\,,\, L(1^2)\}.$$
But, by Lemmas \ref{lem:F1}, \ref{lem:F4} and Proposition \ref{prop:homs1}, 
the objects $\F(L(-\alpha_i+\delta) \n L_1)$ and $\F(L(-\alpha_i+\delta) \n L_2)$ are either zero or simple. 
It follows that $\F(L(2))$ and $\F(L(1^2))$ must be either zero or simple. 
Thus $\F(s) \ne \pm 1$ which completes the proof. 
\end{proof}

\begin{Lemma}\label{lem:simple2simple}
Suppose $L,L' \in \gmod(R)^\heartsuit$ are simple with at least one real, 
and $\scrM = \F(L), \scrM' = \F(L') \in \cKP$ both simple with at least one real. 
Then either $\F(L \n L')=0$ or $\F(L \n L') \cong \scrM \n \scrM'$. 
\end{Lemma}
\begin{proof}
Consider the sequence $L \circ L' \to L \n L' \to L' \circ L \{\Lambda(L,L')\}$ where the first map is surjective and the second injective. 
Applying $\F$ we get
$$\scrM * \scrM' \to \F(L \n L') \to \scrM' * \scrM \{\Lambda(L,L')\}$$
where, since $\F$ is t-exact, the first map is surjective and the second injective. It follows that either $\F(L \n L') = 0$ or the composition 
above is nonzero. Since $\scrM,\scrM'$ are simple with at least one real this composition is a nonzero scalar multiple of $\r_{\scrM,\scrM'}$ 
by Proposition \ref{prop:homs1}. Thus $\F(L \n \L') \cong \scrM \n \scrM'$. 
\end{proof}

\begin{Proposition}\label{prop:simple2simple}
The monoidal functor $\F$ takes a simple module to either zero or a simple object. 
\end{Proposition}

\begin{proof}
Every simple in $\gmod(R)^\heartsuit$ is of the form 
$$L(\pi) = \hd(L_1 \circ \dots \circ L_r)$$
where, following the notation in \S \ref{sec:quiverHecke}, each $L_i$ is either of the form $L(\psi_j^{m_j})$ or $L(\mu)$ for some $\ell$-multipartition $\mu$ and where all but at most one $L_i$ is real.  We can assume $\F(L_i) \ne 0$ since otherwise $\F(L(\pi))=0$ and we are done. We will prove that $\F(L(\pi))$ is simple or zero by induction on $r$. 

If $r=1$ then we are done by Corollary \ref{cor:F2} and Lemma \ref{lem:F5}. Otherwise, at least one of the simples $L_1$ or $L_r$ is real. Suppose $L_r$ is real (the other case is the same). By induction we can assume $\F(\hd(L_1 \circ \dots \circ L_{r-1}))$ is simple (if it is zero then $\F(L(\pi))=0$ and we are done). Since $L_r$ is simple, real we get from Corollary \ref{cor:F2} that so is $\F(L_r)$. Since $\hd(L_1 \circ \dots \circ L_r)$ is simple we have
$$L(\pi) \cong \hd(L_1 \circ \dots \circ L_{r-1}) \n L_r.$$
By Lemma \ref{lem:simple2simple} either $\F(L(\pi)) = 0$ or 
$$\F(L(\pi)) \cong \F(\hd(L_1 \circ \dots \circ L_{r-1})) \n \F(L_r)$$
which is simple since $\F(L_r)$ is simple, real. 
\end{proof}

\subsection{Generation of $\cKP$}

It is known \cite{W} that the Coulomb branches of 3d $\calN=4$ quiver gauge theories is generated by dressed minuscule monopole operators.  
The same holds for the K-theoretic Coulomb branches, see, e.g.,  \cite[\S4.2.2]{VV2}.  
For our purposes we need the following weaker but categorical version of this result. 
It is weaker because, in K-theory, it does not imply the aforementioned generation.

\begin{Proposition}\label{prop:generation}
Up to loop and Koszul shifts, 
every simple in $\cKP$ occurs as a subquotient of a product of various $\cP_{\pm e_i,n}$ where $i \in I$ and $n \in \Z$.
\end{Proposition}

\begin{proof}
To streamline the exposition we will ignore loop and Koszul shifts and say that a simple object has property Q if it appears as a subquotient of a product of $\cP_{\pm e_i,n}$. Note that if $\cP_1$ and $\cP_2$ have property Q then any simple subquotient of $\cP_1 * \cP_2$ has property Q. Recall from \S \ref{sec:simples} that every simple in $\cKP$ is of the form $\cP_{\lambda^\vee,\mu}$. We will now show that every $\cP_{\lambda^\vee,\mu}$ has property Q in three steps. 

Step 1: Show that every $\cP_{\lambda^\vee,0}$ has property Q. By a standard argument there is a nonzero map
$$\cP^\Gr_{(k-1)e_i} * \cP^\Gr_{e_i,k} \to \cP^\Gr_{ke_i}$$
between sheaves on the affine Grassmannian. By Corollary \ref{cor:composition} this implies that there is also a nonzero map in $\cKP$
$$\cP_{(k-1)e_i} * \cP_{e_i,k} \to \cP_{ke_i}$$
where $\cP_{ke_i}$ is simple by Corollary \ref{cor:simple}. Applying this repeatedly shows that all $\cP_{ke_i}$ have property Q. Similarly, one can show that $\cP_{-ke_i}$ have property Q. Thus $\cP_{\lambda^\vee,0}$ has property Q if $\lambda^\vee$ is minuscule. Writing a general $\lambda^\vee$ as a sum of minuscules it suffices to show that if $\cP_{\lambda_1^\vee,0}$ and $\cP_{\lambda_2^\vee,0}$ have property Q then so does $\cP_{\lambda_1^\vee+\lambda_2^\vee,0}$.

To see this first note that it suffices to show that $j^*(\cP_{\lambda_1^\vee,0} * \cP_{\lambda_2^\vee,0})$ contains 
$$j^*(\cP_{\lambda_1^\vee+\lambda_2^\vee,0}) \cong \O_{\calR_{\lambda_1^\vee+\lambda_2^\vee}^\cl} [D]$$
as a subquotient, where $j: \calR_{\lambda_1^\vee+\lambda_2^\vee} \to \calR_{\le \lambda_1^\vee+\lambda_2^\vee}$ is the natural open immersion and 
$$D = \frac12(\dim \Gr_{\lambda_1^\vee} + \dim \Gr_{\lambda_2^\vee}) = \frac12 \dim \Gr_{\lambda_1^\vee + \lambda_2^\vee}.$$
Now consider the diagram 
\begin{equation}\label{eq:conv3}
\begin{split}
\xymatrix{
\calR_{\lambda^\vee_1}^\cl \ttimes \calR_{\lambda^\vee_2}^\cl \ar[d]^{\cl} & \ar[l]_-{d_o'} \calS_o' \ar[d]^{\cl'} & \\
\calR_{\lambda^\vee_1} \ttimes \calR_{\lambda^\vee_2} \ar[d]^{j_1 \ttimes j_2} & \calS_o \ar[d]^{j'} \ar[l]_-{d_o} \ar[r]^-{m_o} & \calR_{\lambda^\vee_1+\lambda^\vee_2} \ar[d]^j \\
\calR_{\le \lambda^\vee_1} \ttimes \calR_{\le \lambda^\vee_2} & \calS \ar[l]_-{d} \ar[r]^-{m} & \calR_{\le \lambda^\vee_1+\lambda^\vee_2}
}
\end{split}
\end{equation}
where the bottom right and top left squares are Cartesian and $j_1,j_2$ are the obvious immersions. Then 
\begin{align*}
j^*(\cP_{\lambda_1^\vee,0} * \cP_{\lambda_2^\vee,0})
&\cong j^* m_* d^* (\cP_{\lambda_1^\vee,0} \widetilde{\boxtimes} \cP_{\lambda_2^\vee,0}) \\
&\cong m_{o*} j^{\prime *} d^* (\cP_{\lambda_1^\vee,0} \widetilde{\boxtimes} \cP_{\lambda_2^\vee,0}) \\
&\cong m_{o*} d_o^* (j_1^*(\cP_{\lambda_1^\vee,0}) \widetilde{\boxtimes} j_2^*(\cP_{\lambda_2^\vee,0})) \\
&\cong m_{o*} d_o^* \cl_{*}(\O_{\calR_{\lambda^\vee_1}^\cl \ttimes \calR_{\lambda^\vee_2}^\cl}) [D] \\
&\cong m_{o*} \cl'_{*} (\O_{\calS_o'}) [D].
\end{align*}
Now $(\calS'_o)^{\cl}$ and $\calR_{\lambda_1^\vee + \lambda_2^\vee}$ are bundles over $\Gr_{\lambda_1^\vee + \lambda_2^\vee}$. 
The induced map 
$$m_o \circ \cl': (\calS'_o)^{\cl} \to \calR^\cl_{\lambda_1^\vee + \lambda_2^\vee}$$
is an injective map of bundles (it is proper so it must be injective). This explains the right map below
$$(m_o \circ \cl')_* \O_{\calS_o'} [D] \rightarrow (m_o \circ \cl')_*(\O_{(\calS_o')^\cl}) [D] \leftarrow 
\O_{\calR^\cl_{\lambda_1^\vee + \lambda_2^\vee}} [D]$$
while the left map is induced by the standard morphism $\O_{\calS'_o} \to \O_{(\calS'_o)^\cl}$. In the heart of the Koszul-perverse t-structure the left 
map is a quotient and the right is injective. It follows that $\O_{\calR^\cl_{\lambda_1^\vee + \lambda_2^\vee}} [D]$ is a subquotient of 
$(m_o \circ \cl')_* \O_{\calS_o'} [D]$ which is what we needed to show. 


Step 2: Show that every $\cP_{0,\mu}$ has property Q. First note that Corollary \ref{cor:duals} 
$$\cP_{-e_i,-m_i^-} \n \cP_{e_i} \cong \psi_i^-$$ 
so that $\psi_i^-$ has property Q, and similarly for $\psi_i^+$. 
This also implies that the left and right duals of $\cP_{e_i,n}$ have property Q. Next, 
by Lemma \ref{lem:V}, see also Proposition \ref{prop:ses4}, we have 
$${}^\L \cP_{e_i,-1} * \cP_{e_i} \cong \cV_i = \cP_{0,\omega_{i,1}}$$ 
which means $\cV_i$ has property Q. Repeating this argument with duals also shows that ${}^\L \cV_i = {}^\R \cV_i$ has property Q. 
But any $\cP_{0,\mu}$ occurs as a summand of products of $\cV_i$ and their duals, the same way every representation of $GL_n$ 
appears as a direct summand of products of its standard representation and its dual. It follows that $\cP_{0,\mu}$ also has proprety Q.

Step 3: Show that every $\cP_{\lambda^\vee,\mu}$ has property Q. Recall that $\calR^\cl_{\lambda^\vee}$ is a bundle over $\Gr_{\lambda^\vee}$ 
and that the restriction of $\cP_{\lambda^\vee,\mu}$ to $\calR_{\lambda^\vee}$ is the pullback of some bundle $\cV_\mu$ on $\Gr_{\lambda^\vee}$ 
to $\calR^\cl_{\lambda^\vee} \subset \calR_{\lambda^\vee}$ with some $[-]$ shift. Moreover, one has an affine map $\Gr_{\lambda^\vee} \to G/P$ 
for some parabolic $P \subset G$ and $\cV_\mu$ is pulled back from $G/P$ where it corresponds to an irreducible representation $V_\mu$ of $P$. 

On the other hand, the restriction of $\cP_{0,\nu} * \cP_{\lambda^\vee,0}$ to $\calR_{\lambda^\vee}$ corresponds to a 
representation $V_\nu$ of $P$ 
but here $V_\nu$ is an irreducible representation of $G$ that has been restricted to $P$. Every irreducible representation $V_\mu$ 
of $P$ occurs as a subquotient of the restriction of some representation $V_\nu$ of $G$. If follows that $\cP_{\lambda^\vee,\mu}$ must occur as a 
subquotient of $\cP_{0,\nu} * \cP_{\lambda^\vee,0}$ and thus must also have property Q. 
\end{proof}

\subsection{Essential surjectivity for simples}

To simplify the exposition we say that $\scrM$ is almost in the image of $\F$ if $\scrM * \prod_{i \in I} \phi_i^{N_i} [s] \la -s \ra$ is in the 
essential image of $\F$ for some integers $N_i$ and $s$. In other words, if $\scrM$ is in the essential image of $\F$ up to multipication by 
invertible sheaves $\phi_i$ and Koszul shifts.

\begin{Lemma}\label{lem:surjective}
Suppose $L \in \gmod(R)^\heartsuit$ and $\scrS$ is a simple subquotient of $\F(L)$. Then there exists a simple subquotient $L'$ of $L$ such 
that $\F(L') \cong \scrS$.  
\end{Lemma}
\begin{proof}
We proceed by induction on the length of $\F(L)$. If the length is one then we are done since by Proposition \ref{prop:simple2simple} 
the object $\F(L)$ must be simple and thus isomorphic to $\scrS$. Otherwise let $L'$ be a simple submodule of $L$. If $\F(L') \cong \scrS$ 
we are done again. Otherwise $\F(L')$ does not contain $\scrS$ as a subquotient and thus $\F(L/L')$ must contain it. The result follows by 
induction. 
\end{proof}

\begin{Lemma}\label{lem:uptophi}
For $i \in I$ and $n \in \Z$ every $\cP_{\pm e_i, n}$ is almost in the image of $\F$.
\end{Lemma}
\begin{proof}
By Proposition \ref{prop:delta5} and Lemma \ref{lem:F5}
 every $\cP_{e_i,n}$ and $\Lambda^n(\cV_i)$ is in the essential image of $\F$ for $i \in I$ and $n \ge 0$. In particular, since 
$$\Lambda^{a_i-1}(\cV_i) * \phi_i^{-1} \ddcong (i_\u0)_* \Big(\O_{\calR_\u0^\cl} \otimes (L^{(i)}_0/tL^{(i)}_0)^\vee \Big) \ddcong {}^\L \cV_i$$
it follows that ${}^\L \cV_i$ is almost in the image of $\F$. 

On the other hand, by Lemma \ref{lem:V} we have that 
$$\cP_{e_i,\ell-1} \cong \cP_{e_i,\ell} \n ({}^\L \cV_i) \{-1\}$$
for any $\ell \in \Z$. It follows from Lemma \ref{lem:surjective} that if $\cP_{e_i,\ell}$ is almost in the image of $\F$ then so is $\cP_{e_i,\ell-1}$. This proves that $\cP_{e_i,n}$ is almost in the image of $\F$ for all $n \in \Z$. The claim for $\cP_{-e_i,n}$ is similar. 
\end{proof}

\begin{Corollary}\label{cor:surjective}
Up to products by $\phi_i$ and Koszul shifts every simple in $\cKP_{X^{(\bullet)}}$ is the image under $\F$ of a simple in $\gmod(R)^\heartsuit$. 
\end{Corollary}
\begin{proof}
It suffices to show that every simple in $\cKP_{X^{(\bullet)}}$ is almost in the image of $\F$. By Lemma \ref{lem:uptophi} every $\cP_{\pm e_i,n}$ is almost in the image of $\F$. The result follows by combining Proposition \ref{prop:generation} and Lemma \ref{lem:surjective}. 
\end{proof}

\appendix

\section{Monoidal categories with renormalized $r$-matrices}\label{app:A}

In this section we review the notation, theory and implications of renormalized $r$-matrices in the sense of \cite{KKKO1}, \cite{KK}, \cite{KP}, \cite{KKKO2} and \cite{KKOP3}. We work with a monoidal graded $\bbC$-linear Abelian category of finite length $(\calC,*)$ which is equipped with 
renormalized $r$-matrices. Although the terminology and subsequent results below apply to any such category the reader should keep in mind the 
category $\cKP$ introduced in \S \ref{sec:KP}.

We denote by $\{-\}$ the grading shift functor. Let $M \nabla N$ and $M \Delta N$ denote the head and socle of $M*N$ respectively. A simple object $M$ is called real if $M*M$ is simple. 

A system of renormalized $r$-matrices in $\calC$ is the datum for any nonzero objects $M,$ $N$ of nonzero maps 
\begin{align}\label{rMN}\r_{M,N}: M * N \to N * M \{\Lambda(M,N)\}\end{align}
where  $\Lambda(M,N)$ is some integer. We set
\begin{align}\label{b-Lambda}\b(M,N)=\frac{1}{2}\Big(\Lambda(M,N)+\Lambda(N,M)\Big) \in \bbZ.\end{align}
The maps $\bfr_{M,N}$ above are required to satisfy some natural properties as in \cite[\S 4.1]{CW1}. 
One important and more elaborate property is the following.

\begin{Lemma}[{\cite[Prop.~3.2.8]{KKKO2}}]\label{lem:lambda}
For any objects $M,N_1,N_2$ and morphism $f: N_1 \to N_2$, we consider the following diagram 
where we omit the $\{-\}$ shifts
\begin{equation}\label{eq:lasquare} 
\begin{split}
\xymatrix{
M*N_1 \ar[r]^{\bfr_{M,N_1}} \ar[d]_{\id_M * f} & N_1 * M \ar[d]^{f*\id_M} \\
M * N_2 \ar[r]^{\bfr_{M,N_2}} & N_2 * M
}
\end{split}
\end{equation}
then 
\begin{enumerate}[label=$\mathrm{(\alph*)}$,leftmargin=8mm]
\item if $\La(M,N_1) = \La(M,N_2)$ the diagram commutes,
\item if $\La(M,N_1) < \La(M,N_2)$ the bottom left composition is zero,
\item if $\La(M,N_1) > \La(M,N_2)$ the top right composition is zero.
\end{enumerate}
The corresponding statements also hold when the product with $M$ is taken on the other side. Further
\begin{enumerate}[label=$\mathrm{(\alph*)}$,leftmargin=8mm]
\item[$\mathrm{(d)}$] if $f$ is injective then $\La(M,N_1) \le \La(M,N_2)$ and $\La(N_1,M) \le \La(N_2,M)$,
\item[$\mathrm{(e)}$] if $f$ is surjective then $\La(M,N_1) \ge \La(M,N_2)$ and $\La(N_1,M) \ge \La(N_2,M)$.
\end{enumerate}
\end{Lemma}

\begin{Proposition}[{\cite[Lem.~3.2.3, Prop.~3.2.5]{KKKO2}}]\label{prop:qcom}
Let $M$, $N$ be simple objects with at least one real.
The following conditions are equivalent
\begin{enumerate}[label=$\mathrm{(\alph*)}$,leftmargin=8mm]
\item
$\b(M,N)=0$,
\item
$M*N$ is simple,
\item
$M*N$ and $N*M$ are isomorphic up to a grading shift,
\item
$M\nabla N$ and $N\Delta M$ are isomorphic up to a grading shift,
\item
$\bfr_{M,N}$ and $\r_{N,M}$ are inverse to each other up to a constant multiple.
\end{enumerate}
Further, if $M$, $N$ are both real then $M* N$ is also real.
\end{Proposition}

We say that the simple objects $M$ and $N$ $q$-commute if one of the equivalent conditions above holds.

\begin{Proposition}[{\cite[Lem.~3.2.18]{KKKO2}}]\label{prop:bselfdual}
Let $M$, $N$ be simple objects with at least one real. If there are exact sequences
\begin{align*}
0 \to X \to M*N \to Y\{m\} \to 0   \\
0 \to Y \to N*M \to X\{n\} \to 0   
\end{align*}
with $m,n\in\bbZ$ then $\La(M,N)=m$, $\La(N,M)=n$ and $\b(M,N)=m+n$.
\end{Proposition}

\begin{Proposition}[{\cite[Thm.~4.1.1, Cor.~4.1.2]{KKKO2}}]\label{prop:lsubquotient}
Let $M$, $N$ be simple objects with at least one real and $\b(M,N)\neq 0$.
Let $L$ be any simple subquotient of $M*N$.
\hfill
\begin{enumerate}[label=$\mathrm{(\alph*)}$,leftmargin=8mm]
\item
$M \n N$ and $M \d N$ are simple. 
In the Grothendieck group, we have
$$[M* N]=[M\n N]+[M\Delta N]+\sum_k[L_k]$$
with simple objects $L_k$ not isomorphic to $M\n N$, $M\Delta N$ and their shifts. 
\item
Suppose $M$ is real. Then
$\Lambda(M,L) < \Lambda(M,N)$ unless $L=M\nabla N$, and
$\Lambda(L,M) < \Lambda(N,M)$ unless $L=M\Delta N$.
Further $\Lambda(M,M\nabla N)=\Lambda(M,N)$ 
and $\Lambda(M\Delta N,M) = \Lambda(N,M)$.
\item
Suppose $N$ is real. Then
$\Lambda(N,L) < \Lambda(N,M)$ unless $L=M\Delta N$,
and $\Lambda(L,N) < \Lambda(M,N)$ unless $L=M\nabla N$.
Further $\Lambda(N,M\Delta N) = \Lambda(N,M)$
and $\Lambda(M\nabla N,N) = \Lambda(M,N)$. 
\qed
\end{enumerate}
\end{Proposition}

\begin{Proposition}[{\cite[Lem.~2.28, 2.29]{KKOP3}}]\label{prop:Lb=1}
Let $M$, $N$ be simple with at least one real and $\b(M,N)=1$. 
\hfill
\begin{enumerate}[label=$\mathrm{(\alph*)}$,leftmargin=8mm]
\item There is an exact sequence
$0 \to M \d N \to M*N \to M \n N \to 0$.
\item If $M$, $N$ are both real then $M \d N$ and $M \n N$ are also real, simple and $q$-commute with $M$ and $N$. Moreover, 
\begin{align*}
\Lambda(M, M \n N) &= \Lambda(M,N)=-\Lambda(M \n N,M) \\
\Lambda(M \n N, N) &= \Lambda(M,N)=-\Lambda(N, M \n N).
\end{align*}
\end{enumerate}
\end{Proposition}
\begin{proof}
Part (a) follows from \cite[Prop.~3.2.17]{KKKO2}. Part (b) follows from Proposition \ref{prop:lsubquotient} and \cite[Lemma 2.28]{KKOP3}. 
\end{proof}

\begin{Proposition}[{\cite[Lem.~3.1.5, Prop.~3.2.11]{KKKO2}}]\label{prop:Lambda*}
Let $L,$ $M$, $N$ be simple objects. \hfill
\begin{enumerate}[label=$\mathrm{(\alph*)}$,leftmargin=8mm]
\item
We have 
\begin{align*}
\Lambda(L,M *N) &= \Lambda(L, M)+\Lambda(L,N), \\
 \Lambda(M * N,L) &= \Lambda(M, L)+\Lambda(N, L),\\
 \b(L,M *N) &= \b(L, M)+\b(L,N).
\end{align*}
\item
If $L$ is real and $q$-commutes with $M$, then we have 
\begin{align*}
\Lambda(L,M \n N) &= \Lambda(L, M*N), \\
 \Lambda(N \n M,L) &= \Lambda(N*M, L).
 \end{align*}
 \qed
\end{enumerate}
\end{Proposition}

\begin{Proposition}[{\cite[Lem.~2.7]{KK}}]\label{prop:normal}
If $L,$ $M,$ $N$ are real, simple objects then
$$\Lambda(L,M \n N) = \Lambda(L,M) + \Lambda(L,N) \Rightarrow L \n (M \n N) \cong (L \n M) \n N.$$
\qed
\end{Proposition}

\begin{Proposition}\label{prop:homs1}
Suppose $M,$ $N$ are simple objects with at least one real.
Then, we have 
\hfill
\begin{enumerate}[label=$\mathrm{(\alph*)}$,leftmargin=8mm]
\item $\Im(\r_{M,N})$ is simple and isomorphic to $M \n N$ and $(N\Delta M)\{\Lambda(M,N\}$,
\item $\Hom(M * N, M * N) \cong \C \cdot \id_{M*N}$ and $\Hom(M * N, N * M) \cong \C \cdot \r_{M,N}$.
\end{enumerate}
\end{Proposition}

\begin{proof}
See  \cite[Prop.~2.10]{KP} and the references there.
\end{proof}

An affinization of an object $M\in\calC$ is an object $M_z\in\calC$ satisfying the conditions in \S\ref{sec:ren}.
If an affinization exists, we assume also that there is an homogeneous homomorphism
\begin{align*}\bfR_{M_z,N}: M_z * N \to N * M_z \end{align*}
such that $\r_{M,N}$ is recovered from $\bfR_{M_z,N}$ as in \eqref{KrMN}, i.e., we have
\begin{align*}
\r_{M,N} = \bfR_{M_z,N}^\ren|_{z=0}: M * N\to N * M \{\Lambda(M,N)\}
\end{align*}

\begin{Proposition}\label{prop:homs2}
Let $M,$ $N$ be simple objects with $M$ real and $M_z$ an affinization of $M$.
\hfill
\begin{enumerate}[label=$\mathrm{(\alph*)}$,leftmargin=8mm]
\item $\Hom(M_z * N, M_z * N) \cong \C[z] \cdot \id_{M_z*N}$ and $\Hom(N*M_z, N * M_z) \cong \C[z] \cdot \id_{N*M_z}$,
\item $\Hom(M_z * N, N*M_z) \cong \C[z] \cdot \bfR^{\ren}_{M_z,N}$ and $\Hom(N*M_z , M_z*N) \cong \C[z] \cdot \bfR^{\ren}_{N,M_z}$,
\item $\bfR^{\ren}_{N,M_z} \circ \bfR^{\ren}_{M_z,N} = z^{\b(M,N)} \cdot \id_{M_z * N}$ and
$\bfR^{\ren}_{M_z,N} \circ \bfR^{\ren}_{N,M_z} = z^{\b(N,M)} \cdot \id_{N*M_z}$ up to constant multiples.
\end{enumerate}
\end{Proposition}

\begin{proof}
Parts (a), (b) follow from \cite[Prop.~2.11, 2.12, Cor.~2.14]{KP}.
Part (c) is proved in \cite[Lem.~3.2.1]{KKKO2}.
\end{proof}

Now suppose  that $(\calC,*)$ is also rigid. Then any object $M$ has left and right duals ${}^\L M$ and ${}^\R  M$ and adjunction maps
\begin{align}\label{adjunction}
{}^\L M * M \to 1 \to M * {}^\L M ,\quad M * {}^\R  M \to 1 \to {}^\R  M * M
\end{align}
which yield the following isomorphisms
\begin{align}\label{adjunction3}
\begin{split}
\Hom(L *M,N)&\cong\Hom(L,N*{}^\L M)\cong\Hom(M,{}^\R L*N),\\
&\cong\Hom({}^\L M*L,N)\cong\Hom(L*{}^\R N,M).
\end{split}
\end{align}
In  particular, we have
\begin{align}\label{adjunction2}
\Hom(M *N,N*M)\cong\Hom({}^\L N*M,M *{}^\L N)\cong\Hom(N*{}^\R M,{}^\R M*N),
\end{align}
see, e.g., \cite[Prop.~2.10.8]{EGNO} for more details.

\begin{Proposition}\label{prop:Lduals}
Let $L,$ $M$, $N$ be simple objects
\hfill
\begin{enumerate}[label=$\mathrm{(\alph*)}$,leftmargin=8mm]
\item
If at least one of $M$, $N$ is real, then  
$\Lambda(M,N) = \Lambda({}^\L N, M) = \Lambda(N, {}^\R M).$
\item
If $L$ is real, then  
$M\cong L\n N\iff  M\n ({}^\R L)\cong N$,
and
$M\cong N\n L\iff  ({}^\L L)\n M\cong N$.
\end{enumerate}
\end{Proposition}

\begin{proof}
Claim (a) follows from \eqref{adjunction2} and Proposition \ref{prop:homs1}.
Part (b) is proved as in \cite[Cor.~3.14]{KKKO1}.
\end{proof}

Lastly, we say that $M \in \calC$ is invertible if the adjunctions maps ${}^\L M * M \to 1 \to M * {}^\L M$ are isomorphisms, see, e.g., \cite[\S 2.11]{EGNO}. In this case we ${}^\L M \cong {}^\R M$ and write $M^{-1}={}^\L M \cong {}^\R M$.

\end{document}